\documentclass[10pt,a4paper]{amsart}
\usepackage{verbatim}

\usepackage[T1]{fontenc}
\usepackage{graphicx}
\usepackage{enumerate}
\usepackage{amsmath,amsfonts,amssymb}
\usepackage{color}
\usepackage{cite}
\definecolor{citation}{rgb}{0.2,0.58,0.2} 
\definecolor{formula}{rgb}{0.1,0.2,0.6}
\definecolor{url}{rgb}{0.3,0,0.5}
\usepackage{pgf,tikz}
\usepackage{mathrsfs}
\usepackage{fancyhdr}
\usepackage{dutchcal}

\usepackage{dsfont}

\newcommand{\reqnomode}{\tagsleft@false}

\usepackage[colorlinks,pdfpagelabels,pdfstartview = FitH,bookmarksopen = true,bookmarksnumbered = true,urlcolor=url,linkcolor = formula,plainpages = false,hypertexnames = false,citecolor = citation] {hyperref}

\def\dx{\,{\rm d}x}

\def \d{\,{\rm d}}

\def\dist{\,{\rm dist}}

\allowdisplaybreaks
\makeatletter
\DeclareRobustCommand*{\bfseries}{%
  \not@math@alphabet\bfseries\mathbf
  \fontseries\bfdefault\selectfont
  \boldmath
}

\makeatother

\newlength{\defbaselineskip}
\newcommand{\mint}{\mathop{\int\hskip -1,05em -\, \!\!\!}\nolimits}

\newtheorem{theorem}{Theorem}
\newtheorem{corollary}{Corollary}[section]
\newtheorem{definition}{Definition}
\newtheorem{remark}{Remark}[section]
\newtheorem{lemma}{Lemma}[section]
\newtheorem{proposition}{Proposition}[section]
\numberwithin{equation}{section}
\allowdisplaybreaks[4]

\newcommand{\kk}{\kappa}
\newcommand{\m}{\mathcal{M}}

\def\er{\mathbb R}

\newcommand\eps\varepsilon
\def\eqn#1$$#2$${\begin{equation}\label#1#2\end{equation}}

\newcommand{\be}{\begin{equation}}
\newcommand{\ee}{\end{equation}}

\newcommand{\rr}{\varrho}

\newcommand{\const}{\operatorname{const}}
\newcommand{\snr}[1]{\lvert #1\rvert}

\newcommand{\nr}[1]{\lVert #1 \rVert}

\def\er{\mathbb R}

\newcommand{\RN}{\mathbb{R}^{N}}

\newcommand{\N}{\mathbb{N}}

\def\name[#1, #2]{#1 #2}

\title[$p(x)$-harmonic maps]{Boundary regularity for manifold constrained $p(x)$-harmonic maps}
\author[Chlebicka]{Iwona Chlebicka} \address{Iwona Chlebicka\\Faculty of Mathematics, Informatics and Mechanics, University of Warsaw\\ul. Banacha 2, 02-097 Warsaw, Poland} \email{\texttt{iskrzypczak@mimuw.edu.pl}}
\author[De Filippis]{Cristiana De Filippis}  \address{Cristiana De Filippis\\Mathematical Institute, University of Oxford\\ Andrew Wiles Building, Radcliffe Observatory Quarter, Woodstock Road, Oxford, OX26GG, Oxford, United Kingdom} \email{\texttt{Cristiana.DeFilippis@maths.ox.ac.uk}}
\author[Koch]{Lukas Koch}  \address{Lukas Koch\\Mathematical Institute, University of Oxford\\ Andrew Wiles Building, Radcliffe Observatory Quarter, Woodstock Road, Oxford, OX26GG, Oxford, United Kingdom} \email{\texttt{Lukas.Koch@maths.ox.ac.uk}}

\begin{document}

\subjclass[2010]{35J25, 35J60, 35J70\vspace{1mm}} 

\keywords{Regularity, manifold constrained problems, $p(x)$-Laplacian\vspace{1mm}}

\thanks{{\it Acknowledgements.}\ I. Chlebicka is supported by NCN grant no. 2016/23/D/ST1/01072. C. De Filippis and L. Koch are supported by the Engineering and Physical Sciences Research Council (EPSRC): CDT Grant Ref. EP/L015811/1. 
\vspace{1mm}}

\maketitle

\begin{abstract}
We prove partial and full boundary regularity for manifold constrained $p(x)$-harmonic maps.
\end{abstract}
\vspace{3mm}
{\small \tableofcontents}
\section{Introduction}
In this paper we complete the partial regularity theory for $p(x)$-harmonic maps studied in \cite{decv} providing partial and full boundary regularity for manifold constrained minima of the variable exponent energy:
\begin{flalign}\label{cvp}
g+\left(W^{1,p(\cdot)}(\Omega,\m)\cap W^{1,p(\cdot)}_{0}(\Omega,\RN)\right)\ni w\mapsto\mathcal{E}(w,\Omega):=\int_{\Omega}k(x)\snr{Dw}^{p(x)} \ \dx
\end{flalign}
for a suitable boundary datum $g\colon \bar{\Omega}\to \m$. Our main accomplishment is that there exists a relatively (to $\bar{\Omega}$) open subset $\Omega_{0}\subset \bar{\Omega}$ of full $n$-dimensional Lebesgue measure on which $u$ is locally H\"older continuous and the singular set $\Sigma_{0}:=\bar{\Omega}\setminus \Omega_{0}$ has Hausdorff dimension at the most equal to $n-\gamma_{1}$, see $\eqref{asp}_{1}$ below for more informations on this quantity. This is the content of the following theorem.
\begin{theorem}\label{t1}
Under assumptions \eqref{bdd}, \eqref{asp}, \eqref{ask} and \eqref{m}, let $u\in W^{1,p(\cdot)}(\Omega,\m)$ be a solution to the Dirichlet problem \eqref{cvp} with boundary datum $g\in W^{1,q}(\bar{\Omega},\m)$ satisfying \eqref{g}. Then there exists a relatively (to $\bar{\Omega}$) open subset $\Omega_{0}\subseteq \bar{\Omega}$ so that $u\in C^{0,1-\frac{n}{q}}_{loc}(\Omega_{0},\m)$ with $q$ as in \eqref{g} and $\mathcal{H}^{n-\gamma_{1}}(\Sigma_{0})=0$.
\end{theorem}
Moreover, after strengthening the hypotheses on the variable exponent $p(\cdot)$ and on the boundary datum $g(\cdot)$, we can prove that the singular set of solutions to problem
\begin{flalign}\label{cvp1.1}
g+\left(W^{1,p(\cdot)}(\Omega,\m)\cap W^{1,p(\cdot)}_{0}(\Omega,\RN)\right)\ni w\mapsto\mathcal{J}(w,\Omega):=\int_{\Omega}\snr{Dw}^{p(x)} \ \dx
\end{flalign}
does not intersect the boundary $\partial \Omega$. In this respect we have
\begin{theorem}\label{t2}
Under assumptions \eqref{bdd}, \eqref{aspinf} and \eqref{m}, let $u\in W^{1,p(\cdot)}(\Omega,\m)$ be a solution to the Dirichlet problem \eqref{cvp1.1} with boundary datum $g\colon \bar{\Omega}\to \m$ satisfying \eqref{g}. Then there exists a constant $\Upsilon\equiv \Upsilon(\texttt{data})\in (0,1]$ such that if 
\begin{flalign}\label{t2.1}
[g]_{0,1-\frac{n}{q};\bar{\Omega}}<\Upsilon,
\end{flalign}
then $\Sigma_{0}\Subset \Omega$ and so $u$ is $\left(1-\frac{n}{q}\right)$-H\"older continuous in a neighborhood of $\partial \Omega$.
\end{theorem}
We immediately refer to Section \ref{ma} for the complete list of assumptions in force concerning the regularity of $\partial \Omega$, the coefficients appearing in the energies displayed in \eqref{cvp}-\eqref{cvp1.1} and the topology of the manifold $\m$. The results exposed in Theorems \ref{t1}-\ref{t2} are new already in the case $p(\cdot)\equiv \const$. In fact, we recover for the $p(x)$-Laplacian the boundary regularity theory already available for $p$-harmonic maps, under weaker assumptions on the boundary datum than those considered in \cite{fu3,harlin,shouhlb}. Let us put our results into the context of the available literature. The regularity theory for vector-valued minimizers of functionals modelled upon the $p$-Laplacean integral, i.e. variational problems like
\begin{flalign}\label{pp}
&W^{1,p}_{loc}(\Omega,\RN)\ni w\mapsto \int_{\Omega}F(x,Dw) \ \dx\\
&\snr{z}^{p}\lesssim F(x,z)\lesssim (1+\snr{z}^{2})^{\frac{p}{2}}, \quad 1<p<\infty\nonumber,
\end{flalign}
started with the seminal paper \cite{uhl} and received several contributions later on, see \cite{giamar,giamodsou,giu,ha,kumi1,man} and references therein for an overview of the state of the art concerning $p$-laplacean type problems. On the other hand, the regularity theory in the case when both minimizers and competitors take values into a manifold $\mathcal M \subset \er^N$ faces additional difficulties. The cornerstones of the theory were laid down by the fundamental papers \cite{ES,evga,shouhl,shouhlb} analyzing harmonic maps, i.e., constrained minimizers of the functional in \eqref{pp} for $p=2$, see also \cite{sim}. We mention also the recent works \cite{MMS2,nava} for a fine analysis of the singular set of harmonic maps. The extension of  such basic results to the case $p\not =2$ has been done in the by now classical papers: \cite{fu1,fu2,fu3,harlin,luc}. Moreover, several results have been extended to more general functionals with $p$-growth, for instance the quasiconvex case has been treated in \cite{ho} while a purely PDE approach has been proposed in \cite{dumi}. The matter of boundary regularity for vectorial problems is rather delicate and received lots of attention in the literature, starting from \cite{jm}, which covers the case of quadratic functionals. This theory has been extended later on to variational integrals of $p$-laplacean type, see \cite{dugrkr} for the first results in this direction and \cite{be,dukrmi,ha,gr,km} for general systems with standard $p$-growth. On the other hand, we notice that energies of the type in \eqref{cvp} do not satisfy conditions as in \eqref{pp}, but rather, the more general and flexible one
\begin{flalign}\label{pq}
&W^{1,p}_{loc}(\Omega,\RN)\ni w\mapsto \int_{\Omega}F(x,Dw) \ \dx\\
&\snr{z}^{p}\lesssim F(x,z)\lesssim (1+\snr{z}^{2})^{\frac{q}{2}}\quad 1<p\le q<\infty.\nonumber
\end{flalign}
The systematic study of functionals as in \eqref{pq} started in \cite{ma1,ma2} and, subsequently, has undergone an intensive development over the last years, see for instance \cite{ba,bemi,besc,besc1,demi2,sharp,fomami,harhas,hisc,haok}. In particular, the energy in \eqref{cvp} have been introduced in the setting of Calculus of Variations and Homogenization in the seminal works \cite{zhi1, zhi2, zhi3}. Energies as in \eqref{cvp} also occur 
in the modelling of electro-rheological fluids, a class of non-newtonian fluids whose viscosity properties are influenced by the presence of external electromagnetic fields \cite{acmi}, see also \cite{dhroblemsr} for the basic properties of the $p(x)$-Laplacian. 
As for regularity, the first result in the vectorial case has been obtained in \cite{cosmin}, where it is shown that local minimizers of energy \eqref{cvp1.1} are locally $C^{1,\beta}$-regular in the unconstrained case. Subsequently, the regularity theory of functionals with variable growth has been developed in a series of interesting papers, \cite{rata,ratata,tac}, where the authors established partial regularity results for unconstrained minimizers that are on the other hand obviously related to the constrained case. Especially, in \cite{tac} is given an interesting partial regularity result and some singular set estimates for a class of functionals related to the constrained minimization problem in which minimizers are assumed to take values in a single chart. Finally, \cite{decv} is devoted to the study of partial inner regularity of manifold constrained $p(x)$-harmonic maps and to the analysis and dimension-reduction of their singular set.
\subsubsection*{Organization of the paper.} This paper is organized as follows: Section \ref{pre} contains our notation, the list of the assumptions which will rule problems \eqref{cvp}-\eqref{cvp1.1}, several by now classical tools in the framework of regularity theory and some results of geometric and topological nature on Lipschitz retractions. Finally, Sections \ref{partial}-\ref{fullb} are devoted to the proof of Theorem \ref{t1} and Theorem \ref{t2} respectively.
\section{Preliminaries}\label{pre}
In this section we display our notation, list the main assumptions in force throughout the paper and collect some useful tools for regularity theory and several well-known results in the framework of manifold-valued maps.
\subsection{Notation}
Following a usual custom, we denote by $c$ a general constant larger than one. Different occurrences from line to line will be still denoted by $c$, while special occurrences will be denoted by $c_1, c_2,  \tilde c$ or the like. Relevant
dependencies on parameters will be emphasized using parentheses, i.e., $c\equiv c(p,\nu, L)$ means that $c$ depends on $p,\nu, L$. Given any measurable subset $U\subset \mathbb{R}^{n}$, we denote by $\snr{U}$ its $n$-dimensional Lebesgue measure and with $\mathcal{H}^{k}(U)$ its $k$-dimensional Hausdorff measure, for some $k\ge 0$. For a point $x_{0}\in \mathbb{R}^{n}$ and a number $\rr>0$ we indicate with $B_{\rr}(x_{0}):=\left\{x\in \mathbb{R}^{n}\colon \snr{x-x_{0}}<\rr\right\}$ the open ball centered at $x_{0}$ and with radius $\rr$ and further, $B_{\rr}\equiv B_{\rr}(0)$. Similarly, for $x_{0}\in \mathbb{R}^{n-1}\times\{0\}$ we define the half ball centered at $x_{0}$ as: $B_{\rr}^{+}(x_{0}):=\left\{x\in B_{\rr}(x_{0})\colon x^{n}>0\right\}$. We moreover set $B_{\rr}^{+}\equiv B_{\rr}^{+}(0)$. We also name $\Gamma_{\rr}(x_{0})$ the set $\left\{x\in \mathbb{R}^{n}\colon x^{n}=0\ \mbox{and} \ \snr{x_{0}-x}<\rr\right\}$ and $\partial^{+}B_{\rr}^{+}(x_{0}):=\partial B_{\rr}^{+}(x_{0})\setminus \Gamma_{\rr}(x_{0})$. As before, $\Gamma_{\rr}\equiv \Gamma_{\rr}(0)$. With $U \subset \mathbb{R}^{n}$ being a measurable subset having finite and positive $n$-dimensional Lebesgue measure, and with $h \colon  U \to \er^{k}$, being a measurable map, we shall denote by  $$
   (h)_{U} \equiv \mint_{ U}  h(x) \ \dx  := \frac{1}{\snr{U}}\int_{U}  h(x) \ \dx
$$
its integral average. Similarly, with $\gamma \in (0,1)$ we denote the H\"older seminorm of $h$ as
$$
[h]_{0,\gamma; U} := \sup_{x,y \in U, x \not= y} \, 
\frac{|h(x)-h(y)|}{|x-y|^\gamma}.$$
It is well known that the quantity defined above is a seminorm and when $[h]_{0,\gamma;U}<\infty$, we will say that $h$ belongs to the H\"older space $C^{0,\gamma}(U,\mathbb{R}^{k})$. When clear from the context, we will omit the reference to $U$, i.e.: $[h]_{0,\gamma;U}\equiv [h]_{0,\gamma}$. Finally, given any set $\Gamma$ allowing for a trace operator, we denote by $\texttt{tr}_{\Gamma}(h)$ the trace of $h$ on $\Gamma$. 
\subsection{Main assumptions}\label{ma}
 Let us turn to the main assumptions that will characterize our problem. The set $\Omega\subset \mathbb{R}^{n}$, $n\ge 2$ is open, bounded, connected and 
 \begin{flalign}\label{bdd}
 \partial \Omega \ \mbox{is} \ C^{2}\mbox{-regular}.
 \end{flalign}
 When considering the functional in \eqref{cvp}, the exponent $p(\cdot)$ will always satisfy
 \begin{flalign}\label{asp}
\begin{cases}
\ p\in C^{0,\alpha}(\bar{\Omega}) \ \ \mbox{for some} \ \ \alpha\in (0,1]\\
\ 1<\gamma_{1}:=\inf_{x\in \bar{\Omega}}p(x)\le p(x)\le \gamma_{2}:=\sup_{x\in \bar{\Omega}}p(x)<\infty,
\end{cases}
\end{flalign}
while the coefficient $k(\cdot)$ is so that
\begin{flalign}\label{ask}
\begin{cases}
\ k\in C^{0,\nu}(\bar{\Omega}) \ \ \mbox{for some} \ \ \nu\in (0,1]\\
\ 0<\lambda\le k(x)\le \Lambda<\infty \ \ \mbox{for all} \ \ x\in \bar{\Omega}
\end{cases}
\end{flalign}
holds true. We anticipate that in the estimates contained in Section \ref{t1p}, only $\min\left\{\alpha,\nu\right\}$ will be relevant, so, for simplicity, for the proof of Theorem \ref{t1} we will assume that $\alpha=\nu$, i.e.: $p(\cdot)$, $k(\cdot)\in C^{0,\alpha}(\bar{\Omega})$. When dealing with the question of full boundary regularity, we need higher regularity for $p(\cdot)$. Precisely, we shall suppose that
\begin{flalign}\label{aspinf}
\begin{cases}
\ p\in C^{0,1}(\bar{\Omega}) \\
\ 2\le \gamma_{1}\le p(x)\le \gamma_{2}<\infty,
\end{cases}
\end{flalign}
with $\gamma_{1}$ and $\gamma_{2}$ as in $\eqref{asp}_{2}$. Given an half ball $B_{R}^{+}$ and a ball $B_{\rr}(x_{0})$ with $x_{0}\in B_{R}^{+}$ and $\rr\in (0,R-\snr{x_{0}})$, we denote
\begin{flalign}\label{p1p2}
p_{1}(x_{0},\rr):=\inf_{x \in B_{\rr}(x_{0})\cap B_{R}^{+}}p(x) \quad \mbox{and}\quad p_{2}(x_{0},\rr):=\sup_{x \in B_{\rr}(x_{0})\cap B_{R}^{+}}p(x).
\end{flalign}
Since in \eqref{p1p2} we will always consider the intersection with the same ball $B_{R}^{+}$, the reference to $R$ in the symbols $p_{1}$, $p_{2}$ is omitted. When clear from the context, in \eqref{p1p2} we shall not mention $x_{0}$ i.e.: $p_{i}(x_{0},\rr)\equiv p_{i}(\rr)$ for $i\in \{1,2\}$. With a little abuse, we will adopt the notation in \eqref{p1p2} also to denote the infimum (resp. the supremum) of $p(\cdot)$ on $B_{R}^{+}$: the context will remove any ambiguity. Notice that there is no loss of generality in assuming $\gamma_{1}<\gamma_{2}$, otherwise $p(\cdot)\equiv \const$ on $\bar{\Omega}$, and in this case the problem is very well understood, \cite{fu3,harlin,shouhlb}. Furthermore, we need to impose some topological restriction on the manifold $\m$. Precisely, we ask that
\begin{flalign}\label{m}
\begin{cases}
\ \m \ \mbox{is a compact,} \ m\mbox{-dimensional}, \ C^{3} \ \mbox{Riemannian submanifold of} \ \RN\\
\ \m \ \mbox{is} \ [\gamma_{2}]-1 \ \mbox{connected}\\
\ \partial \m=\emptyset.
\end{cases}
\end{flalign}
Here $[x]$ denotes the integer part of $x$ and the definition of $j$-connectedness is given in Section~\ref{exsec}, Definition \ref{def:j-con}. Moreover, we assume that the boundary datum satisfies:
\begin{flalign}\label{g}
g\in W^{1,q}(\bar{\Omega},\m)\quad \mbox{for some} \ \ q>\max\left\{n,\gamma_{2}\right\}.
\end{flalign}
Combining $\eqref{g}$ with Morrey embedding theorem we automatically get that
\begin{flalign}\label{reg}
g\in C^{0,1-\frac{n}{q}}(\bar{\Omega},\m).
\end{flalign}
Finally, to shorten the notation we shall collect the main parameters of the problem in the quantities
\begin{flalign*}
&\texttt{data}_{p(\cdot)}:=(n,N,\m,\lambda,\Lambda,\gamma_{1},\gamma_{2},q,[p]_{0,\alpha},\alpha);\\
&\texttt{data}:=(n,N,\m,\lambda,\Lambda,\gamma_{1},\gamma_{2},q,[k]_{0,\nu},[p]_{0,\alpha},\nu,\alpha).
\end{flalign*}
Any dependencies of the constants appearing in the forthcoming estimates from quantities depending on the characteristics of $\m$, such as, for instance, the $L^{\infty}$-norm of maps with range in $\m$ (which is clearly finite being $\m$ compact) will be simply denoted as a dependency from $\m$ in the form: $c\equiv c(\m)$.
\begin{remark}
\emph{Assumption \eqref{bdd} assures that there exists a positive constant $\hat{r}\equiv \hat{r}(n,\Omega)$ such that $B_{\rr}(x_{0})\cap \Omega$ is simply connected for all $\rr\in (0,\hat{r}]$ and any $x_{0}\in \partial \Omega$. This renders the existence of a positive constant $c\equiv c(n,\Omega)$ such that}
\begin{flalign*}
\frac{\mathcal{H}^{n-1}(B_{\rr}(x_{0})\cap \partial\Omega)}{\mathcal{H}^{n-1}(\partial B_{\rr}(x_{0})\cap \Omega)}>c\quad \mbox{for all} \ \ \rr\in (0,\hat{r}], \ x_{0}\in \partial \Omega.
\end{flalign*}
\emph{Moreover, the Ahlfors condition holds
\begin{flalign*}
\snr{B_{\rr}(x_{0})\cap \Omega}\sim \rr^{n}\quad \mbox{for all} \ \ x_{0}\in \bar{\Omega}, \ \rr\in (0,\hat{r}],
\end{flalign*}
with constants implicit in "$\sim$" depending on $n,\Omega$. We shall refer to such constants with the term "Ahlfors constants", see \cite[Section 2]{dugrkr}.}
\end{remark}

As to fully clarify the framework we are going to adopt, we need to introduce some basic terminology on the so-called Musielak-Orlicz-Sobolev spaces. Essentially, these are Sobolev spaces defined by the fact that the distributional derivatives lie in a suitable Musielak-Orlicz space, rather than in a Lebesgue space as usual. Classical Sobolev spaces are then a particular case. Such spaces and related variational problems are discussed for instance in \cite{o2,o1,dhroblemsr,zhi4}, to which we refer for more details. Here, we will consider spaces related to the variable exponent case in both unconstrained and manifold-constrained settings. 
\begin{definition} \label{N11}
Given an open set $\Omega \subset \mathbb{R}^{n},$ the Musielak-Orlicz space $L^{p(\cdot)}(\Omega,\mathbb{R}^{k})$, $k\ge 1$, with $p(\cdot)$ satisfying \eqref{asp}, is defined as
\begin{flalign*}
L^{p(\cdot)}(\Omega,\mathbb{R}^{k}):=\left \{ \ w\colon \Omega \to \mathbb{R}^{k} \ \mbox{measurable  and} \ \int_{\Omega}\snr{w}^{p(x)}\ \dx <\infty \ \right\}
\end{flalign*}
endowed with the Luxemburg norm $\|w\|_{L^{p(\cdot)}(\Omega,\mathbb{R}^{k})}=\inf\{\lambda>0:\ \int_\Omega |w/\lambda|^{p(x)}\ \dx<1\}$.
Consequently, 
\begin{flalign*}
W^{1,p(\cdot)}(\Omega,\mathbb{R}^{k}):=\left \{ \ w\in W^{1,1}(\Omega,\mathbb{R}^{k})\cap L^{p(\cdot)}(\Omega,\mathbb{R}^{k})\ \mbox{such  that}\ | Dw| \in L^{p(\cdot)}(\Omega, \mathbb{R}^{k\times n}) \ \right\}
\end{flalign*}
with the norm $\|w\|_{W^{1,p(\cdot)}(\Omega,\mathbb{R}^{k})}=\|w\|_{L^{p(\cdot)}(\Omega,\mathbb{R}^{k})}+\|\,|Dw|\,\|_{L^{p(\cdot)}(\Omega,\mathbb{R}^{k})}. $ The variant $W^{1,p(\cdot)}_{loc}(\Omega,\mathbb{R}^{k})$ is defined as in the classical case, whereas  $W^{1,p(\cdot)}_{0}(\Omega,\mathbb{R}^{k})$ is a closure of smooth and compactly supported functions in the norm  $\|\cdot\|_{W^{1,p(\cdot)}(\Omega,\mathbb{R}^{k})}$.
\end{definition}
It is well known that, under assumptions \eqref{asp}, the set of smooth maps is dense in $W^{1,p(\cdot)}(\Omega,\mathbb{R}^{k})$, see e.g. \cite{sharp,zhi4}. Following \cite{decv} we also recall the analogous definition of such spaces when mappings take values into $\m$.
\begin{definition}\label{N12}
Let $\m$ be a compact submanifold of $\mathbb{R}^{k}$, $k\ge 2$, without boundary and $\Omega \subset \mathbb{R}^{n}$ an open set. For $p(\cdot)$ satisfying $\eqref{asp}$, the Musielak-Orlicz-Sobolev space $W^{1,p(\cdot)}(\Omega,\m)$ of functions into $\m$ can be defined as
\begin{flalign*}
W^{1,p(\cdot)}(\Omega,\m):=\left\{ \ w \in W^{1,p(\cdot)}( \Omega,\mathbb{R}^{k})\colon \ w(x) \in \m \ \mbox{for a.e.} \ x \in \Omega \ \right\}.
\end{flalign*}
The local space $W^{1,p(\cdot)}_{loc}(\Omega,\m)$ consists of maps belonging to $W^{1,p(\cdot)}(B,\m)$ for all open sets $B\Subset \Omega$.
\end{definition}
Of course, when $p(\cdot)\equiv \const$, Definitions \ref{N11} and \ref{N12} reduce to the classical Sobolev spaces $W^{1,p}(\Omega,\mathbb{R}^{k})$ and $W^{1,p}(\Omega,\m)$ respectively. Since the regularity question in $\Omega$ is local in nature, we can choose coordinates $\{x^{i}\}_{i=1}^{n}$ centered at $x_{0}\in \partial \Omega$ such that locally $\Omega$ is the upper half space $\mathbb{R}^{n}\cap\{x^{n}>0\}$, therefore, to avoid unnecessary complications, from now on we will assume that $\Omega\equiv B_{1}^{+}$, see \cite{dugrkr,dukrmi,harlin,jm,km,shouhlb} for a more detailed discussion on this matter. Let us display the definition of constrained $W^{1,p(\cdot)}$-minimizer of \eqref{cvp} in $B_{1}^{+}$.
\begin{definition}\label{D1}
Let $g\in W^{1,q}(\bar{B}_{1}^{+},\m)$ be as in \eqref{g}. A map $u\in W^{1,p(\cdot)}(B_{1}^{+},\m)$ with $\emph{\texttt{tr}}_{\Gamma_{1}}(u)=g$, is a constrained minimizer of the functional in \eqref{cvp} in the Dirichlet class $\mathcal{C}_{g}^{p(\cdot)}(B_{1}^{+},\m)$ provided that:
\begin{flalign*}
x\mapsto k(x)\snr{Du(x)}^{p(x)}\in L^{1}(B_{1}^{+}), \qquad \emph{\texttt{tr}}_{\Gamma_{1}}(u)=\emph{\texttt{tr}}_{\Gamma_{1}}(g)
\end{flalign*}
and
\begin{flalign*}
\mathcal{E}(u,B_{1}^{+})\le \mathcal{E}(w,B_{1}^{+})
\end{flalign*}
for all maps $w\in W^{1,p(\cdot)}(B_{1}^{+},\m)$ so that $(u-w)\in W^{1,p(\cdot)}_{0}(B_{1}^{+},\mathbb{R}^{N})$. The conditions displayed above define class $\mathcal{C}_{g}^{p(\cdot)}(B_{1}^{+},\m)$. 
\end{definition}
To shorten the notation, for $\rr\in (0,1]$, $x_{0}\in \mathbb{R}^{n}\cap \left\{x^{n}\ge 0\right\}$, $f\in W^{1,p(\cdot)}(\bar{B}_{\rr}^{+}(x_{0}),\m)$ and a subset $\mathcal{X}\subseteq \mathbb{R}^{N}$, we also introduce the general Dirichlet class 
\begin{flalign*}
\hat{\mathcal{C}}^{p(\cdot)}_{f}(B_{\rr}^{+}(x_{0}),\mathcal{X}):=f+\left(W^{1,p(\cdot)}(B_{\rr}^{+}(x_{0}),\mathcal{X})\cap W^{1,p(\cdot)}_{0}(B_{\rr}^{+}(x_{0}),\RN)\right).
\end{flalign*}
Clearly, the previous position makes sense also when $p(\cdot)\equiv \const$.
\subsection{Well-known results} When dealing with $p$-Laplacean type problems, we shall often use the auxiliary vector fields $V_{s,t}\colon \er^{N\times n} \to  \er^{N\times n}$, defined by
\begin{flalign}\label{vpvqm}
V_{s,t}(z):= (s^{2}+|z|^{2})^{(t-2)/4}z, \qquad t\in (1,\infty)\ \ \mbox{and}\ \ s\in [0,1]
\end{flalign}
whenever $z \in \er^{N\times n}$. If $s=0$ we shall simply write $V_{s,t}\equiv V_{t}$. A useful related inequality is contained in the following
\begin{flalign}\label{Vm}
\snr{V_{s,t}(z_{1})-V_{s,t}(z_{2})}\approx (s^{2}+\snr{z_{1}}^{2}+\snr{z_{2}}^{2})^{(t-2)/4}\snr{z_{1}-z_{2}}, 
\end{flalign}
where the equivalence holds up to constants depending only on $n,k,t$. An important property which is usually related to such field is recorded in the following lemma.
\begin{lemma}\label{l6}
Let $t>-1$, $s\in [0,1]$ and $z_{1},z_{2}\in \mathbb{R}^{N\times n}$ be so that $s+\snr{z_{1}}+\snr{z}_{2}>0$. Then
\begin{flalign*}
\int_{0}^{1}\left[s^2+\snr{z_{1}+\lambda(z_{2}-z_{1})}^{2}\right]^{\frac{t}{2}} \ \d\lambda\sim (s^2+\snr{z_{1}}^{2}+\snr{z_{2}}^{2})^{\frac{t}{2}},
\end{flalign*}
with constants implicit in "$\sim$" depending only on $n,N,t$.
\end{lemma}
The next are a couple of simple inequalities which will be used several times throughout the paper. They are elementary, see e.g.: \cite{cosmin,decv,rata, tac}.
\begin{lemma}\label{L0}
The following inequalities hold true.
\begin{itemize}
\item[\emph{i.}] For any $\varepsilon_{0}>0$, there exists a constant $c\equiv c(\varepsilon_{0})$ such that for all $t\ge 0$, $l\ge m\ge 1$ there holds
$
\snr{t^{l}-t^{m}}\le c(l-m)\left(1+t^{(1+\varepsilon_{0})l}\right).
$
\item[\emph{ii.}] For $t \in (0,1]$, consider the function $g_{1}(t):=t^{\tilde{c}t^{\gamma}}$, where $\tilde{c}$ is an absolute real constant and $\gamma \in (0,1]$. Then $\lim_{t\to 0}g_{1}(t)=1$ and $\sup_{t \in (0,1]}g_{1}(t)\le c(\tilde{c},\gamma)$. Via the substitution $t\mapsto t^{-1}$, we have an analogous property for the function $[1,\infty)\ni t\mapsto g_{2}(t):=t^{\tilde{c}t^{-\gamma}}$, for $\tilde{c}$ and $\gamma$ as before. Precisely there holds that $\lim_{t\to \infty}g_{2}(t)=1$ and $\sup_{t\in [1,\infty)}g_{2}(t)\le c(\tilde{c},\gamma)$.
\end{itemize}
\end{lemma}
We conclude this section by recalling the celebrated iteration lemma, \cite{giu}.
\begin{lemma}\label{iter}
Let $h\colon [\varrho, R_{0}]\to \mathbb{R}$ be a non-negative, bounded function and $0<\theta<1$, $0\le A$, $0<\beta$. Assume that
$
h(r)\le A(d-r)^{-\beta}+\theta h(d),
$
for $\varrho \le r<d\le R_{0}$. Then
$
h(\varrho)\le cA/(R_{0}-\varrho)^{-\beta}
$ holds, 
where $c\equiv c(\theta, \beta)>0$.
\end{lemma}
\subsection{Extensions}\label{exsec}
In this section we shall borrow from \cite{decv} some useful lemmas concerning locally Lipschitz retractions. Such results were first introduced in \cite{harlin} and intensively used in the literature for dealing with possibly non-homogeneous variational problems whose structure is a priori non-compatible with any kind of monotonicity formulae, \cite{demi1,ho}. We refer to Remark~\ref{nomo} below for a quick discussion on this matter. We start with clarifying a key assumption in our paper, which is the concept of $j$-connectedness.
\begin{definition}\label{def:j-con}
Given an integer $j\ge 0$, a manifold $\m$ is said to be $j$-connected if its first $j$ homotopy groups vanish identically, that is $\pi_{0}(\m)=\pi_{1}(\m)=\cdots=\pi_{j-1}(\m)=\pi_{j}(\m)=0$.
\end{definition}
It is well-known that a compact manifold $\m\subset \RN$ without boundary admits a tubular neighborhood $\m\subset\omega\subset \RN$. Identifying $\m$ with its image in $\RN$, we say that a neighborhood $\omega$ of $\m$ has the nearest point property if for every $x\in \omega$ there is a unique point $\Pi_{\m}(x)\in \m$ such that $\dist(x,\m)=\snr{x-\Pi_{\m}(x)}$. The map $\Pi_{\m}\colon \omega\to \m$ is called the retraction onto $\m$, we shall refer to it also as "projector". Moreover, the regularity of $\m$ influences the regularity of $\Pi_{\m}$ in the following way:
\begin{flalign}\label{proreg}
\m \ \ \mbox{is $C^{k}$-regular for} \ \ k\ge 2\Rightarrow \Pi_{\m}\in C^{k-1}(\omega,\m),
\end{flalign}
see \cite{ho} for a deeper discussion on this matter. It is important to stress that manifolds endowed with the relatively simple topology described by Definition \ref{def:j-con} enjoy good properties in terms of retractions.
\begin{lemma}\label{hoplem}
Let $\mathcal{M}\subset \RN$ be a compact, $j$-connected submanifold for some integer $j \in \{0,\cdots, N-2\}$ contained in an $N$-dimensional cube $Q$. Then there exists a closed $(N-j-2)$-dimensional Lipschitz polyhedron $X\subset Q\setminus \mathcal{M}$ and a locally Lipschitz retraction $\psi\colon Q\setminus X \to \mathcal{M}$ such that for any $x \in Q\setminus X$,
$
\snr{D\psi (x)}\le c/\dist(x,X)
$ holds, 
for some positive $c\equiv c(N,j,\mathcal{M})$.
\end{lemma}
\begin{proof}
We refer to \cite[Lemma 6.1]{harlin} for the original proof, or \cite[Lemma 4.5]{ho} for a simplified version relying on some Lipschitz extensions of maps between Riemannian manifolds.
\end{proof}
The next lemma allows modifying the image of a map while keeping under control boundary values and $p(\cdot)$-energy, see also \cite[Lemma 5]{decv}. 
\begin{lemma}\label{exlem}
Let $\mathcal{M}$ be as in \eqref{m} and $U\subseteq B_{1}^{+}$ a subset with positive measure and piecewise $C^{1}$-regular boundary. If $w \in W^{1,p(\cdot)}(U,\RN)\cap L^{\infty}(U, \RN)$ is so that its image lies in $\m$ in a neighborhood of $\partial U$, then there exists $\tilde{w}\in w+\hat{\mathcal{C}}^{p(\cdot)}_{u}(U,\m)$ satisfying
$$
\int_{U}\snr{D\tilde{w}}^{p(x)} \ \dx\le c\int_{U}\snr{Dw}^{p(x)} \ \dx,
$$
where $c\equiv c(N,\mathcal{M},\gamma_{2})$.
\end{lemma}
\begin{remark}\label{nomo}
\emph{When dealing with manifold constrained minima of the $p$-Laplacean energy it is customary to recover the fundamental Caccioppoli inequality by exploiting the so-called monotonicity formula, see \cite{fu1,fu2,fu3,luc,sim,shouhl,shouhlb}. This way cannot be used in our case. Even though it is possible to show a monotonicity formula for the $p(x)$-energy, Lemma \ref{mono} below, see also \cite[Lemma 4.1]{tac} or \cite[Lemma 12]{decv}, its proof crucially requires some corollaries of Gehring Lemma, which, in turn, is implied by Caccioppoli inequality, whose proof requires the monotonicity formula. Lemma \ref{exlem} breaks this vicious circle giving the chance of deriving Caccioppoli inequality directly by minimality, as we will see in Section~\ref{basic}.}
\end{remark}
\section{Partial boundary regularity}\label{partial}
As mentioned in Section \ref{ma}, to avoid unnecessary complications, we shall take $\Omega\equiv B_{1}^{+}$. In fact, since $\partial \Omega$ is $C^{2}$-regular, given any $x_{0}\in \partial \Omega$, there exists an open neighborhood $B_{x_{0}}$ of $x_{0}$ and a change of variable $\Psi_{0}\in C^{2}(\bar{B}_{x_{0}},\mathbb{R}^{n})$ so that in the new coordinates $y^{i}:=\Psi_{0}^{i}(x)$ there holds that
\begin{flalign*}
\Psi_{0}(x_{0})=0,\qquad \Psi_{0}(\bar{B}_{x_{0}}\cap \bar{ \Omega})=\bar{B}_{1}^{+},\qquad  \Psi_{0}(\bar{B}_{x_{0}}\cap \partial \Omega)=\Gamma_{1}.
\end{flalign*}
Moreover, there exists a positive constant $c_{0}\equiv c_{0}(n,\partial \Omega)$ such that
\begin{flalign*}
0<c_{0}^{-1}\le \nr{D\Psi_{0}}_{L^{\infty}(\bar{B}_{x_{0}}\cap \bar{ \Omega})}\le c_{0}<\infty.
\end{flalign*}
We stress that, being $\partial \Omega$ compact, the constant $c_{0}$ does not depend from $x_{0}$. A straightforward computation shows that, if $u\in W^{1,p(\cdot)}(\Omega,\m)$ solves \eqref{cvp}, then the map $\tilde{u}:=u\circ \Psi_{0}^{-1}$ solves an analogous problem still satisfying \eqref{asp} and \eqref{ask}. Assumption \eqref{g} on the boundary condition is preserved as well: if $g\in W^{1,q}(\bar{\Omega},\m)$ then $\tilde{g}:=g\circ \Psi_{0}^{-1}\in W^{1,q}(\bar{B}_{1}^{+},\m)$.  We refer to \cite{dugrkr,harlin,jm} for more details on this matter. Therefore, keeping Definition \ref{D1} in mind, we shall study problem
\begin{flalign}\label{pd1}
\mathcal{C}_{g}^{p(\cdot)}(B_{1}^{+},\m)\ni w\mapsto \min \int_{B_{1}^{+}}k(x)\snr{Dw}^{p(x)} \ \dx,
\end{flalign}
with $k(\cdot)$ and $p(\cdot)$ as in \eqref{ask}-\eqref{asp} respectively and $g$ as in \eqref{g}.
\subsection{Basic regularity results}\label{basic}
We first fix a threshold radius $R_{*}\in (0,1]$ so that
\begin{flalign}\label{r*1}
0<R_{*}\le \min\left\{1,\left(\frac{\gamma_{1}^{2}}{4n[p]_{0,\alpha}}\right)^{\frac{1}{\alpha}},\left(\frac{\gamma_{1}q\left(1-\frac{n}{q}\right)}{4n[p]_{0,\alpha}}\right)^{\frac{1}{\alpha}}\right\}
\end{flalign}
and choose a $R\in (0,R_{*})$. Further restrictions on the size of $R_{*}$ will be imposed in Section \ref{t1p}. An immediate consequence of \eqref{r*1} is that, given any half-ball $B_{R}^{+}$ and all balls $B_{\rr}(x_{0})$ with $x_{0}\in B_{R}^{+}$ and $\rr\in (0,R-\snr{x_{0}})$, there holds 
\begin{flalign}\label{p1p2*}
\begin{cases}
\ p_{1}^{*}(x_{0},\rr)>p_{2}(x_{0},\rr)\\
\ \frac{np_{2}(x_{0},\rr)}{q}\le p_{1}(x_{0},\rr)
\end{cases} \ \  \mbox{for all} \ \ R\in (0,R_{*}], \ \ \rr\in (0,R-\snr{x_{0}}),
\end{flalign}
which is, on the other hand, automatic when $p_{1}(x_{0},\rr)\ge n$. Obviously, in \eqref{p1p2*} we adopted the usual terminology
\begin{flalign*}
p^{*}:=\begin{cases}
\ \frac{np}{n-p}\quad &\mbox{if} \ \ 1<p<n\\
\ \mbox{any finite number larger than} \ p\quad &\mbox{if} \ \ p\ge n.
\end{cases}
\end{flalign*}
Recall now that, if $B_{\rr}(x_{0})\Subset B_{R}^{+}$ and $w\in W^{1,p}(B_{\rr}(x_{0}),\RN)$ is such that $w\equiv 0$ on $U\subset B_{\rr}(x_{0})$ with $\snr{U}>\hat{c}\snr{B_{\rr}(x_{0})}$ for some positive, absolute $\hat{c}$, then Sobolev-Poincar\'e's inequality gives
\begin{flalign}\label{poi0}
\int_{B_{\rr}(x_{0})}\snr{w/r}^{p} \ \dx\le c\rr^{-n(p/p_{*}-1)}\left(\int_{B_{\rr}(x_{0})}\snr{Dw}^{p_{*}} \ \dx\right)^{\frac{p}{p_{*}}},
\end{flalign}
for $c\equiv c(n,N,p,\hat{c})$. Here $p_{*}:=\max\left\{1,\frac{np}{n+p}\right\}$. We consider now an intrinsic version of \cite[Theorem 2.4]{dugrkr}.
\begin{proposition}\label{p1}
Let $U\subset \mathbb{R}^{n}$ be an open, bounded domain with piecewise $C^{1}$-regular boundary and finite Ahlfors constants depending only from $n$. Let also $A\subset \bar{U}$ be a closed subset. Consider two non-negative functions $f_{1}\in L^{1}(U)$ and $f_{2}\in L^{1+\hat{\sigma}}(U)$ for some $\hat{\sigma}>0$. With $\theta\in (0,1)$, assume that there holds
\begin{flalign}\label{gehes}
\mint_{B_{\rr/2}(x_{0})\cap U}f_{1}\ \dx \le b\left\{\left(\mint_{B_{\rr}(x_{0})\cap U}f_{1}^{\theta} \ \dx\right)^{\frac{1}{\theta}}+\mint_{B_{\rr}(x_{0})\cap U}f_{2} \ \dx\right\}
\end{flalign}
for almost all $x_{0}\in U\setminus A$ with $B_{\rr}(x_{0})\cap A=\emptyset$ and a positive constant $b$. Set
\begin{flalign*}
d(x):=\frac{\snr{B_{\dist(x,A)}(x)\cap U}}{\snr{U}}\quad \mbox{and}\quad \tilde{f}_{1}(x):=d(x)f_{1}(x).
\end{flalign*}
Then there exists a positive threshold $\sigma_{g}\equiv \sigma_{g}(b,\theta,\hat{\sigma})\in (0,\hat{\sigma})$ such that
\begin{flalign*}
\left(\mint_{U}\tilde{f}_{1}^{1+\sigma} \ \dx\right)^{\frac{1}{1+\sigma}}\le c(n,\theta,b,\hat{\sigma})\left\{\left(\mint_{U}f_{1} \ \dx\right)+\left(\mint_{U}f_{2}^{1+\sigma} \ \dx\right)^{\frac{1}{1+\sigma}}\right\}
\end{flalign*}
for all $\sigma\in [0,\sigma_{g})$.
\end{proposition}
\begin{proof}
The proof is essentially the same as the one in \cite{dugrkr} with minor changes due to the fact that, in our case, \eqref{gehes} involves the whole integrand; see also \cite[Lemma 6.2]{giu}.
\end{proof}
As a consequence of Proposition \ref{p1}, we derive some higher integrability results for solutions to problem \eqref{cvp}.
\begin{lemma}\label{geh}
Under assumptions \eqref{asp}, \eqref{ask}, \eqref{m} and \eqref{g}, let $u\in W^{1,p(\cdot)}(B_{R}^{+},\m)$ be a solution of problem \eqref{pd1}. Then, for $x_{0}\in \bar{B}^{+}_{R}$, with $R\in (0,R_{*}]$, $R_{*}$ as in \eqref{r*1} and $0<\rr<R-\snr{x_{0}}$, there exists a positive threshold $\sigma_{g}\equiv \sigma_{g}(\texttt{data}_{p(\cdot)},q)\in \left(0,\frac{q}{\gamma_{2}}-1\right)$ such that for all $\sigma\in (0,\sigma_{g})$ there holds that:
\begin{flalign}\label{bougeh2}
&\left(\mint_{B_{\rr/2}(x_{0})\cap B_{R}^{+}}(1+\snr{Du}^{2})^{\frac{p(x)(1+\sigma)}{2}} \ \dx\right)^{\frac{1}{1+\sigma}}\nonumber \\
&\qquad\le c\left[\mint_{B_{\rr}(x_{0})\cap B_{R}^{+}}(1+\snr{Du}^{2})^{\frac{p(x)}{2}} \ \dx+\left(\mint_{B_{\rr}(x_{0})\cap B_{R}^{+}}\snr{Dg}^{p(x)(1+\sigma)} \ \dx\right)^{\frac{1}{1+\sigma}}\right].
\end{flalign}
for $c\equiv c(\texttt{data}_{p(\cdot)},q)$. If $B_{\rr}(x_{0})\Subset B_{R}^{+}$ then, there exists a positive threshold $\sigma_{g}'\equiv \sigma_{g}'(\texttt{data}_{p(\cdot)})>0$ so that
\begin{flalign}\label{inngeh}
\left(\mint_{B_{\rr/2}(x_{0})}(1+\snr{Du}^{2})^{\frac{p(x)(1+\sigma)}{2}} \ \dx\right)^{\frac{1}{1+\sigma}}\le c\mint_{B_{\rr}(x_{0})}(1+\snr{Du}^{2})^{\frac{p(x)}{2}} \ \dx,
\end{flalign}
for all $\sigma\in (0,\sigma_{g}')$ with $c\equiv c(\texttt{data}_{p(\cdot)})$. In particular,
\begin{flalign}\label{glob}
\snr{Du}^{p(\cdot)(1+\sigma)}\in L^{1}(B_{R}^{+})\quad \mbox{for all} \ \ \sigma \in \left[0,\min\left\{\sigma_{g},\sigma_{g}'\right\}\right).
\end{flalign}
\end{lemma}
\begin{proof}
We take $x_{0}\in \bar{B}^{+}_{R}$, $0<\rr<R-\snr{x_{0}}$ and distinguish two cases: $x_{0}^{n}\le \frac{3\rr}{4}$ and $x_{0}^{n}>\frac{3\rr}{4}$.
\subsubsection*{Case 1: $x_{0}^{n}\le \frac{3\rr}{4}$.} We fix parameters $\frac{\rr}{2}<\tau_{1}<\tau_{2}\le \rr$ and a cut-off function $\eta\in C^{1}_{c}(B_{\tau_{2}}(x_{0}))$ with the following specifics
\begin{flalign}\label{ex1}
\mathds{1}_{B_{\tau_{1}}(x_{0})}\le \eta\le \mathds{1}_{B_{\tau_{2}}(x_{0})}\qquad \mbox{and}\qquad \snr{D\eta}\le \frac{4}{\tau_{2}-\tau_{1}}.
\end{flalign}
Notice that in this case the intersection $B_{\tau_{2}}(x_{0})\cap \Gamma_{R}$ can be non-empty and the map $w:=u-\eta(u-g)$ agrees with $u$ in the sense of traces on $\partial(B_{\tau_{2}}(x_{0})\cap B_{R}^{+})$. This means that we can use Lemma \ref{exlem} to recover a map $\tilde{w}\in \hat{\mathcal{C}}_{u}^{p(\cdot)}(B_{\tau_{2}}(x_{0})\cap B_{R}^{+})$ satisfying the energy inequality \eqref{ex1} and so that
\begin{flalign*}
\int_{B_{\tau_{2}}(x_{0})\cap B^{+}_{R}}&\snr{Du}^{p(x)} \ \dx \le \lambda^{-1}\int_{B_{\tau_{2}}(x_{0})\cap B^{+}_{R}}k(x)\snr{Du}^{p(x)} \ \dx \nonumber \\
\le& \lambda^{-1}\int_{B_{\tau_{2}}(x_{0})\cap B^{+}_{R}}k(x)\snr{D\tilde{w}}^{p(x)} \ \dx\nonumber \\
\le &\frac{\Lambda}{\lambda}\int_{B_{\tau_{2}}(x_{0})\cap B^{+}_{R}}\snr{D\tilde{w}}^{p(x)} \ \dx\le c\int_{(B_{\tau_{2}}(x_{0})\setminus B_{\tau_{1}}(x_{0}))\cap B^{+}_{R}}\snr{Du}^{p(x)} \ \dx\nonumber \\
&+c\int_{B_{\tau_{2}}(x_{0})\cap B^{+}_{R}}\left[\snr{Dg}^{p(x)}+\left| \frac{u-g}{\tau_{2}-\tau_{1}} \right|^{p(x)}\right] \ \dx, 
\end{flalign*}
with $c\equiv c(N,\lambda,\Lambda,\gamma_{2},\m)$. Once the inequality on the previous display is available, we can use Widmann's hole filling technique, Lemma \ref{iter} and Lemma \ref{L0} (\emph{ii}) to end up with
\begin{flalign}\label{cacc}
\int_{B_{\rr/2}(x_{0})\cap B^{+}_{R}}&\snr{Du}^{p(x)} \ \dx \le c\int_{B_{\rr}(x_{0})\cap B_{R}^{+}}\snr{Dg}^{p(x)} \ \dx+c\rr^{-p_{2}(\rr)}\int_{B_{\rr}(x_{0})\cap B_{R}^{+}}\snr{u-g}^{p(x)} \ \dx\nonumber \\
\le &c\int_{B_{\rr}(x_{0})\cap B_{R}^{+}}\snr{Dg}^{p(x)} \ \dx+c\int_{B_{\rr}(x_{0})\cap B_{R}^{+}}\left| \frac{u-g}{\rr} \right|^{p(x)} \ \dx\nonumber \\
\le &c\int_{B_{\rr}(x_{0})\cap B_{R}^{+}}\snr{Dg}^{p(x)} \ \dx+c\int_{B_{\rr}(x_{0})\cap B_{R}^{+}}\left| \frac{u-g}{\rr} \right|^{p_{1}(\rr)} \ \dx +c\snr{B_{\rr}(x_{0})\cap B_{R}^{+}},
\end{flalign}
where $c\equiv c(N,\lambda,\Lambda,\gamma_{2},\m)$. Now we extend $u=g$ in $B_{\rr}(x_{0})\setminus B_{R}^{+}$, notice that condition $x_{0}^{n}\le 3\rr/4$ implies that $\snr{B_{\rr}(x_{0})\setminus B_{R}^{+}}\ge c(n)\snr{B_{\rr}(x_{0})}$ and use \eqref{poi0} to bound
\begin{flalign*}
\mint_{B_{\rr}(x_{0})\cap B_{R}^{+}}&\left| \frac{u-g}{\rr} \right|^{p_{1}(\rr)} \ \dx\le c\left(\mint_{B_{\rr}(x_{0})\cap B_{R}^{+}}\snr{Du-Dg}^{(p_{1}(\rr))_{*}} \ \dx\right)^{\frac{p_{1}(\rr)}{(p_{1}(\rr))_{*}}}\nonumber \\
\le &c\left(\mint_{B_{\rr}(x_{0})\cap B_{R}^{+}}(1+\snr{Du}^{2})^{\frac{p(x)(p_{1}(\rr))_{*}}{2p_{1}(\rr)}} \ \dx\right)^{\frac{p_{1}(\rr)}{(p_{1}(\rr))_{*}}}\nonumber \\
&+c\left(\mint_{B_{\rr}(x_{0})\cap B_{R}^{+}}\snr{Dg}^{\frac{p(x)(p_{1}(\rr))_{*}}{p_{1}(\rr)}} \ \dx\right)^{\frac{p_{1}(\rr)}{(p_{1}(\rr))_{*}}},
\end{flalign*}
for $c\equiv c(n,\gamma_{1},\gamma_{2})$. Merging the content of the two previous displays we obtain
\begin{flalign}\label{b0}
\mint_{B_{\rr/2}(x_{0})\cap B^{+}_{R}}&(1+\snr{Du}^{2})^{\frac{p(x)}{2}} \ \dx\le c\mint_{B_{\rr}(x_{0})\cap B_{R}^{+}}\snr{Dg}^{p(x)} \ \dx\nonumber \\
&+c\left(\mint_{B_{\rr}(x_{0})\cap B_{R}^{+}}(1+\snr{Du}^{2})^{\frac{p(x)}{2}\cdot \frac{(p_{1}(\rr))_{*}}{p_{1}(\rr)}} \ \dx\right)^{\frac{p_{1}(\rr)}{(p_{1}(\rr))_{*}}},
\end{flalign}
where $c\equiv c(n,N,\lambda,\Lambda,\gamma_{1},\gamma_{2},\m)$.
\subsubsection*{Case 2: $x_{0}^{n}>\frac{3\rr}{4}$} In this case, we see that $B_{\frac{3\rr}{4}}\Subset B_{R}^{+}$, so, as in \cite[Lemma 9]{decv} we recover
\begin{flalign}\label{b0.1}
\mint_{B_{\rr/2}(x_{0})}&(1+\snr{Du}^{2})^{\frac{p(x)}{2}} \ \dx\le c \left(\mint_{B_{3\rr/4}(x_{0})}(1+\snr{Du}^{2})^{\frac{p(x)}{2}\cdot \frac{(p_{1}(\rr))_{*}}{p_{1}(\rr)}} \ \dx\right)^{\frac{p_{1}(\rr)}{(p_{1}(\rr))_{*}}}\nonumber \\
\le &c \left(\mint_{B_{\rr}(x_{0})}(1+\snr{Du}^{2})^{\frac{p(x)}{2}\cdot \frac{(p_{1}(\rr))_{*}}{p_{1}(\rr)}} \ \dx\right)^{\frac{p_{1}(\rr)}{(p_{1}(\rr))_{*}}},
\end{flalign}
for $c\equiv c(n,N,\lambda,\Lambda,\gamma_{1},\gamma_{2},\m)$. Once \eqref{b0}-\eqref{b0.1} are available, we can apply Proposition \ref{p1} with $U\equiv B_{\rr}(x_{0})\cap B_{R}^{+}$ and $A\equiv \partial B_{\rr}(x_{0})\cap B_{R}^{+}$ to conclude with \eqref{bougeh2}-\eqref{inngeh}.
\\\\ 
Combining \eqref{bougeh2}, \eqref{inngeh} and a standard covering argument, we obtain \eqref{glob} and the proof is complete.
\end{proof}
\begin{remark}
\emph{Since $Dg\in L^{q}(B_{1}^{+},\mathbb{R}^{N\times n})$, with H\"older inequality we can rearrange \eqref{bougeh2} as follows:}
\begin{flalign}\label{bougeh}
&\left(\mint_{B_{\rr/2}(x_{0})\cap B_{R}^{+}}(1+\snr{Du}^{2})^{\frac{p(x)(1+\sigma)}{2}} \ \dx\right)^{\frac{1}{1+\sigma}}\nonumber \\
&\qquad\le c\left[\mint_{B_{\rr}(x_{0})\cap B_{R}^{+}}(1+\snr{Du}^{2})^{\frac{p(x)}{2}} \ \dx+\left(\mint_{B_{\rr}(x_{0})\cap B_{R}^{+}}\snr{Dg}^{q} \ \dx\right)^{\frac{p_{2}(R)}{q}}\right],
\end{flalign}
\emph{for $c\equiv c(\texttt{data}_{p(\cdot)},q)$.}
\end{remark}
Let us point out a particularly helpful inequality contained in the proof of Lemma \ref{geh}.
\begin{corollary}\label{gehc}
Under assumptions \eqref{asp}, \eqref{ask}, \eqref{m} and \eqref{g}, let $u\in W^{1,p(\cdot)}(B_{1}^{+},\m)$ be a solution of problem \eqref{pd1}. Then for any half-ball $B_{R}\subset \bar{B}_{1}^{+}$ and all balls $B_{\rr}(x_{0})$ with $x_{0}\in B_{R}^{+}$, $\rr\in (0,R-\snr{x_{0}})$, $R\in (0,R_{*}]$ and $R_{*}$ as in \eqref{r*1}, there holds that
\begin{flalign}\label{caccine}
\mint_{B_{\rr/2}(x_{0})\cap B_{R}^{+}}&\snr{Du}^{p(x)} \ \dx\le c\mint_{B_{\rr}(x_{0})\cap B_{R}^{+}}\left[\left| \frac{u-g}{r}\right|^{p(x)}+\snr{Dg}^{p(x)}\right] \ \dx 
\end{flalign}
with $c\equiv c(\texttt{data}_{p(\cdot)})$. In case $B_{\rr}(x_{0})\Subset B_{1}^{+}$, the inequality
\begin{flalign}\label{caccineint}
\mint_{B_{\rr/2}(x_{0})\cap B_{R}^{+}}&\snr{Du}^{p(x)} \ \dx\le c\mint_{B_{\rr}(x_{0})\cap B_{R}^{+}}\left[\left| \frac{u-(u)_{\rr}}{r}\right|^{p(x)}\right] \ \dx, 
\end{flalign}
for $c\equiv c(\texttt{data}_{p(\cdot)})$. Moreover, the following inequalities are satisfied:
\begin{flalign}\label{ca1}
\mint_{B_{\rr/4}(x_{0})\cap B_{R}^{+}}\snr{Du}^{p(x)} \ \dx \le c\rr^{-p_{2}(\rr)}\quad \mbox{and}\quad \mint_{B_{\rr/4}(x_{0})\cap B_{R}^{+}}\snr{Du}^{p(x)(1+\sigma)} \ \dx \le c\rr^{-p_{2}(\rr)(1+\sigma)}
\end{flalign}
with $c\equiv c(n,N,\m,\gamma_{1},\gamma_{2},q,\nr{Dg}_{L^{q}(B_{1}^{+})})$ and for all $\sigma\in \left[0,\min\left\{\sigma_{g},\frac{q}{n}-1\right\}\right]$, where $\sigma_{g}$ is the same higher integrability threshold appearing in Lemma \ref{geh}.
\end{corollary}
\begin{proof}
Inequality \eqref{caccine} is the same as \eqref{cacc} in the proof of Lemma \ref{geh}, while the proof of \eqref{caccineint} is contained in \cite[Lemma 8]{decv}. To prove \eqref{ca1} we only need to notice that by $\eqref{g}_{2}$ it immediately follows that
\begin{flalign}\label{gg}
\rr^{p_{2}(\rr)}\mint_{B_{\rr}(x_{0})\cap B_{R}^{+}}&\snr{Dg}^{p(x)(1+\sigma)} \ \dx \le c\left[\rr^{p_{2}(\rr)}+\rr^{p_{2}(\rr)}\left(\mint_{B_{\rr}(x_{0})\cap B_{R}^{+}}\snr{Dg}^{p_{2}(\rr)(1+\sigma)} \ \dx\right)\right]\nonumber \\
\le &c\left[\rr^{p_{2}(\rr)}+\rr^{p_{2}(\rr)\left(1-\frac{n(1+\sigma)}{q}\right)}\left(1+\nr{Dg}_{L^{q}(B_{1}^{+})}^{2\gamma_{2}}\right)\right]\le c(n,\gamma_{2},\nr{Dg}_{L^{q}(B_{1}^{+})}). 
\end{flalign}
Using this information together with \eqref{caccine} and $\eqref{m}_{1}$, we obtain $\eqref{ca1}_{1}$. Combining \eqref{bougeh2}-\eqref{inngeh} with $\eqref{ca1}_{1}$ and \eqref{gg} we get $\eqref{ca1}_{2}$ and the proof is complete.
\end{proof}
By Proposition \ref{p1} with $A\equiv\emptyset$, we can prove a globally higher integrability result for $p$-harmonic functions, see e.g. \cite[Lemma 10]{decv} or \cite[Lemma 3.3]{dugrkr}.
\begin{lemma}\label{gehp}
Let $R\in (0,1]$, $x_{0}\in \Gamma_{R}$ and $\rr\in (0,R-\snr{x_{0}})$. Assume $\eqref{asp}_{2}$, $\eqref{ask}_{2}$ and \eqref{m}, take $p\in [\gamma_{1},\gamma_{2}]$ and $f\in W^{1,p}(\bar{B}_{\rr}^{+}(x_{0})\cap \bar{B}_{R}^{+},\m)$ so that $\snr{Df}^{p}\in L^{1+\hat{\delta}}(\bar{B}_{\rr}^{+}(x_{0})\cap \bar{B}_{R}^{+})$. If $v\in W^{1,p}(B_{\rr}^{+}(x_{0})\cap B_{R}^{+},\m)$ is a solution of the Dirichlet problem
\begin{flalign}\label{pd2.2}
\hat{\mathcal{C}}^{p}_{f}(B_{\rr}^{+}(x_{0})\cap B_{R}^{+},\m)\ni w\mapsto \min \int_{B_{\rr}^{+}(x_{0})\cap B_{R}^{+}} k(x)\snr{Dw}^{p} \ \dx,
\end{flalign}
then there exists a positive threshold $\delta_{g}\equiv \delta_{g}(n,N,\m,\gamma_{1},\gamma_{2},\lambda,\Lambda)\in (0,\hat{\delta})$ so that
\begin{flalign}\label{bougehp}
\left(\mint_{B_{\rr}^{+}(x_{0})\cap B_{R}^{+}}\snr{Dv}^{p(1+\delta)} \ \dx\right)^{\frac{1}{1+\delta}}\le c\left\{\mint_{ B_{\rr}(x_{0})\cap B_{R}^{+}}\snr{Dv}^{p} \ \dx+\left(\mint_{ B_{\rr}(x_{0})\cap B_{R}^{+}}\snr{Df}^{p(1+\delta)} \ \dx\right)^{\frac{1}{1+\delta}}\right\}
\end{flalign}
for all $\delta\in [0,\delta_{g})$. In \eqref{bougehp}, $c\equiv c(n,N,\m,\gamma_{1},\gamma_{2},\lambda,\Lambda)$.
\end{lemma}
\subsection{Proof of Theorem \ref{t1}} \label{t1p}
The proof of Theorem \ref{t1} relies on the following result. 
\begin{proposition}\label{pp1}
Under assumptions \eqref{bdd}, \eqref{asp}, \eqref{ask} and \eqref{m}, let $u\in W^{1,p(\cdot)}(B_{1}^{+},\m)$ be a solution of problem \eqref{pd1} with boundary datum $g\colon \bar{B}_{1}^{+}\to \m$ satisfying \eqref{g}. Then, there exist a threshold radius $R_{*}\equiv R_{*}(\texttt{data})\in (0,1]$ and a smallness parameter $\varepsilon\equiv \varepsilon(\texttt{data})\in (0,1]$ such that if
\begin{flalign}\label{smallq}
\left(\rr^{p_{2}(x_{0},\rr)-n}\int_{B_{\rr}(x_{0})\cap B_{R}^{+}}\snr{Du}^{p_{2}(x_{0},\rr)} \ \dx\right)^{\frac{1}{p_{2}(x_{0},\rr)}}+\left(\rr^{q-n}\int_{B_{\rr}(x_{0})\cap B_{R}^{+}}\snr{Dg}^{q} \ \dx\right)^{\frac{1}{q}}<\varepsilon,
\end{flalign}
for some $R\in (0,R_{*}]$, $x_{0}\in B_{R}^{+}$ and $\rr\in (0,R-\snr{x_{0}})$, then
\begin{flalign*}
u\in C^{0,1-\frac{n}{q}}_{loc}\left((B_{\rr}(x_{0})\cap \bar{B}_{R}^{+})\setminus \Sigma_{0}(u,B_{\rr}(x_{0})\cap \bar{B}_{R}^{+}),\m\right),
\end{flalign*}
where $\Sigma_{0}(u,B_{\rr}(x_{0})\cap \bar{B}_{R}^{+})\subset \bar{B}_{R}^{+}$ is a closed subset with $\dim_{\mathcal{H}}(\Sigma_{0}(u,B_{\rr}(x_{0})\cap \bar{B}_{R}^{+}))<n-\gamma_{1}$.
\end{proposition}
\begin{proof}
For the sake of simplicity, we split the proof into six steps.
\subsubsection*{Step 1: Setting a threshold radius.} As mentioned in Section \ref{basic}, there is no loss of generality in reducing the size of the half ball we are working on. Precisely, in addition to \eqref{r*1}, we choose a radius $R\in (0,R_{*}]$, where now it is
\begin{flalign}\label{r*}
0<R_{*}<\min\left\{1,\left[\frac{\gamma_{1}^{2}}{4n[p]_{0,\alpha}}\right]^{\frac{1}{\alpha}},\left(\frac{\gamma_{1}q\left(1-\frac{n}{q}\right)}{4n[p]_{0,\alpha}}\right)^{\frac{1}{\alpha}},\left(\frac{\sigma_{0}\gamma_{1}}{2[p]_{0,\alpha}(2+\sigma_{0})}\right)^{\frac{1}{\alpha}}\right\},
\end{flalign}
for $\sigma_{0}\in (0,1)$ defined as
\begin{flalign}\label{t10}
\sigma_{0}:=\min\left\{\frac{1}{4},\frac{\sigma_{g}'}{2},\frac{\sigma_{g}}{2},\frac{2}{\gamma_{2}-1},\frac{\alpha}{\gamma_{2}},\frac{q-\gamma_{2}}{\gamma_{2}}\right\}.
\end{flalign}
In \eqref{t10}, $\sigma_{g}$ and $\sigma_{g}'$ are the higher integrability thresholds appearing Lemma \ref{geh}, therefore, given an half-ball $B_{R}^{+}\subset B_{1}^{+}$, by \eqref{glob} there holds that
\begin{flalign}\label{t11}
\snr{Du}^{p(\cdot)(1+\sigma)}\in L^{1}(B_{R}^{+})\quad \mbox{for all} \ \ \sigma\in [0,\sigma_{0}].
\end{flalign}
Moreover, in addition to \eqref{p1p2*}, another straightforward consequence of the restriction imposed in \eqref{r*} yields that
\begin{flalign}\label{t12}
p_{2}(x_{0},\rr)<p_{2}(x_{0},\rr)\left(1+\frac{\sigma_{0}}{2}\right)\le (1+\sigma_{0})p_{1}(x_{0},\rr),
\end{flalign}
whenever $x_{0}\in B_{R}^{+}$ and $\rr\in (0,R-\snr{x_{0}})$. Hence, combining \eqref{t11} and \eqref{t12} we can conclude that
\begin{flalign}\label{t13}
\snr{Du}^{p_{2}(x_{0},\rr)}\in L^{1}(B_{\rr}(x_{0})\cap B_{R}^{+}).
\end{flalign}
Let us stress that by continuity, for any point $\bar{x}\in \bar{B}_{R}^{+}$ for which $p(\bar{x})\ge n$, we can find a small ball $B_{\rr_{\bar{x}}}(\bar{x})\subset \bar{B}_{R}^{+}$ so that $p(x)>n-\frac{\sigma_{0}}{2}$ for all $x\in B_{\rr_{\bar{x}}}(\bar{x})$. Combining this information with~\eqref{inngeh} and observing that, by \eqref{t10} $$\left(n-\frac{\sigma_{0}}{2}\right)(1+\sigma_{0})>n+\frac{\sigma_{0}}{4},$$ with Sobolev-Morrey embedding theorem we obtain that $u\in C^{0,\frac{\sigma_{0}}{4n+\sigma_{0}}}(B_{\rr_{\bar{x}}/2}(\bar{x})\cap B_{1}^{+},\m)$. Therefore, for the rest of the paper, we shall assume that $\gamma_{2}<n$. Moreover, since from now on we work on sets of the type $B_{\rr}(x_{0})\cap B_{R}^{+}$ with $x_{0}\in B_{R}^{+}$ and $\rr\in (0,R-\snr{x_{0}})$, we shall simplify the notation in \eqref{p1p2} as follows: $p_{1}(x_{0},\rr)\equiv p_{1}(\rr)$ and $p_{2}(x_{0};\rr)\equiv p_{2}(\rr)$.
\subsubsection*{Step 2: Comparison, first time.} Let $u\in W^{1,p(\cdot)}(B_{1}^{+},\m)$ be a solution to the minimization problem \eqref{pd1} with \eqref{g} in force. We introduce the extensions
\begin{flalign}\label{tiu}
\tilde{u}(x):=\begin{cases}
\ u(x',x^{n})-g(x',x^{n})\quad &\mbox{if} \ \ x^{n}\ge 0\\
\ -\left(u(x',-x^{n})-g(x',-x^{n})\right)\quad &\mbox{if} \ \ x^{n}<0.
\end{cases}
\end{flalign}
Since $\texttt{tr}_{\Gamma_{1}}(u)=\texttt{tr}_{\Gamma_{1}}(g)$, it easily follows that $\tilde{u}\in W^{1,p(\cdot)}(B_{1},\RN)$ and, by \eqref{t11}, for all $B_{\rr}(x_{0})\subseteq B_{R}\subset B_{R_{*}}$ with $R_{*}$ as in \eqref{r*} there holds that
\begin{flalign}\label{eneq}
&\int_{B_{\rr}(x_{0})}\snr{D\tilde{u}}^{p_{2}(R)} \ \dx \le c\int_{B_{\rr}(x_{0})\cap B_{R}^{+}}\left[\snr{Du}^{p_{2}(R)}+\snr{Dg}^{p_{2}(R)}\right] \ \dx
\end{flalign}
with $c\equiv c(\gamma_{1},\gamma_{2})$. Before going on we define the following quantities:
\begin{flalign*}
&\phi(x_{0},\rr,p):=\left(\rr^{p}\mint_{B_{\rr}(x_{0})}(1+\snr{D\tilde{u}}^{2})^{p/2} \ \dx\right)^{\frac{1}{p}};\\
&\phi^{+}(x_{0},\rr,p):=\left(\rr^{p}\mint_{B_{\rr}(x_{0})\cap B_{R}^{+}}(1+\snr{Du}^{2})^{p/2} \ \dx\right)^{\frac{1}{p}};\\
&\psi(x_{0},\rr):=\phi(x_{0},\rr,p_{2}(\rr)),\qquad \psi^{+}(x_{0},\rr):=\phi^{+}(x_{0},\rr,p_{2}(\rr));\\
&\chi^{+}(x_{0},\rr):=\psi^{+}(x_{0},\rr)+\left(\rr^{q-n}\int_{B_{\rr}(x_{0})\cap B_{R}^{+}}\snr{Dg}^{q} \ \dx\right)^{\frac{1}{q}},\end{flalign*}
where $x_{0}$, $\rr$ and $R$ satisfy the usual relation $R\in (0,R_{*}]$, $x_{0}\in B_{R}^{+}$ and $\rr\in (0,R-\snr{x_{0}})$. In the definition of $\psi(x_{0},\rr)$, $p_{2}(\rr)$ is defined as in \eqref{p1p2}. We shall start our analysis by considering a point $x_{0}\in \Gamma_{R}$ and imposing \eqref{smallq} on $B_{\rr}(x_{0})\cap B_{R}^{+}\equiv B_{\rr}^{+}(x_{0})$, which, with the terminology introduced above reads as
\begin{flalign}\label{small1}
\chi^{+}(x_{0},\rr)<\varepsilon,
\end{flalign}
where $\varepsilon\in (0,1)$ is a small parameter whose size will be suitably reduced along the proof. Notice that, as done in the case of \eqref{eneq}, for all balls $B_{\rr}(x_{0})\subset B_{R}$, by H\"older inequality we have
\begin{flalign}\label{eneq1}
\chi^{+}(x_{0},\rr)\le& c'\left[\psi(x_{0},\rr)+\left(\rr^{p_{2}(\rr)-n}\int_{B_{\rr}(x_{0})\cap B_{R}^{+}}\snr{Dg}^{p_{2}(\rr)} \ \dx\right)^{\frac{1}{p_{2}(\rr)}}\right]\nonumber \\
&+c'\left(\rr^{q-n}\int_{B_{\rr}(x_{0})\cap B_{R}^{+}}\snr{Dg}^{q} \ \dx\right)^{\frac{1}{q}}\nonumber \\
\le &c'\left[\psi(x_{0},\rr)+\left(\rr^{q-n}\int_{B_{\rr}(x_{0})\cap B_{R}^{+}}\snr{Dg}^{q} \ \dx\right)^{\frac{1}{q}}\right],
\end{flalign}
for $c'\equiv c'(n,\gamma_{1},\gamma_{2},q)$. Now we compare $u$ to a solution $v\in W^{1,p_{2}(\rr)}(B_{\rr/2}^{+}(x_{0}),\m)$ of the Dirichlet problem
\begin{flalign}\label{pd2}
\hat{\mathcal{C}}^{p_{2}(\rr)}_{u}(B_{\rr/2}^{+}(x_{0}),\m)\ni w\mapsto \min \int_{B_{\rr/2}^{+}(x_{0})}k(x)\snr{Dw}^{p_{2}(\rr)} \ \dx.
\end{flalign}
Such a solution exists, given that by \eqref{t13}, class $\hat{\mathcal{C}}^{p_{2}(\rr)}_{u}(B_{\rr/2}^{+}(x_{0}),\m)$ is non-empty. The minimality of $v$ in class $\hat{\mathcal{C}}^{p_{2}(\rr)}_{u}(B_{\rr/2}^{+}(x_{0}),\m)$ yields that it satisfies the Euler-Lagrange equation
\begin{flalign}\label{el1}
0=\int_{B_{\rr/2}^{+}(x_{0})}k(x)p_{2}(\rr)\snr{Dv}^{p_{2}(\rr)-2}\left[Dv\cdot D\varphi-A_{v}(Dv,Dv)\varphi\right] \ \dx,
\end{flalign}
for any $\varphi \in W^{1,p_{2}(\rr)}_{0}(B_{\rr/2}^{+}(x_{0}),\RN)\cap L^{\infty}(B_{\rr/2}^{+}(x_{0}),\RN)$, where, for $y \in \m$, $A_{y}\colon T_{y}\m\times T_{y}\m\to (T_{y}\m)^{\perp}$ denotes the second fundamental form of $\m$. In particular, by tangentiality,
\begin{flalign}\label{20}
\nabla^{2}\Pi(v)(Dv,Dv)=-A_{v}(Dv,Dv)\quad \mbox{and}\quad \snr{A_{v}(Dv,Dv)}\le c_{\m}\snr{Dv}^{2},
\end{flalign}
where $c_{\m}$ depends only on the geometry of $\m$, see \cite[Appendix to Chapter 2]{sim}. Let us quantify the $L^{p_{2}(\rr)}$-distance between $Du$ and $Dv$. We first notice that, by \eqref{t13}, the map $\varphi:=u-v$ is admissible as a test in \eqref{el1}, thus exploiting the monotonicity properties of the integrand in \eqref{pd2}, \eqref{Vm} and Lemma \ref{l6} we obtain 
\begin{flalign}\label{t15}
\int_{B_{\rr/2}^{+}(x_{0})}&\snr{V_{p_{2}(\rr)}(Du)-V_{p_{2}(\rr)}(Dv)}^{2} \ \dx \nonumber \\
\le & c\int_{B_{\rr/2}^{+}(x_{0})}k(x)\left[\snr{Du}^{p_{2}(\rr)}-\snr{Dv}^{p_{2}(\rr)}\right] \ \dx+c\int_{B_{\rr/2}^{+}(x_{0})}\snr{Dv}^{p_{2}(\rr)}\snr{u-v} \ \dx
\end{flalign}
where $c\equiv c(n,N,\gamma_{1},\gamma_{2},\lambda,\Lambda,\m)$. Let us estimate the two quantities appearing on the right-hand side of \eqref{t15}. Notice that, being $v$ a solution of \eqref{pd2}, it satisfies the assumptions of Lemma \ref{gehp} with $p=p_{2}(\rr)$, $f=u$ and $\hat{\delta}=\frac{\sigma_{0}}{2}$, therefore, choosing any $\sigma'\in \left(0,\min\{\delta_{g},\hat{\delta}\}\right)$, by H\"older inequality we control: 
\begin{flalign*}
&\int_{B_{\rr/2}^{+}(x_{0})}\snr{Dv}^{p_{2}(\rr)}\snr{u-v} \ \dx\le c\rr^{n}\left(\mint_{B_{\rr/2}^{+}(x_{0})}\snr{Dv}^{p_{2}(\rr)\left(1+\sigma'\right)} \ \dx\right)^{\frac{1}{1+\sigma'}} \nonumber \\
&\qquad {\cdot} \left(\mint_{B_{\rr/2}^{+}(x_{0})}\snr{u-v}^{\frac{1+\sigma'}{\sigma'}} \ \dx\right)^{\frac{\sigma'}{1+\sigma'}}=:c(n)\rr^{n}\left[\mbox{(I)}\cdot\mbox{(II)}\right].
\end{flalign*}
By \eqref{bougeh2}, \eqref{bougehp}, \eqref{t12}, the minimality of $v$ in class $\hat{\mathcal{C}}^{p_{2}(\rr)}_{u}(B_{\rr/2}^{+}(x_{0}),\m)$, H\"older inequality, \eqref{small1} and Lemma \ref{L0} (\emph{ii.}) we bound
\begin{flalign*}
\mbox{(I)}\le &c\left(\mint_{B_{\rr/2}^{+}(x_{0})}\snr{Du}^{p_{2}(\rr)(1+\sigma')} \ \dx\right)^{\frac{1}{1+\sigma'}}\nonumber \\
\le &c\left(\mint_{B_{\rr/2}^{+}(x_{0})}\snr{Du}^{p_{1}(\rr)(1+\sigma_{0})} \ \dx\right)^{\frac{p_{2}(\rr)}{p_{1}(\rr)(1+\sigma_{0})}}\nonumber \\
\le &c\left[\mint_{B_{\rr}^{+}(x_{0})}(1+\snr{Du}^{2})^{\frac{p(x)}{2}} \ \dx+\left(\mint_{B_{\rr}^{+}(x_{0})}\snr{Dg}^{p(x)(1+\sigma_{0})} \ \dx\right)^{\frac{1}{1+\sigma_{0}}}\right]^{\frac{p_{2}(\rr)}{p_{1}(\rr)}}\nonumber \\
\le &c\left[\mint_{B_{\rr}^{+}(x_{0})}(1+\snr{Du}^{2})^{\frac{p(x)}{2}} \ \dx+\left(\mint_{B_{\rr}^{+}(x_{0})}\snr{Dg}^{q} \ \dx\right)^{\frac{p_{2}(\rr)}{q}}\right]^{\frac{p_{2}(\rr)}{p_{1}(\rr)}}\nonumber \\
\le &c\varepsilon^{\frac{p_{2}(\rr)-p_{1}(\rr)}{p_{1}(\rr)}}\left[\mint_{B_{\rr}^{+}(x_{0})}(1+\snr{Du}^{2})^{\frac{p_{2}(\rr)}{2}} \ \dx+\left(\mint_{B_{\rr}^{+}(x_{0})}\snr{Dg}^{q} \ \dx\right)^{\frac{p_{2}(\rr)}{q}}\right]
\end{flalign*}
with $c\equiv c(\texttt{data}_{p(\cdot)})$. With Poincar\'e inequality, \eqref{t10}, the minimality of $v$ in class $\hat{\mathcal{C}}_{u}^{p_{2}(\rr)}(B_{\rr/2}^{+}(x_{0}),\m)$ and \eqref{small1} we get
\begin{flalign*}
\mbox{(II)}\le &c\left(\rr^{p_{2}(\rr)}\mint_{B_{\rr/2}^{+}(x_{0})}\snr{Du-Dv}^{p_{2}(\rr)} \ \dx\right)^{\frac{\sigma'}{1+\sigma'}}\nonumber \\
\le &c\left(\rr^{p_{2}(\rr)-n}\int_{B_{\rr}^{+}(x_{0})}(1+\snr{Du}^{2})^{\frac{p_{2}(\rr)}{2}} \ \dx\right)^{\frac{\sigma'}{1+\sigma'}}\le c\varepsilon^{\frac{\gamma_{1}\sigma'}{1+\sigma'}},
\end{flalign*}
where $c\equiv c(n,\m,\gamma_{1},\gamma_{2},\lambda,\Lambda)$. Finally, by \eqref{bougeh}, \eqref{bougehp}, Lemma \ref{L0} (\emph{i.}) with $\varepsilon_{0}=\sigma'$, the minimality of $v$ in class $\hat{\mathcal{C}}_{u}^{p_{2}(\rr)}(B_{\rr/2}^{+}(x_{0}),\m)$, \eqref{t10} and \eqref{small1} we have
\begin{flalign*}
\int_{B_{\rr/2}^{+}(x_{0})}&k(x)\left[\snr{Du}^{p_{2}(\rr)}-\snr{Dv}^{p_{2}(\rr)}\right] \ \dx\le \int_{B_{\rr/2}^{+}(x_{0})}k(x)\left| \ \snr{Du}^{p_{2}(\rr)}-\snr{Du}^{p(x)} \ \right| \ \dx\nonumber \\
&+\int_{B_{\rr/2}^{+}(x_{0})}k(x)\left| \ \snr{Dv}^{p(x)}-\snr{Dv}^{p_{2}(\rr)} \ \right| \ \dx\nonumber \\
\le &c\rr^{n+\alpha}\left[\mint_{B_{\rr/2}^{+}(x_{0})}(1+\snr{Du}^{2})^{\frac{p_{2}(\rr)}{2}(1+\sigma')} \ \dx+\mint_{B_{\rr/2}^{+}(x_{0})}(1+\snr{Dv}^{2})^{\frac{p_{2}(\rr)}{2}(1+\sigma')} \ \dx\right]\nonumber \\
\le &c\rr^{n+\alpha}\left(\mint_{B_{\rr/2}^{+}(x_{0})}(1+\snr{Du}^{2})^{\frac{p_{1}(\rr)(1+\sigma_{0})}{2}} \ \dx\right)\nonumber \\
\le &c\rr^{n+\alpha}\left[\left(\mint_{B_{\rr}^{+}(x_{0})}(1+\snr{Du}^{2})^{\frac{p(x)}{2}} \ \dx\right)+\left(\mint_{B_{\rr}^{+}(x_{0})}\snr{Dg}^{q} \ \dx\right)^{\frac{p_{2}(\rr)}{q}}\right]^{1+\sigma_{0}}\nonumber \\
\le &c\rr^{n+\alpha-\sigma_{0}p_{2}(\rr)}\left[\rr^{p_{2}(\rr)-n}\int_{B_{\rr}^{+}(x_{0})}(1+\snr{Du}^{2})^{\frac{p_{2}(\rr)}{2}} \ \dx+\left(\rr^{q-n}\int_{B_{\rr}^{+}(x_{0})}\snr{Dg}^{q} \ \dx\right)^{\frac{p_{2}(\rr)}{q}}\right]^{\sigma_{0}}\nonumber \\
&\cdot\left[\mint_{B_{\rr}^{+}(x_{0})}(1+\snr{Du}^{2})^{\frac{p_{2}(\rr)}{2}} \ \dx+\left(\mint_{B_{\rr}^{+}(x_{0})}\snr{Dg}^{q} \ \dx\right)^{\frac{p_{2}(\rr)}{q}}\right]\nonumber \\
\le &c\varepsilon^{\sigma_{0}\gamma_{1}}\left[\int_{B_{\rr}^{+}(x_{0})}(1+\snr{Du}^{2})^{\frac{p_{2}(\rr)}{2}} \ \dx+\rr^{n\left(1-\frac{p_{2}(\rr)}{q}\right)}\left(\int_{B_{\rr}^{+}(x_{0})}\snr{Dg}^{q} \ \dx\right)^{\frac{p_{2}(\rr)}{q}}\right],
\end{flalign*}
where $c\equiv c(\texttt{data}_{p(\cdot)})$. Merging the content of all the previous displays we end up with
\begin{flalign}\label{t16}
&\int_{B_{\rr/2}^{+}(x_{0})}\snr{V_{p_{2}(\rr)}(Du)-V_{p_{2}(\rr)}(Dv)}^{2} \ \dx\nonumber \\
&\qquad \le c\varepsilon^{\frac{\sigma'\gamma_{1}}{1+\sigma'}}\left[\int_{B_{\rr}^{+}(x_{0})}(1+\snr{Du}^{2})^{\frac{p_{2}(\rr)}{2}} \ \dx+\rr^{n\left(1-\frac{p_{2}(\rr)}{q}\right)}\left(\int_{B_{\rr/2}^{+}(x_{0})}\snr{Dg}^{q} \ \dx\right)^{\frac{p_{2}(\rr)}{q}}\right],
\end{flalign}
where $c\equiv c(\texttt{data}_{p(\cdot)})$. If $p_{2}(\rr)\ge 2$, by \eqref{Vm} and \eqref{t16} we directly obtain that
\begin{flalign*}
\int_{B_{\rr/2}^{+}(x_{0})}&\snr{Du-Dv}^{p_{2}(\rr)} \ \dx\le \int_{B_{\rr/2}^{+}(x_{0})}\snr{V_{p_{2}(\rr)}(Du)-V_{p_{2}(\rr)}(Dv)}^{2} \ \dx\nonumber \\
\le & c\varepsilon^{\frac{\sigma'\gamma_{1}}{1+\sigma'}}\left[\int_{B_{\rr}^{+}(x_{0})}(1+\snr{Du}^{2})^{\frac{p_{2}(\rr)}{2}} \ \dx+\rr^{n\left(1-\frac{p_{2}(\rr)}{q}\right)}\left(\int_{B_{\rr}^{+}(x_{0})}\snr{Dg}^{q} \ \dx\right)^{\frac{p_{2}(\rr)}{q}}\right],
\end{flalign*}
while, when $1<p_{2}(\rr)<2$, via H\"older inequality, \eqref{t16}, the minimality of $v$ in class $\hat{\mathcal{C}}^{p_{2}(\rr)}_{u}(B_{\rr/2}^{+}(x_{0}),\m)$ and \eqref{Vm} we can conclude that
\begin{flalign*}
&\int_{B_{\rr/2}^{+}(x_{0})}\snr{Du-Dv}^{p_{2}(\rr)} \ \dx\nonumber \\
&\quad \le \left(\int_{B_{\rr/2}^{+}(x_{0})}\snr{Du-Dv}^{2}(\snr{Du}^{2}+\snr{Dv}^{2})^{\frac{p_{2}(\rr)-2}{2}} \ \dx\right)^{\frac{p_{2}(\rr)}{2}}\nonumber \\
&\quad \cdot\left(\int_{B_{\rr/2}^{+}(x_{0})}(\snr{Du}^{2}+\snr{Dv}^{2})^{\frac{p_{2}(\rr)}{2}} \ \dx\right)^{\frac{2-p_{2}(\rr)}{2}}\nonumber \\
&\quad \le c\varepsilon^{\frac{\sigma'\gamma_{1}^{2}}{2(1+\sigma')}}\left[\int_{B_{\rr}^{+}(x_{0})}(1+\snr{Du}^{2})^{\frac{p_{2}(\rr)}{2}} \ \dx+\rr^{n\left(1-\frac{p_{2}(\rr)}{q}\right)}\left(\int_{B_{\rr}^{+}(x_{0})}\snr{Dg}^{q} \ \dx\right)^{\frac{p_{2}(\rr)}{q}}\right].
\end{flalign*}
All in all, setting $\kappa:=\frac{\gamma_{1}\sigma'}{1+\sigma'}\min\left\{1,\frac{\gamma_{1}}{2}\right\}$, we get
\begin{flalign}\label{t17}
\int_{B_{\rr/2}^{+}(x_{0})}&\snr{Du-Dv}^{p_{2}(\rr)} \ \dx\le c\varepsilon^{\kappa}\int_{B_{\rr}^{+}(x_{0})}(1+\snr{Du}^{2})^{\frac{p_{2}(\rr)}{2}} \ \dx\nonumber \\
&+c\varepsilon^{\kappa}\rr^{n\left(1-\frac{p_{2}(\rr)}{q}\right)}\left(\int_{B_{\rr}^{+}(x_{0})}\snr{Dg}^{q} \ \dx\right)^{\frac{p_{2}(\rr)}{q}},
\end{flalign}
for $c\equiv c(\texttt{data}_{p(\cdot)})$. 
\subsubsection*{Step 3: Comparison, second time.} Set $k_{0}:=k(x_{0})$. We confront $v$ with the solution $h\in W^{1,p_{2}(\rr)}(B_{\rr/4}^{+}(x_{0}),\RN)$ of the Dirichlet problem
\begin{flalign}\label{pd3}
\hat{\mathcal{C}}_{v}^{p_{2}(\rr)}(B_{\rr/4}^{+}(x_{0}),\RN)\ni w\mapsto \int_{B_{\rr/4}^{+}(x_{0})}k_{0}\snr{Dw}^{p_{2}(\rr)} \ \dx.
\end{flalign}
Furthermore, $h$ solves the Euler-Lagrange equation
\begin{flalign}\label{ELh}
0=\int_{B_{\rr/4}^{+}(x_{0})}k_{0}p_{2}(\rr)\snr{Dh}^{p_{2}(\rr)}Dh\cdot D\varphi \ \dx
\end{flalign}
for all $\varphi\in W^{1,p_{2}(\rr)}_{0}(B_{\rr/4}^{+}(x_{0}),\mathbb{R}^{N})$. Notice that, by the results in \cite{leosie} there holds that
\begin{flalign}\label{t18}
\nr{h}_{L^{\infty}(B_{\rr/4}^{+}(x_{0}))}\le c(N)\nr{v}_{L^{\infty}(B_{\rr/4}^{+}(x_{0}))}\le c(N,\m).
\end{flalign}
Recalling \cite[Lemma 3.4]{dugrkr}, see also \cite[Proof of Lemma 2]{jm} there holds that
\begin{flalign}\label{t19}
\int_{B_{\varsigma}^{+}(x_{0})}&\snr{Dh}^{p_{2}(\rr)} \ \dx \le c\left(\frac{\varsigma}{\rr}\right)^{\vartheta}\int_{B_{\rr/2}^{+}(x_{0})}\snr{Du}^{p_{2}(\rr)} \ \dx\nonumber \\
&+c\varsigma^{n\left(1-\frac{p_{2}(\rr)}{q}\right)}\left(\int_{B_{\rr/2}^{+}(x_{0})}\snr{Dg}^{q} \ \dx\right)^{\frac{p_{2}(\rr)}{q}},
\end{flalign}
for all $\vartheta\in \left(n\left(1-\frac{p_{2}(\rr)}{q}\right),n\right)$ with $c\equiv c(n,N,\gamma_{1},\gamma_{2},\lambda,\Lambda,q)$. For \eqref{t19} we also used that, by \eqref{pd3} and \eqref{pd2},
$$\texttt{tr}_{\Gamma_{1}}(h)=\texttt{tr}_{\Gamma_{R/4}}(v)=\texttt{tr}_{\Gamma_{R/4}}(u)=\texttt{tr}_{\Gamma_{R/4}}(g),$$
the minimality of $h$ in class $\hat{\mathcal{C}}^{p_{2}(\rr)}_{v}(B_{\rr/4}^{+}(x_{0}),\mathbb{R}^{N})$ and the one of $v$ in class $\hat{\mathcal{C}}^{p_{2}(\rr)}_{u}(B_{\rr/2}^{+}(x_{0}),\m)$. Exploiting now the monotonicity properties of the integrand in \eqref{pd3}, Lemma \ref{l6}, \eqref{el1}, \eqref{ELh}, H\"older inequality and the minimality of $h$ in class $\hat{\mathcal{C}}_{v}^{p_{2}(\rr)}(B_{\rr/4}^{+}(x_{0}),\mathbb{R}^{N})$, we estimate
\begin{flalign*}
c\int_{B_{\rr/4}^{+}(x_{0})}&\snr{V_{p_{2}(\rr)}(Dv)-V_{p_{2}(\rr)}(Dh)}^{2} \ \dx\nonumber \\
\le& \int_{B_{\rr/4}^{+}(x_{0})}k_{0}p_{2}(\rr)\left(\snr{Dv}^{p_{2}(\rr)-2}Dv-\snr{Dh}^{p_{2}(\rr)-2}Dh\right)\cdot (Dv-Dh) \ \dx\nonumber \\
=&\int_{B_{\rr/4}^{+}(x_{0})}(k_{0}-k(x))p_{2}(\rr)\snr{Dv}^{p_{2}(\rr)-2}Dv\cdot (Dv-Dh) \ \dx\nonumber \\
&+\int_{B_{\rr/4}^{+}(x_{0})}k(x)p_{2}(\rr)\snr{Dv}^{p_{2}(\rr)-2}Dv\cdot (Dv-Dh) \ \dx \nonumber \\
\le &c\rr^{\alpha}\int_{B_{\rr/4}^{+}(x_{0})}\snr{Dv}^{p_{2}(\rr)-1}\snr{Dv-Dh} \ \dx +c\int_{B_{\rr/4}^{+}(x_{0})}\snr{Dv}^{p_{2}(\rr)}\snr{v-h} \ \dx\nonumber \\
\le &c\rr^{\alpha}\int_{B_{\rr/4}^{+}(x_{0})}\snr{Dv}^{p_{2}(\rr)} \ \dx +c\int_{B_{\rr/4}^{+}(x_{0})}\snr{Dv}^{p_{2}(\rr)}\snr{v-h} \ \dx=:c\left[\rr^{\alpha}\mbox{(I)}+\mbox{(II)}\right],
\end{flalign*}
with $c\equiv c(n,N,\lambda,\Lambda,\gamma_{1},\gamma_{2},[k]_{0,\alpha},\alpha)$. The minimality of $v$ in class $\hat{\mathcal{C}}_{u}^{p_{2}(\rr)}(B_{\rr/2}^{+}(x_{0}),\m)$ yields that
\begin{flalign*}
\mbox{(I)}\le \int_{B_{\rr}^{+}(x_{0})}\snr{Du}^{p_{2}(\rr)} \ \dx
\end{flalign*}
and, recalling also \eqref{small1}, we see that
\begin{flalign}\label{t20}
\left(\frac{\rr}{2}\right)^{p_{2}(\rr)-n}&\int_{B_{\rr/2}^{+}(x_{0})}\snr{Dv}^{p_{2}(\rr)} \ \dx \le 2^{n-\gamma_{1}}\rr^{p_{2}(\rr)-n}\int_{B_{\rr/2}^{+}(x_{0})}\snr{Du}^{p_{2}(\rr)} \ \dx\nonumber \\
\le &2^{n-\gamma_{1}}\left[\rr^{p_{2}(\rr)-n}\int_{B_{\rr}^{+}(x_{0})}(1+\snr{Du}^{2})^{\frac{p_{2}(\rr)}{2}} \ \dx+\left(\rr^{q-n}\int_{B_{\rr}^{+}(x_{0})}\snr{Dg}^{q} \ \dx\right)^{\frac{p_{2}(\rr)}{q}}\right]\nonumber \\
<&2^{n-\gamma_{1}}\varepsilon^{p_{2}(\rr)}.
\end{flalign}
By H\"older inequality, the minimality of $h$ in class $\hat{\mathcal{C}}_{v}^{p_{2}(\rr)}(B_{\rr/4}^{+}(x_{0}),\RN)$ and the one of $v$ in class $\hat{\mathcal{C}}_{u}^{p_{2}(\rr)}(B_{\rr/2}^{+}(x_{0}),\m)$, Lemma \ref{gehp}, \eqref{inngeh}, \eqref{t18} and \eqref{t20} we bound 
\begin{flalign*}
\mbox{(II)}\le &c\rr^{n}\left(\mint_{B_{\rr/4}^{+}(x_{0})}\snr{Dv}^{p_{2}(\rr)(1+\sigma')} \ \dx\right)^{\frac{1}{1+\sigma'}}\left(\mint_{B_{\rr/4}^{+}(x_{0})}\snr{v-h}^{p_{2}(\rr)} \ \dx\right)^{\frac{\sigma'}{1+\sigma'}}\nonumber \\
\le &c\rr^{n}\left[\mint_{B_{\rr/2}^{+}(x_{0})}\snr{Du}^{p_{2}(\rr)} \ \dx+\left(\mint_{B_{\rr/2}^{+}(x_{0})}\snr{Dg}^{q} \ \dx\right)^{\frac{p_{2}(\rr)}{q}}\right]\nonumber \\
&\cdot\left(\rr^{p_{2}(R)-n}\int_{B_{\rr/4}^{+}(x_{0})}\snr{Dv-Dh}^{p_{2}(\rr)} \ \dx\right)^{\frac{\sigma'}{1+\sigma'}}\nonumber \\
\le &c\varepsilon^{\frac{\gamma_{1}\sigma'}{1+\sigma'}}\left[\int_{B_{\rr}^{+}(x_{0})}(1+\snr{Du}^{2})^{\frac{p_{2}(\rr)}{2}} \ \dx+\rr^{1-\frac{p_{2}(\rr)}{q}}\left(\int_{B_{\rr}^{+}(x_{0})}\snr{Dg}^{q} \ \dx\right)^{\frac{p_{2}(\rr)}{q}}\right],
\end{flalign*}
for $c\equiv c(\texttt{data}_{p(\cdot)})$. Merging the content of the two previous displays and proceeding as in the last part of \emph{Step 2} we end up with
\begin{flalign}\label{t.21}
&\int_{B_{\rr/4}^{+}(x_{0})}\snr{Dv-Dh}^{p_{2}(\rr)} \ \dx\nonumber \\
&\quad\le c\left[\varepsilon^{\kappa}+\rr^{\alpha}\right]\left[\int_{B_{\rr}^{+}(x_{0})}(1+\snr{Du}^{2})^{\frac{p_{2}(\rr)}{2}} \ \dx+\rr^{n\left(1-\frac{p_{2}(\rr)}{q}\right)}\left(\int_{B_{\rr}^{+}(x_{0})}\snr{Dg}^{q} \ \dx\right)^{\frac{p_{2}(\rr)}{q}}\right],
\end{flalign}
with $c\equiv c(\texttt{data})$. Collecting inequalities \eqref{t17} and \eqref{t.21} we obtain
\begin{flalign}\label{t.22}
&\int_{B_{\rr/4}^{+}(x_{0})}\snr{Du-Dh}^{p_{2}(\rr)} \ \dx\nonumber \\
&\quad\le c\left[\varepsilon^{\kappa}+\rr^{\alpha}\right]\left[\int_{B_{\rr}^{+}(x_{0})}(1+\snr{Du}^{2})^{\frac{p_{2}(\rr)}{2}} \ \dx+\rr^{n\left(1-\frac{p_{2}(\rr)}{q}\right)}\left(\int_{B_{\rr}^{+}(x_{0})}\snr{Dg}^{q} \ \dx\right)^{\frac{p_{2}(\rr)}{q}}\right],
\end{flalign}
where $c\equiv c(\texttt{data})$. 
\subsubsection*{Step 4: Morrey decay estimates at the boundary.} Let $\varsigma\in \left(0,\frac{\rr}{4}\right)$ and estimate, via \eqref{eneq}, \eqref{t19} and \eqref{t.22},
\begin{flalign*}
\int_{B_{\varsigma}(x_{0})}&(1+\snr{D\tilde{u}}^{2})^{\frac{p_{2}(\rr)}{2}} \ \dx\le  c\left[\int_{B_{\varsigma}^{+}(x_{0})}\snr{Du}^{p_{2}(\rr)} \ \dx+\int_{B_{\varsigma}^{+}(x_{0})}\snr{Dg}^{p_{2}(\rr)} \ \dx\right]+c\varsigma^{n}\nonumber \\
\le &c\left[\int_{B_{\varsigma}^{+}(x_{0})}\snr{Du-Dh}^{p_{2}(\rr)} \ \dx+\int_{B_{\varsigma}^{+}(x_{0})}\snr{Dh}^{p_{2}(\rr)} \ \dx +\int_{B_{\varsigma}^{+}(x_{0})}\snr{Dg}^{p_{2}(\rr)} \ \dx\right]+c\varsigma^{n}\nonumber \\
\le &c\left[\left(\frac{\varsigma}{\rr}\right)^{n}+\varepsilon^{\kappa}+\rr^{\alpha}+\left(\frac{\varsigma}{\rr}\right)^{\vartheta}\right]\nonumber \\
&\cdot\left[\int_{B_{\rr}^{+}(x_{0})}(1+\snr{Du}^{2})^{\frac{p_{2}(\rr)}{2}} \ \dx+\rr^{n\left(1-\frac{p_{2}(\rr)}{q}\right)}\left(\int_{B_{\rr}^{+}(x_{0})}\snr{Dg}^{q} \ \dx\right)^{\frac{p_{2}(\rr)}{q}}\right]\nonumber \\
&+c\left(\frac{\varsigma}{\rr}\right)^{n\left(1-\frac{p_{2}(\rr)}{q}\right)}\left(\rr^{q-n}\int_{B_{\rr}^{+}(x_{0})}\snr{Dg}^{q} \ \dx\right)^{\frac{p_{2}(\rr)}{q}},
\end{flalign*}
where $c\equiv c(\texttt{data})$. Now recall that $\vartheta>n\left(1-\frac{p_{2}(\rr)}{q}\right)$, so we can always find $\hat{\nu}\in \left(n\left(1-\frac{p_{2}(\rr)}{q}\right),\vartheta\right)$. Moreover, set
\begin{flalign*}
\hat{p}_{2}(\rr):=n\left(1-\frac{p_{2}(\rr)}{q}\right)\qquad \mbox{and}\qquad \tilde{p}_{2}(\rr):=p_{2}(\rr)-n+\hat{\nu}
\end{flalign*}
and choose $\varsigma=\tau \rr$ for some $\tau\in \left(0,\frac{1}{4}\right)$. Multiplying both sides of the previous inequality by $(\tau \rr)^{p_{2}(\rr)-n}$ we obtain
\begin{flalign}\label{t.23}
&(\tau \rr)^{p_{2}(\rr)-n}\int_{B_{\tau \rr}(x_{0})}(1+\snr{D\tilde{u}}^{2})^{\frac{p_{2}(\rr)}{2}} \ \dx\nonumber \\
&\quad\le \tau^{\tilde{p}_{2}(\rr)}\left[c\tau^{n-\nu}+c\tau^{-\nu}\left(\varepsilon^{\kappa}+R_{*}^{\alpha}\right)+c\tau^{\vartheta-\nu}\right]\nonumber \\
&\quad\cdot\left[\rr^{p_{2}(\rr)-n}\int_{B_{\rr}^{+}(x_{0})}(1+\snr{Du}^{2})^{\frac{p_{2}(\rr)}{2}} \ \dx+\left(\rr^{q-n}\int_{B_{\rr}^{+}(x_{0})}\snr{Dg}^{q} \ \dx\right)^{\frac{p_{2}(\rr)}{q}}\right]\nonumber \\
&\quad+c\tau^{\hat{p}_{2}(\rr)}\left(\rr^{q-n}\int_{B_{\rr}^{+}(x_{0})}\snr{Dg}^{q} \ \dx\right)^{\frac{1}{q}},
\end{flalign}
for $c\equiv c(\texttt{data})$. With the notation introduced in \emph{Step 2}, the inequality in \eqref{t.23} reads as
\begin{flalign}\label{t.24}
&\phi(x_{0},\tau \rr,p_{2}(\rr))\le \tau^{\frac{\tilde{p}_{2}(\rr)}{p_{2}(\rr)}}\left[c\tau^{\frac{(n-\nu)}{p_{2}(\rr)}}+c\tau^{-\frac{\nu}{p_{2}(\rr)}}\left(\varepsilon^{\frac{\kappa}{p_{2}(\rr)}}+R_{*}^{\frac{\alpha}{p_{2}(\rr)}}\right)+c\tau^{\frac{\vartheta-\nu}{p_{2}(\rr)}}\right]\nonumber \\
&\qquad\cdot\left[\psi^{+}(x_{0},\rr)+\left(\rr^{q-n}\int_{B_{\rr}^{+}(x_{0})}\snr{Dg}^{q} \ \dx\right)^{\frac{1}{q}}\right]+c\tau^{1-\frac{n}{q}}\left(\rr^{q-n}\int_{B_{\rr}^{+}(x_{0})}\snr{Dg}^{q} \ \dx\right)^{\frac{1}{q}},
\end{flalign}
therefore since
\begin{flalign}\label{t.25}
\phi(x_{0},r,p)\le \phi(x_{0},r,q) \ \ \mbox{for} \ \ p\le q,
\end{flalign}
we obtain from \eqref{t.24}:
\begin{flalign*}
\psi(x_{0},\tau \rr)\le & \tau^{\frac{\tilde{p}_{2}(\rr)}{p_{2}(\rr)}}\left[c\tau^{\frac{(n-\nu)}{p_{2}(\rr)}}+c\tau^{-\frac{\nu}{p_{2}(\rr)}}\left(\varepsilon^{\frac{\kappa}{p_{2}(\rr)}}+R_{*}^{\frac{\alpha}{p_{2}(\rr)}}\right)+c\tau^{\frac{\vartheta-\nu}{p_{2}(\rr)}}\right]\nonumber \\
&\cdot\left[\psi^{+}(x_{0},\rr)+\left(\rr^{q-n}\int_{B_{\rr}^{+}(x_{0})}\snr{Dg}^{q} \ \dx\right)^{\frac{p_{2}(\rr)}{q}}\right]\nonumber \\
&+c\tau^{1-\frac{n}{q}}\left(\rr^{q-n}\int_{B_{\rr}^{+}(x_{0})}\snr{Dg}^{q} \ \dx\right)^{\frac{1}{q}}
\end{flalign*}
with $c\equiv c(\texttt{data})$. Recalling that $\tau\in \left(0,\frac{1}{4}\right)$, it is easy to see that
\begin{flalign*}
&\tau^{\frac{\tilde{p}_{2}(\rr)}{p_{2}(\rr)}}\left[c\tau^{\frac{(n-\nu)}{p_{2}(\rr)}}+c\tau^{-\frac{\nu}{p_{2}(\rr)}}\left(\varepsilon^{\frac{\kappa}{p_{2}(\rr)}}+R_{*}^{\frac{\alpha}{p_{2}(\rr)}}\right)+c\tau^{\frac{\vartheta-\nu}{p_{2}(\rr)}}\right]\nonumber \\
&\qquad \le \tau^{\frac{\tilde{p}_{2}(\rr)}{p_{2}(\rr)}}\left[c\tau^{\frac{n-\vartheta}{\gamma_{2}}}+c\tau^{-\frac{\vartheta}{\gamma_{1}}}\left(\varepsilon^{\frac{\kk}{\gamma_{2}}}+R_{*}^{\frac{\alpha}{\gamma_{2}}}\right)+c\tau^{\frac{\vartheta-\nu}{\gamma_{2}}}\right],
\end{flalign*}
therefore, merging the content of the two above displays we obtain
\begin{flalign}
\psi(x_{0},\tau \rr)\le& \tau^{\frac{\tilde{p}_{2}(\rr)}{p_{2}(\rr)}}\left[c\tau^{\frac{n-\vartheta}{\gamma_{2}}}+c\tau^{-\frac{\vartheta}{\gamma_{1}}}\left(\varepsilon^{\frac{\kk}{\gamma_{2}}}+R_{*}^{\frac{\alpha}{\gamma_{2}}}\right)+c\tau^{\frac{\vartheta-\nu}{\gamma_{2}}}\right]\nonumber \\
&\cdot\left[\psi^{+}(x_{0},\rr)+\left(\rr^{q-n}\int_{B_{\rr}^{+}(x_{0})}\snr{Dg}^{q} \ \dx\right)^{\frac{p_{2}(\rr)}{q}}\right]\nonumber \\
&+c\tau^{1-\frac{n}{q}}\left(\rr^{q-n}\int_{B_{\rr}^{+}(x_{0})}\snr{Dg}^{q} \ \dx\right)^{\frac{1}{q}}\nonumber\\
\le &\tau^{\frac{\tilde{p}_{2}(\rr)}{p_{2}(\rr)}}\left[c\tau^{\frac{n-\vartheta}{\gamma_{2}}}+c\tau^{-\frac{\vartheta}{\gamma_{1}}}\left(\varepsilon^{\frac{\kk}{\gamma_{2}}}+R_{*}^{\frac{\alpha}{\gamma_{2}}}\right)+c\tau^{\frac{\vartheta-\nu}{\gamma_{2}}}\right]\nonumber \\
&\cdot\left[\psi(x_{0},\rr)+\left(\rr^{q-n}\int_{B_{\rr}^{+}(x_{0})}\snr{Dg}^{q} \ \dx\right)^{\frac{p_{2}(\rr)}{q}}\right]\nonumber \\
&+c\tau^{1-\frac{n}{q}}\left(\rr^{q-n}\int_{B_{\rr}^{+}(x_{0})}\snr{Dg}^{q} \ \dx\right)^{\frac{1}{q}}\label{t.26}
\end{flalign}
with $c\equiv c(\texttt{data})$. Select $\tau$, $\varepsilon$ and $R_{*}$ so small that
\begin{flalign}\label{t.27}
&\tau^{\frac{\tilde{p}_{2}(\rr)}{p_{2}(\rr)}}\le \frac{1}{8},\qquad c'c\tau^{\frac{n-\vartheta}{\gamma_{2}}}\le \frac{1}{3},\qquad c'c\tau^{-\frac{\vartheta}{\gamma_{1}}}\left(\varepsilon^{\frac{\kappa}{\gamma_{2}}}+R_{*}^{\frac{\alpha}{\gamma_{2}}}\right)\le \frac{1}{3}\nonumber \\
&c'c\tau^{\frac{\vartheta-\nu}{\gamma_{2}}}\le \frac{1}{3},\qquad (c'+c)\tau^{1-\frac{n}{q}}\le\frac{1}{8},
\end{flalign}
where $c'$ is the same constant appearing in \eqref{eneq1}. By \eqref{eneq1} and \eqref{small1}, with the choice made above we can conclude that
\begin{flalign*}
\chi^{+}(x_{0},\tau \rr)\le& \frac{1}{2}\left[\psi^{+}(x_{0},\rr)+\left(\rr^{q-n}\int_{B_{\rr}^{+}(x_{0})}\snr{Dg}^{q} \ \dx\right)^{\frac{p_{2}(\rr)}{q}}\right]\nonumber \\
&+\frac{1}{2}\left(\rr^{q-n}\int_{B_{\rr}^{+}(x_{0})}\snr{Dg}^{q} \ \dx\right)^{\frac{1}{q}}<\varepsilon,
\end{flalign*}
so iterations are legal. Moreover, combining \eqref{t.26} and \eqref{t.27} we have
\begin{flalign}\label{t.28}
\psi(x_{0},\tau \rr)\le \tau^{\frac{\tilde{p}_{2}(\rr)}{p_{2}(\rr)}}\psi(x_{0},\rr)+c\rr^{1-\frac{n}{q}}\left(\int_{B_{\rr}^{+}(x_{0})}\snr{Dg}^{q} \ \dx\right)^{\frac{1}{q}}
\end{flalign}
for $c\equiv (\texttt{data},q)$. Iterating \eqref{t.28} we end up with
\begin{flalign}\label{t.29}
\psi(x_{0},\tau^{k}\rr)\le& \tau^{k\frac{\tilde{p}_{2}(\rr)}{p_{2}(\rr)}}\psi(x_{0},\rr)\nonumber \\
&+c\left(\int_{B_{\rr}^{+}(x_{0})}\snr{Dg}^{q} \ \dx\right)^{\frac{1}{q}}\rr^{1-\frac{n}{q}}\tau^{(k-1)\left(1-\frac{n}{q}\right)}\sum_{j=0}^{k-1}\tau^{j\left(\frac{\tilde{p}_{2}(\rr)}{p_{2}(\rr)}-1+\frac{n}{q}\right)}.
\end{flalign}
Since $\frac{\tilde{p}_{2}(\rr)}{p_{2}(\rr)}-1+\frac{n}{q}>0$, the series on the right-hand side of \eqref{t.28} converges, so we have
\begin{flalign}\label{t30}
\psi(x_{0},\tau^{k}\rr)\le \tau^{k\frac{\tilde{p}_{2}(\rr)}{p_{2}(\rr)}}\psi(x_{0},\rr)+c\left(\int_{B_{\rr}^{+}(x_{0})}\snr{Dg}^{q} \ \dx\right)^{\frac{1}{q}}\rr^{1-\frac{n}{q}}\tau^{(k-1)\left(1-\frac{n}{q}\right)}.
\end{flalign}
Whenever $0<\varsigma<\rr$ we can find $k\in \N$ so that $\tau^{k+1}\rr\le \varsigma<\tau^{k}\rr$, so using \eqref{t30} we obtain
\begin{flalign}\label{t31}
\psi(x_{0},\varsigma)\le& \tau^{1-\frac{n}{p_{2}(\rr)}}\psi(x_{0},\tau^{k}\rr)\nonumber \\
\le& \tau^{1-\frac{n}{p_{2}(\rr)}}\left[\tau^{k\frac{\tilde{p}_{2}(\rr)}{p_{2}(\rr)}}\psi(x_{0},\rr)+c\rr^{1-\frac{n}{q}}\tau^{(k-1)\left(1-\frac{n}{q}\right)}\left(\int_{B_{\rr}^{+}(x_{0})}\snr{Dg}^{q} \ \dx\right)^{\frac{1}{q}}\right]\nonumber \\
\le &c\tau^{-1-\frac{n}{\gamma_{1}}}\left[\left(\frac{\varsigma}{\rr}\right)^{\frac{\tilde{p}_{2}(\rr)}{p_{2}(\rr)}}\psi(x_{0},\rr)+\rr^{1-\frac{n}{q}}\left(\int_{B_{\rr}^{+}(x_{0})}\snr{Dg}^{q} \ \dx\right)^{\frac{1}{q}}\right]\nonumber \\
\le &c\left[\left(\frac{\varsigma}{\rr}\right)^{1-\frac{n}{q}}\psi(x_{0},\rr)+\varsigma^{1-\frac{n}{q}}\left(\int_{B_{\rr}^{+}(x_{0})}\snr{Dg}^{q} \ \dx\right)^{\frac{1}{q}}\right],
\end{flalign}
for $c\equiv c(\texttt{data})$. To summarize, we just got that, if $x_{0}\in \Gamma_{R}$ is any point satisfying \eqref{small1} on $B_{\rr}^{+}(x_{0})$ for some $\rr\in (0,R-\snr{x_{0}})$ then
\begin{flalign}\label{gg0}
\psi(x_{0},\varsigma)\le& c\left[\left(\frac{\varsigma}{\rr}\right)^{1-\frac{n}{q}}\psi(x_{0},\rr)+\varsigma^{1-\frac{n}{q}}\left(\int_{B_{\rr}^{+}(x_{0})}\snr{Dg}^{q} \ \dx\right)^{\frac{1}{q}}\right]\nonumber \\
\le &c\left[\left(\frac{\varsigma}{\rr}\right)^{1-\frac{n}{q}}\left(\rr^{p_{2}(\rr)}\mint_{B_{\rr}^{+}(x_{0})}(1+\snr{Du}^{2})^{\frac{p_{2}(\rr)}{2}} \ \dx\right)^{\frac{1}{p_{2}(\rr)}}+\varsigma^{1-\frac{n}{q}}\left(\int_{B_{\rr}(x_{0})}\snr{Dg}^{q} \ \dx\right)^{\frac{1}{q}}\right]\nonumber \\
\le &c\left[\left(\frac{\varsigma}{\rr}\right)^{1-\frac{n}{q}}\chi^{+}(x_{0},\rr)+\varsigma^{1-\frac{n}{q}}\nr{Dg}_{L^{q}(B_{1}^{+})}\right]\le c\left[\left(\frac{\varsigma}{\rr}\right)^{1-\frac{n}{q}}+\varsigma^{1-\frac{n}{q}}\right]\le c\left(\frac{\varsigma}{\rr}\right)^{1-\frac{n}{q}}
\end{flalign}
for $c\equiv c(\texttt{data},\nr{Dg}_{L^{q}(B_{1}^{+})})$. In \eqref{gg0} we also used \eqref{small1} to control $\chi^{+}(x_{0},\rr)$ with $\varepsilon\in (0,1]$.
\subsubsection*{Step 5: Partial H\"older continuity}
Now we aim to prove an estimate analogous to \eqref{gg0} valid also for points $x_{0}\in \bar{B}_{R}^{+}$ not necessarily belonging to $\Gamma_{R}$. As in \cite[Proof of Lemma 2]{jm}, we fix $\iota =\frac{1}{1000}$ and $x_{0}\in B_{R}^{+}$ satisfying \eqref{small1} for some $\rr\in (0,R-\snr{x_{0}})$. For $0<\sigma<\rr$ we distinguish two main cases: $\varsigma< \iota \rr$ or $\varsigma\ge \iota \rr$.\\
\emph{Case 1: $\varsigma< \iota \rr$.} We take $\hat{x}\in \Gamma_{R}$ so that $d:=\dist(x_{0},\Gamma_{R})=\snr{x_{0}-\hat{x}}$. Now, if $\iota \rr \ge d$ we notice that $$B_{d}(x_{0})\subset B_{2d}(\hat{x})\subset B_{\rr/2}(x_{0})\subset B_{\rr}(\hat{x}),$$
therefore
\begin{flalign*}
\chi^{+}\left(\hat{x},\frac{\rr}{2}\right)\le c\chi^{+}(x_{0},\rr)\le c(n,\gamma_{1},\gamma_{2},q)\varepsilon,
\end{flalign*}
so reducing the size of $\varepsilon\equiv \varepsilon(\texttt{data})$ determined in \eqref{t.27} to $\varepsilon':=\frac{\varepsilon}{2c}$ we end up with
\begin{flalign}\label{gg1}
\chi^{+}\left(\hat{x},\frac{\rr}{2}\right)<\varepsilon'.
\end{flalign}
If $\iota\rr>\varsigma \ge d$ we immediately notice that 
$$B_{\varsigma}(x_{0})\subset B_{4\varsigma}(\hat{x})\subset B_{\rr/4}(\hat{x})\subset B_{\rr}(x_{0}),$$ and, since \eqref{gg1} legalizes \eqref{gg0} with $x_{0}$ replaced by $\hat{x}$, we obtain
\begin{flalign*}
\psi(x_{0},\varsigma)\le& c\psi(\hat{x},4\varsigma)\le c8^{1-\frac{n}{q}}\left(\frac{\varsigma}{\rr}\right)^{1-\frac{n}{q}},
\end{flalign*}
for $c\equiv c(\texttt{data},\nr{Dg}_{L^{q}(B_{1}^{+})})$. On the other hand, if $\iota\rr\ge d>\varsigma$ we have that
$$
B_{\varsigma}(x_{0})\subset B_{2d}(\hat{x})\subset B_{\rr/4}(\hat{x})\subset B_{\rr}(x_{0}),
$$
so, by \eqref{gg1} and \eqref{gg0} (with $x_{0}$ replaced by $\hat{x}$ of course) we get
\begin{flalign*}
\psi(x_{0},\varsigma)\le& c\frac{\varsigma}{d}\psi(\hat{x},2d)\le c\frac{\varsigma}{d}\left[\left(\frac{4d}{\rr}\right)^{1-\frac{n}{q}}+(2d)^{1-\frac{n}{q}}\right]\nonumber \\
\le &c\left(\frac{\varsigma}{d}\right)^{\frac{n}{q}}\left[\left(\frac{\varsigma}{\rr}\right)^{1-\frac{n}{q}}+\varsigma^{1-\frac{n}{q}}\right]\le c\left[\left(\frac{\varsigma}{\rr}\right)^{1-\frac{n}{q}}+\varsigma^{1-\frac{n}{q}}\right]\le c\left(\frac{\varsigma}{\rr}\right)^{1-\frac{n}{q}}
\end{flalign*}
also in this case, with $c\equiv c(\texttt{data},\nr{Dg}_{L^{q}(B_{1}^{+})})$.  Now we consider the occurrence $\varsigma<\iota \rr<d$. It follows that $B_{\iota \rr}(x_{0})\Subset B_{R}^{+}$ and
\begin{flalign*}
2\iota \rr\left(\mint_{B_{2\iota\rr}(x_{0})}(1+\snr{Du}^{2})^{\frac{p_{2}(2\iota\rr)}{2}} \ \dx\right)^{\frac{1}{p_{2}(2\iota\rr)}}\le c\iota^{1-\frac{n}{\gamma_{1}}}\chi^{+}(x_{0},\rr)<c(n,\gamma_{1},\gamma_{2},q)\varepsilon.
\end{flalign*}
Reducing the size of $\varepsilon$ in such a way that $c\varepsilon\le \varepsilon_{0}$, where $\varepsilon_{0}$ is the smallness threshold appearing in \cite[(3.16)]{decv} we obtain
\begin{flalign*}
2\iota \rr\left(\mint_{B_{2\iota\rr}(x_{0})}(1+\snr{Du}^{2})^{\frac{p_{2}(2\iota\rr)}{2}} \ \dx\right)^{\frac{1}{p_{2}(2\iota\rr)}}<\varepsilon_{0}.
\end{flalign*}
Then, \cite[estimates (3.40)-(3.43)]{decv} apply, thus getting
\begin{flalign*}
\left(\varsigma^{p_{2}(\varsigma)}\mint_{B_{\varsigma}(x_{0})}(1+\snr{Du}^{2})^{p_{2}(\varsigma)} \ \dx\right)^{\frac{1}{p_{2}(\varsigma)}}\le \iota^{-\beta_{0}}c(\texttt{data},\beta_{0})\left(\frac{\varsigma}{\rr}\right)^{\beta_{0}},
\end{flalign*}
for all $\beta_{0}\in (0,1)$. Recalling the explicit expression of $\psi(x_{0},\cdot)$ we then bound
\begin{flalign*}
\psi(x_{0},\sigma)\le& c\left(\varsigma^{p_{2}(\varsigma)}\mint_{B_{\varsigma}(x_{0})}(1+\snr{Du}^{2})^{\frac{p_{2}(\varsigma)}{2}} \ \dx\right)^{\frac{1}{p_{2}(\varsigma)}}\nonumber \\
&+c\left(\varsigma^{q-n}\int_{B_{\varsigma}(x_{0})\cap B_{R}^{+}}\snr{Dg}^{q} \ \dx \right)^{\frac{1}{q}}\le c\left[\left(\frac{\varsigma}{\rr}\right)^{\beta_{0}}+\varsigma^{1-\frac{n}{q}}\right],
\end{flalign*}
with $c\equiv c(\texttt{data},\nr{Dg}_{L^{q}(B_{1}^{+})},\beta_{0})$. The desired estimate follows by choosing $\beta_{0}=1-\frac{n}{q}$ in the previous display.\\
\emph{Case 2: $\varsigma\ge \iota \rr$.} Estimate \eqref{gg0} trivially holds with a costant $c\equiv c(\texttt{data},\nr{Dg}_{L^{q}(B_{1}^{+})})$. \\
All in all, we have just proved that if $x_{0}\in \bar{B}_{R}^{+}$ satisfies \eqref{small1} on $B_{\rr}(x_{0})\cap B_{R}^{+}$ for some $\rr\in (0,R-\snr{x_{0}})$, then
\begin{flalign}\label{gg2}
\psi(x_{0},\varsigma)\le c(\texttt{data},\nr{Dg}_{L^{q}(B_{1}^{+})})\left(\frac{\varsigma}{\rr}\right)^{1-\frac{n}{q}}.
\end{flalign}
Now, by the continuity of Lebesgue's integral and of the mapping $x_{0}\mapsto p_{2}(x_{0},\rr)$, we can conclude that if \eqref{small1} holds for $x_{0}$ on $B_{\rr}(x_{0})\cap B_{R}^{+}$ then it holds also on $B_{\rr}(y)\cap B_{R}^{+}$ for all $y\in \bar{B}_{1}^{+}$ belonging to a sufficiently small, relatively open neighborhood of $x_{0}$, say, $B_{x_{0}}\subset \bar{B}_{R}^{+}$. Then the set
\begin{flalign*}
D_{0}:=\left\{y\in B_{x_{0}}\colon \chi^{+}(y,\rr)<\varepsilon \ \mbox{on} \ B_{\rr}(y)\cap B_{R}^{+}, \ R\in (0,R_{*}], \ \rr\in (0,R-\snr{y}) \right\}
\end{flalign*}
is relatively open, so via \eqref{gg2} we can conclude that
\begin{flalign}\label{t32}
&\left(\varsigma^{-n\left(1-\frac{\gamma_{1}}{q}\right)}\int_{B_{\varsigma}(x_{0})}\snr{D\tilde{u}}^{\gamma_{1}} \ \dx\right)^{\frac{1}{\gamma_{1}}}\le \psi(x_{0},\varsigma)\le c\left(\frac{\varsigma}{\rr}\right)^{1-\frac{n}{q}},
\end{flalign}
where $c\equiv c(\texttt{data},\nr{Dg}_{L^{q}(B_{1}^{+})})$. By \eqref{t32} and the well-known characterization of H\"older continuity due to Campanato and Meyers we can conclude that $\tilde{u}$ is $\left(1-\frac{n}{q}\right)$-H\"older continuous in a neighborhood of $D_{0}$, which in turn implies that $u\in C^{0,1-\frac{n}{q}}_{loc}(D_{0},\m)$.
\subsubsection*{Step 6: Hausdorff dimension of the singular set} Given the characterization of $D_{0}$, we easily see that the singular set $\Sigma_{0}(u,B_{\rr}(x_{0})\cap B_{R}^{+})$ can be defined as 
$$
\Sigma_{0}(u,B_{\rr}(x_{0})\cap B_{R}^{+}):=\bar{B}_{R}^{+}\cap \bar{B}_{\rr}(x_{0})\setminus D_{0}.
$$
Moreover, for $y\in B_{\rr}(x_{0})\cap B_{R}^{+}$, being $g\in C^{0,1-\frac{n}{q}}(\bar{B}_{1}^{+},\m)$ we see that
\begin{flalign*}
\limsup_{\varsigma\to 0}\chi^{+}(y,\varsigma)\le& \limsup_{\varsigma\to 0}\left(\varsigma^{p_{2}(y,\varsigma)-n}\int_{B_{\varsigma}(y)\cap B_{R}^{+}}(1+\snr{Du}^{2})^{\frac{p_{2}(y,\varsigma)}{2}} \ \dx\right)^{\frac{1}{p_{2}(y,\varsigma)}}\nonumber \\
&+\limsup_{\varsigma\to 0}\left(\varsigma^{q-n}\int_{B_{\varsigma}(y)\cap B_{R}^{+}}\snr{Dg}^{q} \ \dx\right)^{\frac{1}{q}}\nonumber \\
\le &\limsup_{\varsigma\to 0}\left(\varsigma^{p_{2}(y,\varsigma)-n}\int_{B_{\varsigma}(y)\cap B_{R}^{+}}(1+\snr{Du}^{2})^{\frac{p_{2}(y,\varsigma)}{2}} \ \dx\right)^{\frac{1}{p_{2}(y,\varsigma)}},
\end{flalign*}
therefore
\begin{flalign*}
&\Sigma_{0}(u,B_{\rr}(x_{0})\cap B_{R}^{+})\subset \left\{y \in \bar{B}_{\rr}(x_{0})\cap \bar{B}_{R}^{+}\colon \limsup_{\varsigma\to 0}\psi^{+}(y,\varsigma)>0\right\}.
\end{flalign*}
Now, notice that, as in \eqref{t12},
\begin{flalign}\label{61}
p_{2}(y,\varsigma)\le (1+\sigma_{0})p_{1}(x_{0},R_{*}) \ \ \mbox{for all} \ \ 0<\varsigma\le R_{*}, \ \ B_{\varsigma}(y)\cap B_{R}^{+}\subset B_{\rr}(x_{0})\cap B_{R}^{+},
\end{flalign}
so we obtain, 
\begin{flalign*}
&\left(\varsigma^{p_{2}(y,\varsigma)}\mint_{B_{\varsigma}(y)\cap B_{R}^{+}}(1+\snr{Du}^{2})^{\frac{p_{2}(y,\varsigma)}{2}} \ \dx\right)^{\frac{1}{p_{2}(y,\varsigma)}}\nonumber \\
&\qquad \le\left(\varsigma^{p_{1}(x_{0},R_{*})(1+\sigma_{0})}\mint_{B_{\varsigma}(y)\cap B_{R}^{+}}(1+\snr{Du}^{2})^{\frac{p_{1}(x_{0},R_{*})(1+\sigma_{0})}{2}} \ \dx\right)^{\frac{1}{p_{1}(x_{0},R_{*})(1+\sigma_{0})}},
\end{flalign*}
which by \eqref{bougeh2} is finite. This allows concluding that $\Sigma_{0}(u,B_{\rr}(x_{0})\cap B_{R}^{+})$ is contained into the set
\begin{flalign*}
D_{1}:=\left\{y \in \bar{B}_{\rr}(x_{0})\cap\bar{B}_{R}^{+}\colon \limsup_{\varsigma\to 0}\phi^{+}(y,\varsigma,p_{1}(x_{0},R_{*})(1+\sigma_{0}))^{p_{1}(x_{0},R_{*})(1+\sigma_{0})}>0\right\}.
\end{flalign*}
By \cite[Proposition 2.7]{giu} it follows that $\dim_{\mathcal{H}}(D_{1})\le n-p_{1}(x_{0},R_{*})(1+\sigma_{0})$, so, by $\eqref{asp}_{2}$ we easily have that $\dim_{\mathcal{H}}(D_{1})< n-\gamma_{1}$ and so $\dim_{\mathcal{H}}(\Sigma_{0}(u,B_{\rr}(x_{0})\cap B_{R}^{+}))< n-\gamma_{1}$. The proof is complete.
\end{proof}
Once Proposition \ref{pp1} is available, we can cover $B^{+}_{1}$ with balls having the same features of $B_{\rr}(x_{0})\cap B_{R}^{+}$ and remembering that, by $\eqref{asp}_{2}$, $p_{1}(x_{0},R_{*})\ge \gamma_{1}$, we obtain that
$
\dim_{\mathcal{H}}(\Sigma_{0}(u))\le n-\gamma_{1}(1+\delta_{0})<n-\gamma_{1},
$
and so $\dim_{\mathcal{H}}(\Sigma_{0}(u))<n-\gamma_{1}$. Via a standard covering argument, we can conclude that $u\in C^{0,1-\frac{n}{q}}_{loc}(\bar{B}_{1}^{+}\setminus \Sigma_{0}(u),\m)$ and the proof of Theorem \ref{t1} is complete.
\section{Full boundary regularity}\label{fullb}
In this section we recover a regularity criterion based on the result in Theorem \ref{t1}. The main preliminary step consists in proving compactness of sequences of minimizers of \eqref{pd1} under uniform assumptions, see \cite{decv,dugrkr,rata}.
\begin{remark}
\emph{We will always assume that $\gamma_{2}<n$, otherwise, as stressed in \emph{Step 1} of the proof of Theorem \ref{t1}, we would have $u$ H\"older continuous in a small neighborhood of any point $\bar{x}\in \bar{B}_{1}^{+}$ so that $p(\bar{x})\ge n$ for free by Morrey's embedding theorem.}
\end{remark}
\begin{lemma}\label{comp}
Let $\{k_{j}\}, \{p_{j}\}$ be two sequences of H\"older continuous functions satisfying
\begin{flalign}\label{com1}
\begin{cases}
    \ \sup_{j \in \N}[k_{j}]_{0,\nu}<c_{k} \quad \mbox{for some} \ \ \nu\in (0,1]\\
\ \lambda\le k_{j}(x)\le \Lambda \ \mathrm{for \ all \ }x \in \bar{B}_{1}^{+}\\
\ \nr{k_{j}-k_{0}}_{L^{\infty}(\bar{B}_{1}^{+})}\to 0, \ k_{0}(\cdot) \in C^{0,\nu}(\bar{B}_{1}^{+})
\end{cases} 
\end{flalign}
and
\begin{flalign}\label{com2}
\begin{cases}
\ \sup_{j \in \N}[p_{j}]_{0,\alpha}<c_{p}\quad \mbox{for some} \ \ \alpha\in (0,1]\\
\ p_{j}(x)\ge \gamma_{1}>1 \ \mathrm{for \ all \ }x \in \bar{B}_{1}^{+}, j \in \N\\
\ \nr{p_{j}-p_{0}}_{L^{\infty}(\bar{B}_{1}^{+})}\to 0,\ p_{0}\ge \gamma_{1}>1 \ \mathrm{constant},
\end{cases}
\end{flalign}
respectively. For each $j \in \N$, let $u_{j}\in W^{1,p_{j}(\cdot)}(B_{1}^{+},\m)$ be a constrained minimizer of
\begin{flalign*}
\mathcal{E}_{j}(w,B_{1}^{+}):=\int_{B_{1}^{+}}k_{j}(x)\snr{Dw}^{p_{j}(x)} \ \dx,
\end{flalign*}
in class $\mathcal{C}_{g_{j}}^{p_{j}(\cdot)}(B_{1}^{+},\m)$, where the manifold $\m$ is as in \eqref{m} and the sequence $\{g_{j}\}\subset W^{1,q}(\bar{B}_{1}^{+},\m)$, uniformly satisfying \eqref{g}, is weakly convergent to some $g_{0}\in W^{1,q}(\bar{B}_{1}^{+},\m)$. Then, there exists a subsequence, still denoted by $\{u_{j}\}$, such that 
\begin{flalign}\label{com3}
u_{j}\rightharpoonup u_{0} \ \ \mathrm{weakly \ in \ }W^{1,(1+\tilde{\sigma})p_{0}}(B_{R}^{+},\m)
\end{flalign}
for some $\tilde{\sigma}>0$ and any $R \in (0,1)$. In particular, $u_{0}$ is a constrained minimizer of the functional
\begin{flalign*}
\mathcal{E}_{0}(w,B_{R}^{+}):=\int_{B_{R}^{+}}k_{0}(x)\snr{Dw}^{p_{0}} \ \dx
\end{flalign*}
in class $\mathcal{C}^{p_{0}}_{g_{0}}(B_{R}^{+},\m)$. Moreover,
\begin{flalign*}
\mathcal{E}_{j}(u_{j},B_{R}^{+})\to \mathcal{E}_{0}(u_{0},B_{R}^{+})\quad \mbox{for all} \ \ R\in (0,1).
\end{flalign*}
Finally, if $x_{j}$ is a singular point of $u_{j}$ and $x_{j}\to x_{0}$, then $x_{0}$ is a singular point for $u_{0}$.
\end{lemma}
\begin{proof}
For the reader's convenience, we split the proof into three steps.
\subsubsection*{Step 1: Weak convergence.} By assumption, the sequence $\{g_{j}\}$ is weakly convergent in $W^{1,q}(\bar{B}_{1},\m)$, so, we can find a positive, finite constant $M\equiv M(n,\m,q)$ so that
\begin{flalign}\label{t23}
\sup_{j\in N}\nr{g_{j}}_{W^{1,q}(B_{1}^{+})}\le M.
\end{flalign}
Since the whole sequence $\{u_{j}\}$ has image contained into $\m$, which, by $\eqref{m}_{1}$ is compact, we immediately have that $\sup_{j\in \N}\nr{u_{j}}_{L^{\infty}(B_{1}^{+})}\le c(\m)<\infty$, thus, up to extracting a non-relabelled subsequence, 
\begin{flalign}\label{t21}
u_{j}\rightharpoonup u_{0}\quad \mbox{weakly in} \ \ L^{t}(B_{1}^{+},\m)\quad \mbox{for all} \ \ t\in (1,\infty).
\end{flalign}
Moreover, being the assumptions in \eqref{com1}-\eqref{com2} uniform in $j\in \N$, we deduce that Lemma \ref{geh} and Lemma \ref{gehp} for the associated frozen problem hold with constants independent from $j$. In particular, recalling the uniform features of the $g_{j}$'s and combining \eqref{glob} with a standard covering argument we can conclude that $\{u_{j}\}\subset W^{1,p(\cdot)(1+\sigma)}_{loc}(B_{1}^{+},\m)$ for all $\sigma\in\left[0,\min\left\{\sigma_{g},\sigma_{g}'\right\}\right)$. Now we take any ball $B_{\rr}(x_{0})\subset B_{1}$ with $\rr\in \left(0,\frac{1}{4}\min\{1-\snr{x_{0}},R_{*}\}\right]$ and $R_{*}$ as in \eqref{r*}, so we can apply $\eqref{ca1}_{2}$ with any $\sigma\in \left(0,\min\left\{\sigma_{g},\frac{n-\gamma_{2}}{\gamma_{2}},\frac{q}{n}-1\right\}\right)$ to deduce that
\begin{flalign}\label{t22}
\int_{B_{\rr}(x_{0})\cap B_{1}^{+}}\snr{Du_{j}}^{p_{j}(x)(1+\sigma)} \ \dx \le c\rr^{n-p_{2}(x_{0},\rr)(1+\sigma)}\le c(n,N,\m,\gamma_{1},\gamma_{2},q).
\end{flalign} 
In \eqref{t22} we also used \eqref{t23} to incorporate the dependency of the constant from $\nr{Dg_{j}}_{L^{q}(B_{1}^{+})}$ into the one from $(n,\m,q)$. Now set
\begin{flalign*}
\hat{\sigma}_{g}:=\frac{1}{4}\min\left\{\sigma_{g},\sigma'_{g},\delta_{g},\frac{n-\gamma_{2}}{\gamma_{2}},\frac{q}{n}-1\right\},
\end{flalign*}
where $\sigma_{g}$, $\sigma'_{g}$ and $\delta_{g}$ are the same higher integrability threshold determined in Lemmas \ref{geh}, \ref{gehp} respectively and choose any $\sigma\in (0,\hat{\sigma}_{g})$. Because of the uniform convergence of the $p_{j}$'s to the constant $p_{0}$, taking $j\in \N$ sufficiently large we can find positive constants $\gamma_{1}\le q_{1}\le q_{2}\le \gamma_{2}$ such that
\begin{flalign}\label{t24}
&1<q_{1}\le p_{j}(\cdot)\le q_{2}<\infty \ \mbox{on} \ \bar{B}_{1}^{+}, \quad q_{2}\left(1+\frac{\sigma}{2}\right)\le q_{1}(1+\sigma), \quad q_{2}\le p_{0}\left(1+\frac{\sigma}{2}\right)
\end{flalign}
and
\begin{flalign}
0\le q_{2}-q_{1}<\frac{\delta_{g}\gamma_{1}}{16}\quad\mbox{and}\quad 1\le \frac{q_{2}}{q_{1}}<2\label{t24.a}.
\end{flalign}
Combining \eqref{t22}, \eqref{t24} and the choice of $\sigma>0$ we made, we can conclude that
\begin{flalign}\label{t25}
\int_{B_{\rr}(x_{0})\cap B_{1}^{+}}\snr{Du_{j}}^{q_{2}\left(1+\frac{\sigma}{2}\right)} \ \dx \le c(n,\m,\gamma_{1},\gamma_{2},q).
\end{flalign}
By \eqref{t21} and \eqref{t25} we derive the uniform boundedness of the $u_{j}$'s in $W^{1,(1+\sigma/2)q_{2}}(B_{\rr}(x_{0})\cap B_{1}^{+},\m)$, so, up to extract a (non relabelled) subsequence, we obtain that $u_{j}\rightharpoonup \bar{u}_{0}$ weakly in $W^{1,(1+\sigma/2)q_{2}}(B_{\rr}(x_{0})\cap B_{1}^{+},\m)$, for some $\bar{u}_{0}\in W^{1,(1+\sigma/2)q_{2}}(B_{\rr}(x_{0})\cap B_{1}^{+},\m)$. Anyway, by \eqref{t21}, $\bar{u}_{0}(x)=u_{0}(x)$, $u_{0}(x)\in \m$ for a.e. $x\in B_{\rr}(x_{0})\cap B_{1}^{+}$ and, by Rellich-Kondrachov theorem,
\begin{flalign}
&u_{j}\to u_{0} \quad \mbox{strongly in} \ \ L^{(1+\sigma/2)q_{2}}(B_{\rr}(x_{0})\cap B_{1}^{+},\m),\label{t26}\\
&Du_{j}\to Du_{0} \quad \mbox{weakly in}\ \ L^{(1+\sigma/2)q_{2}}(B_{\rr}(x_{0})\cap B_{1}^{+},\mathbb{R}^{N\times n}).\label{t27}
\end{flalign}
From $\eqref{t24}_{1}$ and $\eqref{com2}_{3}$, we see that $q_{2}\ge p_{0}$, therefore \eqref{com3} is proved for instance with 
\begin{flalign}\label{tisi}
\tilde{\sigma}=\frac{\hat{\sigma}_{g}}{4}.
\end{flalign}
Using the lower semicontinuity of the norm, we also have that
\begin{flalign}\label{t28}
\int_{B_{\rr}(x_{0})\cap B_{1}^{+}}\snr{Du_{0}}^{q_{2}\left(1+\frac{\tilde{\sigma}}{2}\right)} \ \dx \le c(\texttt{data}_{p(\cdot)}).
\end{flalign}
Inequality \eqref{t28} and the convergence in \eqref{t26}-\eqref{t27} hold on $B_{\rr}(x_{0})\cap B_{1}^{+}$, but can show that they actually hold on half balls having any radius $R\in (0,1)$. In fact, being $\bar{B}_{1}^{+}$ compact, we can find $m\equiv m(n)$ and a finite family of balls $\left\{B_{\rr_{k}}(x_{k})\right\}_{k=1}^{m}$ so that $\{\rr_{k}\}\subset \left(0,\frac{R_{*}}{4}\right)$ and $B_{1}^{+}\subseteq \bigcup_{k=1}^{m}B_{\rr_{k}}(x_{k})$. Then, given any measurable subset $U\subseteq B_{R}^{+}$ with $R\in (0,1)$, we trivially have that $U\subseteq\bigcup_{k=1}^{m}\left(B_{\rr_{k}}(x_{k})\cap B_{1}^{+}\right)$ and, recalling \eqref{t26}, \eqref{t27} and \eqref{t28},
\begin{flalign}
&\int_{U}\snr{Du_{0}}^{q_{2}\left(1+\frac{\tilde{\sigma}}{2}\right)} \ \dx \le \sum_{k=1}^{m}\int_{B_{\rr_{k}}(x_{k})\cap B_{1}^{+}}\snr{Du_{0}}^{q_{2}\left(1+\frac{\tilde{\sigma}}{2}\right)} \ \dx \le mc\le c(\texttt{data}_{p(\cdot)})\label{t243}\\
&\nr{Du_{j}}_{L^{q_{2}\left(1+\frac{\tilde{\sigma}}{2}\right)}(U)}\le \sum_{k=1}^{m}\nr{Du_{j}}_{L^{q_{2}\left(1+\frac{\tilde{\sigma}}{2}\right)}(B_{\rr_{k}}(x_{k})\cap B_{1}^{+})}\le c(\texttt{data}_{p(\cdot)})\label{t244}\\
&\nr{u_{j}-u_{0}}_{L^{q_{2}\left(1+\frac{\tilde{\sigma}}{2}\right)}(U)}\le \sum_{k=1}^{m}\nr{u_{j}-u_{0}}_{L^{q_{2}\left(1+\frac{\tilde{\sigma}}{2}\right)}(B_{\rr_{k}}(x_{k})\cup B_{1}^{+})}\to 0,\label{t245}
\end{flalign}
so \eqref{com3} is completely proved. Notice that, \eqref{com3} and the weak continuity of the trace operator yield in particular that
\begin{flalign}\label{t240}
\texttt{tr}_{\Gamma_{R}}(u_{0})=\texttt{tr}_{\Gamma_{R}}(g_{0})\quad \mbox{for all} \ \ R\in (0,1).
\end{flalign}
\subsubsection*{Step 2: Compactness.} We fix $R\in (0,1)$ and, as a first step towards the proof of the minimality of $\mathcal{E}_{0}(u_{0},B_{R}^{+})$ in class $\mathcal{C}^{p_{0}}_{g_{0}}(B_{R}^{+},\m)$ we show that 
\begin{flalign}\label{t29}
\mathcal{E}_{0}(u_{0},B_{R}^{+})\le \liminf_{j\to \infty}\mathcal{E}_{j}(u_{j},B_{R}^{+}).
\end{flalign}
Since 
$$\mathcal{E}_{j}(u_{j},B_{R}^{+})=\left(\mathcal{E}_{j}(u_{j},B_{R}^{+})-\mathcal{E}_{0}(u_{j},B_{R}^{+})\right)+\mathcal{E}_{0}(u_{j},B_{R}^{+})$$
and, by weak lower semicontinuity and \eqref{t21} there holds that
\begin{flalign}\label{t231}
\mathcal{E}_{0}(u_{0},B_{R}^{+})\le \liminf_{j\to \infty}\mathcal{E}_{0}(u_{j},B_{R}^{+}),
\end{flalign}
we only need to show that
\begin{flalign}\label{t230}
\snr{\mathcal{E}_{j}(u_{j},B_{R}^{+})-\mathcal{E}_{0}(u_{j},B_{R}^{+})}\to 0,
\end{flalign}
which is a consequence of $\eqref{com1}_{3}$, $\eqref{com2}_{3}$, Lemma \ref{L0} (\emph{i.}) with $\varepsilon_{0}=\frac{\sigma}{2}$ and \eqref{t25}. In fact,
\begin{flalign*}
&\snr{\mathcal{E}_{j}(u_{j},B_{R}^{+})-\mathcal{E}_{0}(u_{j},B_{R}^{+})}\le \left| \ \int_{B_{R}^{+}}(k_{j}(x)-k_{0}(x))\snr{Du_{j}}^{p_{j}(x)} \ \dx \ \right|\nonumber \\
&\qquad + \left| \ \int_{B_{R}^{+}}k_{0}(x)\left[\snr{Du}^{p_{j}(x)}-\snr{Du_{j}}^{p_{0}}\right] \ \dx \ \right|\le \nr{k_{j}-k_{0}}_{L^{\infty}(B_{R}^{+})}\int_{B_{R}^{+}}\snr{Du}^{p_{j}(x)} \ \dx \nonumber \\
&\qquad +c\nr{p_{j}-p_{0}}_{L^{\infty}(B_{R}^{+})}\int_{B_{R}^{+}}(1+\snr{Du_{j}}^{2})^{\frac{q_{2}}{2}\left(1+\frac{\sigma}{2}\right)} \ \dx\nonumber \\
&\qquad \le c\left[\nr{k_{j}-k_{0}}_{L^{\infty}(B_{R}^{+})}+\nr{p_{j}-p_{0}}_{L^{\infty}(B_{R}^{+})}\right]\to 0.
\end{flalign*}
The constant $c$ appearing in the previous display depends only from $\texttt{data}_{p(\cdot)}$. Combining \eqref{t230} and \eqref{t231} we end up with \eqref{t29}. Now, let $\tilde{u}_{0}\in W^{1,p_{0}}(B_{R}^{+},\m)$ be a solution of the Dirichlet problem
\begin{flalign}\label{pd4}
\hat{\mathcal{C}}^{p_{0}}_{u_{0}}(B_{R}^{+},\m)\ni w\mapsto \min \mathcal{E}_{0}(w,B_{R}^{+}).
\end{flalign}
As in \cite{decv,dugrkr,harkinlin} we fix any $\theta\in (0,1)$, a cut-off function $\eta\in C^{1}_{c}(B_{R})$ satisfying
\begin{flalign}\label{t232}
\mathds{1}_{B_{(1-\theta)R}}\le \eta\le \mathds{1}_{B_{R}}\quad \mbox{and}\quad \snr{D\eta}\lesssim \frac{1}{R\theta},
\end{flalign}
and consider a bi-Lipschitz transformation $\Phi\colon \bar{B}_{R}^{+}\to \bar{B}_{R}$ so that 
\begin{flalign}\label{t233}
\left.\Phi\right|_{\partial^{+} B_{R}^{+}}=\mathds{Id}_{\partial^{+} B_{R}^{+}}\quad\mbox{and}\quad \Phi(\Gamma_{R})=\left\{x\in \partial B_{R}\colon x^{n}<0\right\}. 
\end{flalign}
Being $\Phi$ bi-Lipschitz, if $\mathcal{J}_{\Phi}$ is its jacobian, we have that
\begin{flalign}\label{t241}
0<c(n)^{-1}\le\snr{\mathcal{J}_{\Phi}(x)}\le c(n)<\infty.
\end{flalign}
Let us look at the function
\begin{flalign*}
\tilde{u}_{j}(x):=\tilde{u}_{0}(x)+\left(1-\eta(\Phi(x))\right)(u_{j}(x)-u_{0}(x))\quad \mbox{for} \ \ x\in B_{R}^{+}.
\end{flalign*}
By $\eqref{t232}_{1}$, \eqref{t233} and \eqref{t240} we see that
\begin{flalign}\label{t239}
B_{R}^{+}\cap \left\{0\le \eta(\Phi(x))<1\right\}\subseteq B_{R}^{+}\cap \Phi^{-1}(\bar{B}_{R}\setminus \bar{B}_{(1-\theta)R} ),
\end{flalign}
and since, 
\begin{flalign*}
\partial \left(B_{R}^{+}\cap \left\{0\le \eta(\Phi(x))<1\right\}\right)=\partial B_{R}^{+}\cup\partial\left\{\eta(\Phi(x))=1\right\},
\end{flalign*}
also that
\begin{flalign}\label{t237}
&\mbox{in a neighborhood of } \ \partial \left(B_{R}^{+}\cap \left\{0\le \eta(\Phi(x))<1\right\}\right), \ \tilde{u}_{j}  \ \mbox{takes values in} \ \m.
\end{flalign}
In particular, according to \eqref{t240} and to the definition given in \eqref{pd4},
\begin{flalign}
\texttt{tr}_{\Gamma_{R}}(\tilde{u}_{j})=\texttt{tr}_{\Gamma_{R}}(g_{j}),
\quad \texttt{tr}_{\partial^{+} B_{R}^{+}}\tilde{u}_{j}=\texttt{tr}_{\partial^{+} B_{R}^{+}}(u_{j}),\quad \texttt{tr}_{\partial\left\{\eta(\Phi(x))=1\right\}}=\texttt{tr}_{\partial\left\{\eta(\Phi(x))=1\right\}}(\tilde{u}_{0}). \label{t238}
\end{flalign}
Conditions \eqref{t237}-\eqref{t238} legalize the application of Lemma \ref{exlem} on the set $B_{R}^{+}\cap \left\{0\le \eta(\Phi(x))<1\right\}$ to end up with a function $\bar{w}_{j}\in W^{1,p_{j}(\cdot)}_{loc}(B_{1}^{+},\m)$ satisfying
\begin{flalign}\label{t235}
\begin{cases}
\ \bar{w}_{j}\left(\partial \left(B_{R}^{+}\cap \left\{0\le \eta(\Phi(x))<1\right\}\right)\right)\subset \m\\
\ \texttt{tr}_{\Gamma_{R}}(\bar{w}_{j})=\texttt{tr}_{\Gamma_{R}}(g_{j})\\
\ \texttt{tr}_{\partial^{+} B_{R}^{+}}(\bar{w}_{j})=\texttt{tr}_{\partial^{+} B_{R}^{+}}(u_{j})\\
\ \texttt{tr}_{\partial\left\{\eta(\Phi(x))=1\right\}}(\bar{w}_{j})=\texttt{tr}_{\partial\left\{\eta(\Phi(x))=1\right\}}(g_{0}),\\
\ \int_{B_{R}^{+}\cap \left\{0\le \eta(\Phi(x))<1\right\}}\snr{D\bar{w}_{j}}^{p_{j}(x)} \ \dx \lesssim \int_{B_{R}^{+}\cap \left\{0\le \eta(\Phi(x))<1\right\}}\snr{D\tilde{u}_{j}}^{p_{j}(x)} \ \dx
\end{cases}
\end{flalign}
with constants implicit in "$\lesssim$" depending by $(N,\m,\gamma_{2})$. Finally, define
\begin{flalign*}
\tilde{w}_{j}(x):=\begin{cases}
\ \tilde{u}_{0}(x)\quad &\mbox{if} \ \ x\in B_{R}^{+}\cap \left\{\eta(\Phi(x))=1\right\}\\
\ \bar{w}_{j}(x)\quad &\mbox{if} \ \ x\in B_{R}^{+}\cap \left\{0\le \eta(\Phi(x))<1\right\}.
\end{cases}
\end{flalign*}
Now, notice that the choices we made in \eqref{t24.a} and \eqref{tisi} imply that
\begin{flalign}\label{t24.b}
\frac{q_{2}}{p_{0}}\left(1+\frac{\tilde{\sigma}}{2}\right)=1+\left[\frac{q_{2}-p_{0}}{p_{0}}+\frac{q_{2}\tilde{\sigma}}{2p_{0}}\right]<1+\frac{\delta_{g}}{8},
\end{flalign}
so by Lemma \ref{gehp}, \eqref{t24.b}, \eqref{t243} and the minimality of $\tilde{u}_{0}$ in class $\hat{\mathcal{C}}^{p_{0}}_{u_{0}}(B_{R}^{+},\m)$, we get
\begin{flalign}\label{t242}
\int_{B_{R}^{+}\cap\{0\le \eta(\Phi(x))<1\}}&\snr{D\tilde{u}_{0}}^{p_{j}(x)} \ \dx \le \snr{B_{R}^{+}\cap\{0\le \eta(\Phi(x))<1\}}\nonumber \\
&+\int_{B_{R}^{+}}\snr{D\tilde{u}_{0}}^{q_{2}\left(1+\frac{\tilde{\sigma}}{2}\right)} \ \dx\nonumber \\
\le & \snr{B_{R}^{+}\cap\{0\le \eta(\Phi(x))<1\}}+c\int_{B_{R}^{+}}\snr{Du_{0}}^{q_{2}\left(1+\frac{\tilde{\sigma}}{2}\right)} \ \dx<\infty,
\end{flalign}
for $c\equiv c(\texttt{data}_{p(\cdot)})$. In \eqref{t242} we used, in particular, that
\begin{flalign}\label{t242a}
\int_{B_{R}^{+}}\snr{D\tilde{u}_{0}}^{q_{2}\left(1+\frac{\tilde{\sigma}}{2}\right)} \ \dx\le c\int_{B_{R}^{+}}\snr{Du_{0}}^{q_{2}\left(1+\frac{\tilde{\sigma}}{2}\right)} \ \dx.
\end{flalign}
Via \eqref{t241}, \eqref{t239} and a straightforward change of variables we have
\begin{flalign}\label{t236}
&\left| \ B_{R}^{+}\cap \left\{0\le \eta(\Phi(x))<1\right\} \ \right|=\int_{B_{R}^{+}}\mathds{1}_{\{0\le \eta(\Phi(x))<1\}} \ \dx\nonumber \\
&\qquad \le \int_{B_{R}^{+}\cap \{\Phi^{-1}(\bar{B}_{R}\setminus B_{(1-\theta)R})\}} \dx\le \int_{B_{R}\setminus \bar{B}_{(1-\theta)R}}\snr{\mathcal{J}_{\Phi}(x)}^{-1} \ \dx\nonumber \\
&\qquad \le c(n)\snr{\bar{B}_{R}\setminus \bar{B}_{(1-\theta)R}}\to 0\quad \mbox{as} \ \ \theta\to 0.
\end{flalign}
We then estimate
\begin{flalign}\label{t246}
\mathcal{E}_{j}(u_{j},B_{R}^{+})\le& \mathcal{E}_{j}(\tilde{w}_{j},B_{R}^{+})\nonumber \\
\le &\mathcal{E}_{j}(\tilde{w}_{j},B_{R}^{+}\cap\{0\le\eta(\Phi(x))<1\})+\mathcal{E}_{j}(\tilde{u}_{0},B_{R}^{+}\cap\{\eta(\Phi(x))=1\})\nonumber \\
=:&\mbox{(I)}_{j}+\mbox{(II)}_{j}.
\end{flalign}
In the previous display, we used that, in view of $\eqref{t235}_{2,3}$, $\tilde{w}_{j}$ is a legitimate comparison map to $u_{j}$.
The bounds in \eqref{t242}, $\eqref{t235}_{4}$ and \eqref{t236} then legalize the following estimate:
\begin{flalign*}
\mbox{(I)}_{j}\le &c\int_{B_{R}^{+}\cap\{0\le \eta(\Phi(x))<1\}}\left[\snr{D\tilde{u}_{0}}^{p_{j}(x)}+\snr{Du_{j}-Du_{0}}^{p_{j}(x)}+\left| \ \frac{u_{j}-u_{0}}{R\theta} \ \right|^{p_{j}(x)}\right] \ \dx \nonumber \\
\le&c\int_{B_{R}^{+}\cap\{0\le \eta(\Phi(x))<1\}}\snr{D\tilde{u}_{0}}^{p_{j}(x)} \ \dx+c\int_{B_{R}^{+}\cap\{0\le \eta(\Phi(x))<1\}}\left[\snr{Du_{j}}^{p_{j}(x)}+\snr{Du_{0}}^{p_{j}(x)}\right] \ \dx\nonumber \\
&+c\int_{B_{R}^{+}\cap\{0\le \eta(\Phi(x))<1\}}\left| \ \frac{u_{j}-u_{0}}{R\theta} \ \right|^{p_{j}(x)} \ \dx=:c\left[\mbox{(I)}_{j}^{1}+\mbox{(I)}_{j}^{2}+\mbox{(I)}_{j}^{3}\right]
\end{flalign*}
where $c\equiv c(N,\m,\gamma_{1},\gamma_{2})$. Let us bound the three terms appearing on the right-hand side of the above inequality. By Lemma \ref{L0} (\emph{i.}) with $\varepsilon_{0}=\frac{\tilde{\sigma}}{2}$, \eqref{t242a}, \eqref{t24}, \eqref{bougehp}, \eqref{t236}, $\eqref{com2}_{3}$ and the absolute continuity of Lebesgue's integral we have
\begin{flalign*}
\mbox{(I)}_{j}^{1}\le &c\int_{B_{R}^{+}\cap\{0\le \eta(\Phi(x))<1\}}\left[\snr{D\tilde{u}_{0}}^{p_{j}(x)}-\snr{D\tilde{u}_{0}}^{p_{0}(x)}\right] \ \dx +c\int_{B_{R}^{+}\cap\{0\le \eta(\Phi(x))<1\}}\snr{D\tilde{u}_{0}}^{p_{0}} \ \dx\nonumber \\
\le &c\nr{p_{j}-p_{0}}_{L^{\infty}(B_{1}^{+})}\int_{B_{R}^{+}\cap \{0\le \eta(\Phi(x))<1\}}\snr{D\tilde{u}_{0}}^{q_{2}\left(1+\frac{\tilde{\sigma}}{2}\right)} \ \dx+o(\theta)\nonumber \\
\le &c\nr{p_{j}-p_{0}}_{L^{\infty}(B_{1}^{+})}\int_{B_{R}^{+}}\snr{Du_{0}}^{q_{2}\left(1+\frac{\tilde{\sigma}}{2}\right)} \ \dx+o(\theta)=o(j)+o(\theta),
\end{flalign*}
with $c\equiv c(\texttt{data}_{p(\cdot)})$. By \eqref{t24}, \eqref{t243}, \eqref{t244}, \eqref{t236} we get that 
\begin{flalign*}
\mbox{(I)}_{j}^{2}=o(\theta).
\end{flalign*}
Moreover, using \eqref{t245}, H\"older inequality and \eqref{t236} we have
\begin{flalign*}
\mbox{(I)}_{j}^{3}\le& \snr{B_{R}^{+}\cap\{0\le \eta(\Phi(x))<1\}}+\int_{B_{R}^{+}\cap\{0\le \eta(\Phi(x))<1\}}\left| \ \frac{u_{j}-u_{0}}{r\theta} \ \right|^{q_{2}} \ \dx\nonumber \\
\le &o(\theta)+(R\theta)^{-q_{2}}\snr{B_{R}^{+}\cap\{0\le \eta(\Phi(x))<1\}}^{1-\frac{q_{2}}{q_{2}\left(1+\frac{\tilde{\sigma}}{2}\right)}}\nr{u_{j}-u_{0}}^{q_{2}}_{L^{q_{2}\left(1+\frac{\tilde{\sigma}}{2}\right)(B_{R}^{+})}}\nonumber \\
\le &o(\theta)+(R\theta)^{-q_{2}}o(j),
\end{flalign*}
and, trivially,
\begin{flalign*}
\mbox{(II)}_{j}\le \mathcal{E}_{j}(\tilde{u}_{0},B_{R}^{+}).
\end{flalign*}
Finally, by $\eqref{com1}_{3}$, $\eqref{com2}_{3}$, \eqref{t242} and \eqref{t242a} we get
\begin{flalign*}
&\snr{\mathcal{E}_{j}(\tilde{u}_{0},B_{R}^{+})-\mathcal{E}_{0}(\tilde{u}_{0},B_{R}^{+})}\le \left[\nr{k_{j}-k_{0}}_{L^{\infty}(B_{1}^{+})}+\nr{p_{j}-p_{0}}_{L^{\infty}(B_{1}^{+})}\right]\left(1+\int_{B_{R}^{+}}\snr{Du_{0}}^{q_{2}\left(1+\frac{\tilde{\sigma}}{2}\right)} \ \dx\right)\nonumber \\
&\qquad\le c(\texttt{data}_{p(\cdot)})\left[\nr{k_{j}-k_{0}}_{L^{\infty}(B_{1}^{+})}+\nr{p_{j}-p_{0}}_{L^{\infty}(B_{1}^{+})}\right]=o(j).
\end{flalign*}
Plugging the content of all the previous estimates in \eqref{t246} we end up with
\begin{flalign*}
\mathcal{E}_{j}(u_{j},B_{R}^{+})\le \mathcal{E}_{0}(\tilde{u}_{0},B_{R}^{+})+o(j)+o(\theta)+(R\theta)^{-q_{2}}o(j).
\end{flalign*}
By \eqref{t29} we can take the liminf as $j\to \infty$ in the above display to obtain
\begin{flalign}\label{t.243}
\mathcal{E}_{0}(u_{0},B_{R}^{+})\le& \liminf_{j\to \infty}\mathcal{E}_{j}(u_{j},B_{R}^{+})\nonumber \\
\le &\limsup_{j\to \infty}\left[\mathcal{E}_{0}(\tilde{u}_{0},B_{R}^{+})+o(j)+o(\theta)+(R\theta)^{-q_{2}}o(j)\right]\nonumber \\
\le &\mathcal{E}_{0}(\tilde{u}_{0},B_{R}^{+})+o(\theta).
\end{flalign}
Sending $\theta \to 0$ in \eqref{t.243} and using the minimality of $\tilde{u}_{0}$ in class $\hat{\mathcal{C}}_{u_{0}}^{p_{0}}(B_{R}^{+},\m)$, we end up with
\begin{flalign*}
\mathcal{E}_{0}(u_{0},B_{R}^{+})\le \mathcal{E}_{0}(\tilde{u}_{0},B_{R}^{+})\le \mathcal{E}_{0}(w,B_{R}^{+})
\end{flalign*}
for all $w\in \hat{\mathcal{C}}_{u_{0}}^{p_{0}}(B_{R}^{+},\m)$, so, by Definition \ref{D1} and \eqref{t240}, the minimality of $u_{0}$ in class $\mathcal{C}^{p_{0}}_{g_{0}}(B_{R}^{+},\m)$ is proved. Finally, combining \eqref{t.243} with the minimality of $\tilde{u}_{0}$ in class $\hat{\mathcal{C}}_{u_{0}}^{p_{0}}(B_{R}^{+},\m)$, we can conclude that $\mathcal{E}_{j}(u_{j},B_{R}^{+})\to \mathcal{E}_{0}(u_{0},B_{R}^{+})$. 
\subsubsection*{Step 3. Singular points.} Let $\{x_{j}\}\subset \bar{B}_{1}^{+}$ be the sequence of singular points in the statement. The interior case $x_{0}\in B_{1}^{+}$ has already been analyzed in \cite[Section 4.1]{decv}, so we can assume that $x_{0}\in \Gamma_{1}$. Up to choose $j\in \N$ sufficiently large and then relabel, we can also suppose that $\{x_{j}\}\subset B_{R}^{+}$ for some $R\in \left(0,\frac{R_{*}}{2}\right)$, $x_{0}\in \Gamma_{R}$ and \eqref{t24}-\eqref{t24.a} are in force. By Theorem \ref{t1} and \eqref{t23}, we can find a radius $\tilde{R}>0$ and a positive constant $\tilde{\varepsilon}$, both independent on $j\in \N$ so that if $x_{j}$ is a singular point of $u_{j}$, then
\begin{flalign}\label{t.245}
\left(\rr^{p_{2,j}(\rr)-n}\int_{B_{\rr}^{+}(x_{j})}(1+\snr{Du_{j}}^{2})^{\frac{p_{2,j}(\rr)}{2}}\ \dx\right)^{\frac{1}{p_{2,j}(\rr)}}>\tilde{\varepsilon}>0 
\end{flalign}
for all $\rr\in \left(0,\frac{1}{2}\min\left\{\tilde{R},R_{*}-R\right\}\right)$, with $R_{*}$ as in \eqref{r*}. In the above display, we denoted $p_{2,j}(\rr):=\sup_{x\in B_{\rr}(x_{j})\cap B_{R}^{+}}p_{j}(x)$. Set $\sigma':=\min\left\{\tilde{\sigma},\frac{\alpha}{\gamma_{2}}\right\}$. By Lemma \ref{L0} \emph{(i.)} with $\varepsilon_{0}=\frac{\sigma'}{2}$ and $\eqref{ca1}_{2}$, we estimate
\begin{flalign}\label{t.244}
&\left|\ \rr^{p_{2,j}(\rr)-n}\int_{B_{\rr}^{+}(x_{j})}\left[(1+\snr{Du_{j}}^{2})^{\frac{p_{2,j}(\rr)}{2}}-(1+\snr{Du_{j}}^{2})^{\frac{p_{j}(x)}{2}}\right] \ \dx\ \right|^{\frac{1}{p_{2,j}(\rr)}}\nonumber \\
&\qquad \le c\rr^{1+\frac{\alpha}{\gamma_{1}}}\left(\mint_{B_{\rr}^{+}(x_{j})}(1+\snr{Du_{j}}^{2})^{\frac{p_{2,j}(\rr)}{2}\left(1+\frac{\sigma'}{2}\right)} \ \dx\right)^{\frac{1}{p_{2,j}(\rr)}}\le c\rr^{-\frac{\sigma'}{2}+\frac{\alpha}{\gamma_{2}}}\to 0,
\end{flalign}
for $c\equiv c(n,N, \m,\gamma_{1},\gamma_{2},q)$. By \eqref{t.245}, \eqref{t.244}, \eqref{t23} and \eqref{caccine} we then get
\begin{flalign}\label{t.246}
\tilde{\varepsilon}<&c\rr^{-\frac{\sigma'}{2}+\frac{\alpha}{\gamma_{2}}}+c\left(\rr^{p_{2,j}(\rr)-n}\int_{B_{\rr}^{+}(x_{j})}(1+\snr{Du_{j}}^{2})^{\frac{p_{j}(x)}{2}} \ \dx\right)^{\frac{1}{p_{2,j}(\rr)}}\nonumber \\
\le&c\rr^{-\frac{\sigma'}{2}+\frac{\alpha}{\gamma_{2}}}+c\rr+c\rr^{1-\frac{n}{p_{2,j}(\rr)}}\left[\int_{B_{2\rr}^{+}(x_{j})}\left|\frac{u_{j}-g_{j}}{\rr} \ \right|^{p_{j}(x)} \ \dx+\int_{B_{2\rr}^{+}(x_{j})}\snr{Dg_{j}}^{p_{j}(x)} \ \dx\right]^{\frac{1}{p_{2,j}(\rr)}}\nonumber \\
\le &c\rr^{-\frac{\sigma'}{2}+\frac{\alpha}{\gamma_{2}}}+c\rr+c\rr^{1-\frac{n}{q}}\left(\int_{B_{\rr}^{+}(x_{j})}\snr{Dg_{j}}^{q} \ \dx\right)^{\frac{1}{q}}\nonumber \\
&+c\rr^{1-\frac{n}{p_{2,j}(\rr)}}\left[\int_{B_{2\rr}^{+}(x_{j})}\left|\frac{u_{j}-u_{0}}{\rr} \ \right|^{q_{2}\left(1+\frac{\tilde{\sigma}}{2}\right)} \ \dx+\int_{B_{2\rr}^{+}(x_{j})}\left|\frac{g_{j}-u_{0}}{\rr} \ \right|^{p_{j}(x)} \ \dx\right]^{\frac{1}{p_{2,j}(\rr)}}\nonumber \\
\le &c\rr^{\sigma''}+c\left[\mint_{B_{2\rr}^{+}(x_{j})}\snr{u_{j}-u_{0}}^{q_{2}\left(1+\frac{\tilde{\sigma}}{2}\right)} \ \dx+\mint_{B_{2\rr}^{+}(x_{j})}\snr{g_{j}-u_{0}}^{p_{j}(x)} \ \dx\right]^{\frac{1}{p_{2,j}(\rr)}},
\end{flalign}
where we set $\sigma'':=\min\left\{1-\frac{n}{q},\frac{\alpha}{\gamma_{2}}-\frac{\sigma'}{2}\right\}$ and $c\equiv c(n,N,\m,\gamma_{1},\gamma_{2},q)$. By \eqref{t245} we get
\begin{flalign}\label{t.247}
\mint_{B_{2\rr}^{+}(x_{j})}\snr{u_{j}-u_{0}}^{q_{2}\left(1+\frac{\tilde{\sigma}}{2}\right)} \ \dx\to 0\quad \mbox{as} \ \ j\to \infty.
\end{flalign}
Since $g_{j}\rightharpoonup g_{0}$ weakly in $W^{1,q}(\bar{B}_{1}^{+},\m)$, then by Rellich-Kondrachov theorem there holds that, up to subsequences, $g_{j}\to g_{0}$ strongly in $L^{q}(\bar{B}_{1}^{+},\m)$ and pointwise a.e., therefore, keeping also $\eqref{com2}_{3}$ in mind, we can apply the dominated convergence theorem to end up with
\begin{flalign}\label{t.248}
\mint_{B_{2\rr}^{+}(x_{j})}\snr{g_{j}-u_{0}}^{p_{j}(x)} \ \dx \to \mint_{B_{2\rr}^{+}(x_{j})}\snr{g_{0}-u_{0}}^{p_{0}} \ \dx \quad \mbox{as} \ \ j\to \infty.
\end{flalign}
By $\eqref{com2}_{3}$, \eqref{t.247} and \eqref{t.248} we can take the limit superior on both sides of the inequality in \eqref{t.246} to obtain
\begin{flalign}\label{t248}
\tilde{\varepsilon}\le c\rr^{\sigma''}+c\left(\mint_{B_{2\rr}^{+}(x_{j})}\snr{g_{0}-u_{0}}^{p_{0}} \ \dx\right)^{\frac{1}{p_{0}}}.
\end{flalign}
We finally pass to the limit superior for $\rr\to \infty$ in \eqref{t248} and have
\begin{flalign*}
0<\bar{\varepsilon}^{p_{0}}\le \limsup_{\rr\to 0}\mint_{B_{2\rr}^{+}(x_{j})}\snr{u_{0}-g_{0}}^{p_{0}} \ \dx,
\end{flalign*}
meaning that $x_{0}$ is a singular point for $u_{0}$. In the previous display, we set $\bar{\varepsilon}:=\tilde{\varepsilon}/c$.
\end{proof}
The next lemma is a monotonicity formula in the spirit of \cite{decv,fu3,shouhlb,tac}.
\begin{lemma}\label{mono}
Under assumptions \eqref{ask}, $\eqref{aspinf}$, \eqref{m} and \eqref{g}, let $u\in W^{1,p(\cdot)}(B_{1}^{+},\m)$ be a solution of problem \eqref{pd1}. Suppose also that
\begin{flalign}\label{k01}
k(0)=1.
\end{flalign}
Then, there exist $\Upsilon\equiv \Upsilon(n,N,\m,\gamma_{1},\gamma_{2},q)\in (0,1]$ and a threshold $T\equiv T(\texttt{data},\kappa)\in (0,1]$ such that if
\begin{flalign}\label{monoreg}
[g]_{0, 1-\frac{n}{q};\bar{B}_{1}^{+}}<\Upsilon,
\end{flalign}
then for all $\kappa \in \left(0,1-\frac{n}{q}\right]$, the map $\Phi\colon \left(0,\frac{T}{4}\right)\to [0,\infty)$ defined as
\begin{flalign}\label{monophi}
 \Phi(\tau):=\exp\left(\frac{\tilde{c}}{\beta''}\tau^{\beta''}\right)\left[\tau^{p_{2}(\tau)-n}\int_{B_{\tau}^{+}}k(x)\snr{D\tilde{u}}^{p_{2}(\tau)} \ \dx+c\frac{\tau^{\kappa}}{\kappa}\right],
\end{flalign}
with $\tilde{u}$ as in \eqref{tiu}, $\beta''\equiv \beta''(n,q,\nu)$ and $c,\tilde{c}\equiv c,\tilde{c}(\texttt{data},\nr{Dg}_{L^{q}(B_{1}^{+})},\kappa)$, is monotone non-decreasing. Moreover, the following inequality holds true
\begin{flalign}\label{monophi+}
\int_{\partial B_{1}^{+}}&\snr{u(Rx)-u(\rr x)}^{p_{2}(\rr)} \ \d\mathcal{H}^{n-1}(x)\nonumber \\
\le& c\log(R/\rr)\left[\rr^{p_{2}(\rr)-p_{2}(R)}\left(\Phi(R)-\Phi(\rr)\right)\right]+c(R-\rr)^{\gamma_{1}\left(1-\frac{n}{q}\right)},
\end{flalign}
for $c\equiv c(\texttt{data},\nr{Dg}_{L^{q}(B_{1}^{+})},\kappa)$.
\end{lemma}
\begin{proof}
Let $u\in W^{1,p(\cdot)}(B_{1}^{+},\m)$ be a solution of problem \eqref{pd1}, $\kappa\in (0,1)$ be a fixed constant and select $T\in (0,1]$ so that
\begin{flalign*}
0<T\le \min\left\{R_{*},\frac{1-\kappa}{16[p]_{0,1}},\left(\frac{\lambda}{4[k]_{0,\nu}}\right)^{\frac{2}{\nu}}\right\},
\end{flalign*}
where $R_{*}$ is as in \eqref{r*}. Such a position assures that, whenever $\tau\in \left(0,\frac{T}{4}\right]$, \eqref{t11}-\eqref{t13} hold with $R$ replaced by $\tau$, moreover,
\begin{flalign}\label{mono5}
p_{2}(4\tau)-p_{1}(\tau)\le \frac{1-\kappa}{2}\quad \mbox{and}\quad 4[k]_{0,\nu}\tau^{\frac{\nu}{2}}\le \lambda,
\end{flalign}
with $\nu$ as in $\eqref{ask}_{1}$. For $\tau\in \left(0,\frac{T}{4}\right]$, we introduce the functional
\begin{flalign*}
W^{1,p_{2}(\tau)}(B_{\tau}^{+},\m)\ni w\mapsto \mathcal{E}_{\tau}(w,B_{\tau}^{+}):= \int_{B_{\tau}^{+}}k(x)\snr{Dw}^{p_{2}(\tau)} \ \dx
\end{flalign*}
and let $v\in W^{1,p_{2}(\tau)}(B_{\tau}^{+},\m)$ be a solution of problem
\begin{flalign}\label{monopd}
\hat{\mathcal{C}}_{u}^{p_{2}(\tau)}(B_{\tau}^{+},\m)\ni w \mapsto \min \mathcal{E}_{\tau}(w,B_{\tau}^{+}).
\end{flalign}
By the minimality of $v$ in class $\hat{\mathcal{C}}_{u}^{p_{2}(\tau)}(B_{\tau}^{+},\m)$ and that of $u$ in class $\mathcal{C}^{p(\cdot)}_{g}(B_{1}^{+},\m)$ we bound
\begin{flalign*}
&\snr{\mathcal{E}_{\tau}(u,B_{\tau}^{+})-\mathcal{E}_{\tau}(v,B_{\tau}^{+})}=\mathcal{E}_{\tau}(u,B_{\tau}^{+})-\mathcal{E}_{\tau}(v,B_{\tau}^{+})\nonumber \\
&\qquad \le \snr{\mathcal{E}_{\tau}(u,B_{\tau}^{+})-\mathcal{E}(u,B_{\tau}^{+})}+\snr{\mathcal{E}_{\tau}(v,B_{\tau}^{+})-\mathcal{E}(v,B_{\tau}^{+})}=:\mbox{(I)}+\mbox{(II)}.
\end{flalign*}
Let 
\begin{flalign}\label{s''}
\sigma'':=\frac{1}{4}\min\left\{\sigma_{g},\delta_{g},\frac{n-\gamma_{2}}{\gamma_{2}},\frac{1-\kappa}{2\gamma_{2}}\right\},
\end{flalign}
where $\sigma_{g}$ and $\delta_{g}$ are the higher integrability threshold from Lemmas \ref{geh}-\ref{gehp} respectively. Combining $\eqref{asp}_{1}$, $\eqref{ask}_{1}$, Lemma \ref{geh}, Lemma \ref{L0} (\emph{i.}) with $\varepsilon_{0}=\sigma''$ and $\eqref{ca1}_{2}$ we end up with
\begin{flalign*}
\mbox{(I)}\le c\tau\int_{B_{\tau}^{+}}(1+\snr{Du}^{2})^{\frac{p_{2}(\tau)(1+\sigma'')}{2}} \ \dx\le c\tau^{1+n-p_{2}(4\tau)(1+\sigma'')},
\end{flalign*}
with $c\equiv c(\texttt{data}_{p(\cdot)},\nr{Dg}_{L^{q}(B_{1}^{+})})$. In a totally similar way, using this time Lemma \ref{gehp}, $\eqref{ca1}_{2}$ and Lemma \ref{L0} (\emph{ii.}) with $\varepsilon_{0}=\sigma''$ we get
\begin{flalign*}
\mbox{(II)}\le& c\tau \int_{B_{\tau}^{+}}(1+\snr{Dv}^{2})^{\frac{p_{2}(\tau)(1+\sigma'')}{2}} \ \dx\nonumber \\
\le &c\tau\int_{B_{\tau}^{+}}(1+\snr{Dv}^{2})^{\frac{p_{2}(\tau)(1+\sigma'')}{2}} \ \dx\le c\tau^{1+n-p_{2}(4\tau)(1+\sigma'')},
\end{flalign*}
for $c\equiv c(\texttt{data}_{p(\cdot)},\nr{Dg}_{L^{q}(B_{1}^{+})})$. Merging the content of the previous displays we obtain
\begin{flalign}\label{mono6}
\mathcal{E}_{\tau}(u,B_{\tau}^{+})\le \mathcal{E}_{\tau}(v,B_{\tau}^{+})+c\tau^{1+n-p_{2}(4\tau)(1+\sigma'')}.
\end{flalign}
Now, for $\tau$ as above, define $x_{\tau}:=\tau\frac{x}{\snr{x}}$. As in \cite[Lemma 1.3]{shouhlb} we consider the following comparison map
\begin{flalign*}
w_{\tau}(x):=\begin{cases}
\ u(x)\quad &\mbox{if} \ \ x\in B_{1}^{+}\setminus B_{\tau}^{+}\\
\ \tilde{u}(x_{\tau})+g(x)\quad &\mbox{if} \ \ x\in B_{\tau}^{+},
\end{cases}
\end{flalign*}
where $\tilde{u}$ is defined in \eqref{tiu}. Notice that, by \eqref{t12}-\eqref{t13} there holds that
\begin{flalign}\label{mono1}
w_{\tau}\in u+ W^{1,p(\cdot)}_{0}(B_{\tau}^{+},\RN)\quad \mbox{and}\quad w_{\tau}\in W^{1,p_{2}(\tau)}(B_{\tau}^{+},\RN).
\end{flalign}
Moreover, since \eqref{proreg} and \eqref{monoreg} are in force, we see that
\begin{flalign*}
\dist(w_{\tau},\m)\le c(n,q,\beta_{0})\Upsilon \tau^{\beta_{0}}\quad \mbox{with} \ \  \beta_{0}:=1-\frac{n}{q},
\end{flalign*}
therefore, choosing $\Upsilon$ small enough, and thus determining the dependency $\Upsilon\equiv \Upsilon(n,N,\m,\gamma_{1},\gamma_{2},q)$, we can project $w_{\tau}$ onto $\m$ thus obtaining a map $\bar{w}_{\tau}:=\Pi_{\m}(w_{\tau})$ satisfying
\begin{flalign}\label{mono2}
\bar{w}_{\tau}\in \hat{\mathcal{C}}_{u}^{p_{2}(\tau)}(B_{\tau}^{+},\m)\quad \mbox{and}\quad \int_{B_{\tau}^{+}}\snr{D\bar{w}_{\tau}}^{p_{2}(\tau)} \ \dx \le (1+c\Upsilon\tau^{\beta_{0}})\int_{B_{\tau}^{+}}\snr{Dw_{\tau}}^{p_{2}(\tau)} \ \dx,
\end{flalign}
for $c\equiv c(n,N,\m,\gamma_{1},\gamma_{2},q)$. Notice that, by the mean value theorem applied to the function $[0,\infty)\ni s\mapsto (t+s)^{p_{2}(\tau)}$ there holds that
\begin{flalign}\label{mean}
\left(\snr{D\tilde{u}}+\snr{Dg}\right)^{p_{2}(\tau)}\le\snr{D\tilde{u}}^{p_{2}(\tau)}+p_{2}(\tau)(\snr{D\tilde{u}}+\snr{Dg})^{p_{2}(\tau)-1}\snr{Dg},
\end{flalign} 
so by H\"older inequality with conjugate exponents $\left(\frac{p_{2}(\tau)}{p_{2}(\tau)-1},p_{2}(\tau)\right)$, \eqref{monoreg} and \eqref{mono2}
we get
\begin{flalign}\label{mono2.1}
\int_{B_{\tau}^{+}}&\snr{Dw_{\tau}}^{p_{2}(\tau)} \ \dx\le \int_{B_{\tau}^{+}}\left(\snr{D\tilde{u}(x_{\tau})}+\snr{Dg}\right)^{p_{2}(\tau)} \ \dx\nonumber \\
\le &(1+c\tau^{\beta_{0}})\int_{B_{\tau}^{+}}\snr{D\tilde{u}(x_{\tau})}^{p_{2}(\tau)} \ \dx\nonumber \\
&+c\left[\tau^{-\beta_{0}(p_{2}(\tau)-1)}\int_{B_{\tau}^{+}}\snr{Dg}^{p_{2}(\tau)}  \ \dx+\int_{B_{\tau}^{+}}\snr{Dg}^{p_{2}(\tau)} \ \dx\right]\nonumber\\
\le &(1+c\tau^{\beta_{0}})\int_{B_{\tau}^{+}}\snr{D\tilde{u}(x_{\tau})}^{p_{2}(\tau)} \ \dx\nonumber \\
&+c\left[\tau^{-\beta_{0}(p_{2}(\tau)-1)+n\left(1-\frac{p_{2}(\tau)}{q}\right)}+\tau^{n\left(1-\frac{p_{2}(\tau)}{q}\right)}\right]\nr{Dg}_{L^{q}(B_{1}^{+})}\nonumber \\
\le &(1+c\tau^{\beta_{0}})\int_{B_{\tau}^{+}}\snr{D\tilde{u}(x_{\tau})}^{p_{2}(\tau)} \ \dx+c\tau^{n\left(1-\frac{1}{q}\right)+1-p_{2}(\tau)}
\end{flalign}
for $c\equiv c(n,\gamma_{1},\gamma_{2},q,\nr{Dg}_{L^{q}(B_{1}^{+})})$. In the previous expression, we also used the original value of $\beta_{0}$. By $\eqref{ask}$, \eqref{k01} and $\eqref{mono2}$ we can refine \eqref{mono2.1} as
\begin{flalign}\label{mono2.2}
\int_{B_{\tau}^{+}}&k(x)\snr{D\bar{w}_{\tau}}^{p_{2}(\tau)} \ \dx\le (1+4[k]_{0,\nu}\tau^{\nu}) \int_{B_{\tau}^{+}}\snr{D\bar{w}_{\tau}}^{p_{2}(\tau)} \ \dx\nonumber \\
\le &(1+c\tau^{\beta'})\int_{B_{\tau}^{+}}\snr{D\tilde{u}(x_{\tau})}^{p_{2}(\tau)} \ \dx +c\tau^{n\left(1-\frac{1}{q}\right)+1-p_{2}(\tau)},
\end{flalign}
where $\beta':=\min\{\beta_{0},\nu\}$ and $c\equiv c(n,N,\m,\gamma_{1},\gamma_{2},q,\nr{Dg}_{L^{q}(B_{1}^{+})})$. Let us evaluate the $p_{2}(\tau)$-energy of $\tilde{u}$. First, recall that if $\frac{\partial \tilde{u}}{\partial r}:=D\tilde{u}\cdot \frac{x}{\snr{x}}$ denotes the radial derivative of $\tilde{u}$, then
\begin{flalign}\label{rad}
\left| \ \frac{\partial\tilde{u}}{\partial r} \ \right|\le \snr{D\tilde{u}}.
\end{flalign}
Moreover, if $p_{2}(\tau)\ge 2$ and $t\ge s\ge 0$ there holds that
\begin{flalign}\label{rad.1}
(t-s)^{p_{2}(\tau)}\le t^{p_{2}(\tau)}-s^{p_{2}(\tau)}.
\end{flalign}
A straightforward computation renders, for $x\in B_{\tau}^{+}$ that
\begin{flalign*}
\snr{D\tilde{u}(x_{\tau})}^{2}=\frac{\tau^{2}}{\snr{x}^{2}}\left[\snr{D\tilde{u}(x_{\tau})}^{2}-\left| \ D\tilde{u}(x_{\tau})\cdot \frac{x_{\tau}}{\snr{x_{\tau}}} \ \right|^{2}\right],
\end{flalign*}
so by \eqref{rad}, \eqref{rad.1}, area formula, $\eqref{ask}$, \eqref{k01} and $\eqref{mono5}_{2}$ we obtain
\begin{flalign}\label{mono3}
\int_{B_{\tau}^{+}}&\snr{D\tilde{u}(x_{\tau})}^{p_{2}(\tau)} \ \dx=\frac{\tau}{n-p_{2}(\tau)}\int_{\partial B_{\tau}^{+}}\left[\snr{D\tilde{u}(x)}^{2}-\left|\ \frac{\partial \tilde{u}}{\partial r} \ \right|^{2}\right]^{\frac{p_{2}(\tau)}{2}} \ \d \mathcal{H}^{n-1}(x)\nonumber \\
\le &\frac{\tau}{n-p_{2}(\tau)}\left[\int_{\partial B_{\tau}^{+}}\snr{D\tilde{u}(x)}^{p_{2}(\tau)}-\left|\ \frac{\partial \tilde{u}}{\partial r} \ \right|^{p_{2}(\tau)}\right] \ \d \mathcal{H}^{n-1}(x)\nonumber \\
\le &\frac{\tau}{n-p_{2}(\tau)}\left[(1+\tau^{\frac{\nu}{2}})\int_{\partial B_{\tau}^{+}}k(x)\snr{D\tilde{u}(x)}^{p_{2}(\tau)}\d\mathcal{H}^{n-1}(x)-\int_{\partial B_{\tau}^{+}}\left|\ \frac{\partial \tilde{u}}{\partial r} \ \right|^{p_{2}(\tau)} \ \d\mathcal{H}^{n-1}(x)\right].
\end{flalign}
Recalling the position made in \eqref{tiu}, by \eqref{mean}, H\"older inequality with conjugate exponents $\left(\frac{p_{2}(\tau)}{p_{2}(\tau)-1},p_{2}(\tau)\right)$, \eqref{monoreg}, $\eqref{ask}$, \eqref{mono6} and the minimality of $v$ in class $\hat{\mathcal{C}}_{u}^{p_{2}(\tau)}(B_{\tau}^{+},m)$ with $\eqref{mono2}_{1}$ we have
\begin{flalign}\label{mean1.1}
\mathcal{E}_{\tau}(\tilde{u},B_{\tau}^{+})\le&(1+c\tau^{\beta_{0}})\mathcal{E}_{\tau}(u,B_{\tau}^{+})+c\tau^{-\beta_{0}(p_{2}(\tau)-1)}\int_{B_{\tau}^{+}}\snr{Dg}^{p_{2}(\tau)} \ \dx\nonumber \\
\le &(1+c\tau^{\beta_{0}})\mathcal{E}_{\tau}(v,B_{\tau}^{+})+c\left[\tau^{n\left(1-\frac{1}{q}\right)+1-p_{2}(\tau)}+\tau^{n+1-p_{2}(4\tau)(1+\sigma'')}\right]\nonumber \\
\le &(1+c\tau^{\beta_{0}})\mathcal{E}_{\tau}(\bar{w}_{\tau},B_{\tau}^{+})+c\left[\tau^{n\left(1-\frac{1}{q}\right)+1-p_{2}(\tau)}+\tau^{n+1-p_{2}(4\tau)(1+\sigma'')}\right]\nonumber \\
\le &(1+c\tau^{\beta'})\int_{B_{\tau}^{+}}\snr{D\tilde{u}(x_{\tau})}^{p_{2}(\tau)} \ \dx+c\left[\tau^{n\left(1-\frac{1}{q}\right)+1-p_{2}(\tau)}+\tau^{n+1-p_{2}(4\tau)(1+\sigma'')}\right].
\end{flalign}
with $c(\texttt{data}_{p(\cdot)},\nr{Dg}_{L^{q}(B_{1}^{+})})$. Merging \eqref{mean1.1} with \eqref{mono3} and using \eqref{rad}, \eqref{ask} and $\eqref{mono5}_{2}$ we obtain
\begin{flalign}\label{mono9.1}
\mathcal{E}_{\tau}(\tilde{u},B_{\tau}^{+})\le&\frac{\tau}{n-p_{2}(\tau)}\left[(1+c\tau^{\beta''})\int_{\partial B_{\tau}^{+}}k(x)\snr{D\tilde{u}}^{p_{2}(\tau)} \ \d\mathcal{H}^{n-1}(x)\right.\nonumber \\
&-\left.\int_{\partial B_{\tau}^{+}}\left|\ \frac{\partial \tilde{u}}{\partial r}\ \right|^{p_{2}(\tau)} \ \d\mathcal{H}^{n-1}(x)+c\tau^{\beta'}(\tau^{\frac{\nu}{2}}+1)\int_{\partial B_{\tau}^{+}}k(x)\snr{D\tilde{u}}^{p_{2}(\tau)} \ \d\mathcal{H}^{n-1}(x)\right]\nonumber \\
&+c\left[\tau^{n\left(1-\frac{1}{q}\right)+1-p_{2}(\tau)}+\tau^{n+1-p_{2}(4\tau)(1+\sigma'')}\right]\nonumber \\
\le &\frac{\tau(1+c\tau^{\beta''})}{n-p_{2}(\tau)}\int_{\partial B_{\tau}^{+}}k(x)\snr{D\tilde{u}}^{p_{2}(\tau)} \ \d\mathcal{H}^{n-1}(x)\nonumber \\
&-\frac{\tau}{n-p_{2}(\tau)}
\int_{\partial B_{\tau}^{+}}\left|\ \frac{\partial \tilde{u}}{\partial r} \ \right|^{p_{2}(\tau)} \ \d\mathcal{H}^{n-1}(x)+c\left[\tau^{n\left(1-\frac{1}{q}\right)+1-p_{2}(\tau)}+\tau^{n+1-p_{2}(4\tau)(1+\sigma'')}\right],
\end{flalign}
with $\beta'':=\min\left\{\frac{\nu}{2},\beta'\right\}$ and $c\equiv c(\texttt{data},\nr{Dg}_{L^{q}(B_{1}^{+})})$. To summarize, we got
\begin{flalign}\label{mono9}
\tau\int_{\partial B_{\tau}^{+}}&k(x)\snr{D\tilde{u}}^{p_{2}(\tau)} \ \d\mathcal{H}^{n-1}(x)\ge \frac{n-p_{2}(\tau)}{1+c\tau^{\beta''}}\int_{B_{\tau}^{+}}k(x)\snr{D\tilde{u}}^{p_{2}(\tau)} \ \dx\nonumber \\
&+\frac{\tau}{1+c\tau^{\beta''}}\int_{\partial B_{\tau}^{+}}\left| \ \frac{\partial \tilde{u}}{\partial r} \ \right|^{p_{2}(\tau)} \ \d\mathcal{H}^{n-1}(x)\nonumber \\
&-\frac{c(n-p_{2}(\tau))}{1+c\tau^{\beta''}}\left[\tau^{n\left(1-\frac{1}{q}\right)+1-p_{2}(\tau)}+\tau^{n+1-p_{2}(4\tau)(1+\sigma'')}\right]
\end{flalign}
for $c\equiv c(\texttt{data},\nr{Dg}_{L^{q}(B_{1}^{+})},\beta_{0})$. Now, set 
\begin{flalign}\label{monoff}
\left(0,\frac{T}{4}\right)\ni \tau\mapsto \mathcal{f}(\tau):=\tau^{p_{2}(\tau)-n}\int_{B_{\tau}^{+}}k(x)\snr{D\tilde{u}}^{p_{2}(\tau)} \ \dx.
\end{flalign}
Multiplying both sides of \eqref{mono9} by $\tau^{p_{2}(\tau)-n-1}$ and using $\eqref{mono5}_{1}$ and \eqref{s''} we obtain
\begin{flalign}\label{mono10}
\tau^{p_{2}(\tau)-n}\int_{\partial B_{\tau}^{+}}&k(x)\snr{D\tilde{u}}^{p_{2}(\tau)} \ \d\mathcal{H}^{n-1}(x)\ge \frac{n-p_{2}(\tau)}{1+c\tau^{\beta''}}\left[\tau^{-1}\mathcal{f}(\tau)-c\left(\tau^{\kappa-1}+\tau^{-\frac{n}{q}}\right)\right]\nonumber \\
&+\frac{\tau^{p_{2}(\tau)-n}}{1+c\tau^{\beta''}}
\int_{\partial B_{\tau}^{+}}\left|\ \frac{\partial \tilde{u}}{\partial r} \ \right|^{p_{2}(\tau)} \ \d\mathcal{H}^{n-1}(x)
\end{flalign}
for $c\equiv c(\texttt{data},\nr{Dg}_{L^{q}(B_{1}^{+})},\kappa)$. From \eqref{aspinf} follows that $\left(0,\frac{T}{4}\right)\ni\tau\mapsto p_{2}(\tau)$ is differentiable with bounded, non-negative first derivative $0\le p'(\tau)\le c(n,[p]_{0,1})$. We compute:
\begin{flalign*}
\mathcal{f}'(\tau)=&(p_{2}(\tau)-n)\tau^{p_{2}(\tau)-n-1}\int_{B_{\tau}^{+}}k(x)\snr{D\tilde{u}}^{p_{2}(\tau)} \ \dx\nonumber \\
&+\tau^{p_{2}(\tau)-n}\int_{\partial B_{\tau}^{+}}k(x)\snr{D\tilde{u}}^{p_{2}(\tau)} \ \d\mathcal{H}^{n-1}(x)\nonumber \\
&+p_{2}'(\tau)\log(\tau)\tau^{p_{2}(\tau)-n}\int_{B_{\tau}^{+}}k(x)\snr{D\tilde{u}}^{p_{2}(\tau)} \ \dx \nonumber \\
&+p_{2}'(\tau)\tau^{p_{2}(\tau)-n}\int_{B_{\tau}^{+}}k(x)\log(\snr{D\tilde{u}})\snr{D\tilde{u}}^{p_{2}(\tau)} \ \dx.
\end{flalign*}
We record that, for all $\varepsilon_{0}\in (0,1)$ there holds that
\begin{flalign}\label{log}
\snr{\log(t)}\le c(\varepsilon_{0})(1+t)t^{-\varepsilon_{0}}\quad \mbox{for any} \ \ t>0.
\end{flalign}
Let us estimate the last two terms appearing in the expansion of $\mathcal{f}'(\tau)$. Using \eqref{log} with $\varepsilon_{0}=1-\beta''$ we bound
\begin{flalign*}
p_{2}'(\tau)&\log(\tau)\tau^{p_{2}(\tau)-n}\int_{B_{\tau}^{+}}k(x)\snr{D\tilde{u}}^{p_{2}(\tau)} \ \dx\le cp_{2}'(\tau)\tau^{\beta''-1+p_{2}(\tau)-n}\int_{B_{\tau}^{+}}k(x)\snr{D\tilde{u}}^{p_{2}(\tau)} \ \dx.
\end{flalign*}
By \eqref{log} with $\varepsilon_{0}=1-p_{2}(\tau)\min\left\{\frac{\sigma_{0}}{2},\frac{1-\kappa}{2\gamma_{2}}\right\}$, (keep \eqref{t10}-\eqref{t12} in mind) and $\eqref{ca1}_{2}$ we obtain
\begin{flalign*}
p_{2}'(\tau)&\tau^{p_{2}(\tau)-n}\int_{B_{\tau}^{+}}k(x)\log(\snr{D\tilde{u}})\snr{D\tilde{u}}^{p_{2}(\tau)} \ \dx\le c\tau^{p_{2}(\tau)}\mint_{B_{\tau}^{+}}(1+\snr{D\tilde{u}})^{p_{2}(\tau)(1+\varepsilon_{0})} \ \dx \nonumber \\
\le&c\tau^{p_{2}(\tau)-p_{2}(4\tau)(1+\varepsilon_{0})}\le c\tau^{\kappa-1},
\end{flalign*}
for $c\equiv c(\texttt{data}_{p(\cdot)},\nr{Dg}_{L^{1}(B_{1}^{+})},\kappa)$. All in all, we got the following lower bound for $\mathcal{f}'(\tau)$:
\begin{flalign}\label{mono11}
\mathcal{f}'(\tau)\ge& (p_{2}(\tau)-n)\tau^{p_{2}(\tau)-n-1}\int_{B_{\tau}^{+}}k(x)\snr{D\tilde{u}}^{p_{2}(\tau)} \ \dx\nonumber \\
&+\tau^{p_{2}(\tau)-n}\int_{\partial B_{\tau}^{+}}k(x)\snr{D\tilde{u}}^{p_{2}(\tau)} \ \d\mathcal{H}^{n-1}(x)\nonumber \\
&-cp'_{2}(\tau)\tau^{\beta''-1+p_{2}(\tau)-n}\int_{B_{\tau}^{+}}k(x)\snr{D\tilde{u}}^{p_{2}(\tau)} \ \dx-c\tau^{\kappa-1}\nonumber \\
=&\tau^{p_{2}(\tau)-n}\int_{\partial B_{\tau}^{+}}k(x)\snr{D\tilde{u}}^{p_{2}(\tau)} \ \d\mathcal{H}^{n-1}(x)\nonumber \\
&+\left(p_{2}(\tau)-n-cp_{2}'(\tau)\tau^{\beta''}\right)\frac{\mathcal{f}(\tau)}{\tau}-c\tau^{\kappa-1},
\end{flalign}
with $c\equiv c(\texttt{data},\nr{Dg}_{L^{q}(B_{1}^{+})},\kappa)$. Set $$\varphi(\tau):=n-p_{2}(\tau)+cp_{2}'(\tau)\tau^{\beta''}.$$
Merging \eqref{mono10} and \eqref{mono11} we obtain
\begin{flalign*}
\mathcal{f}'(\tau)&+\left(\varphi(\tau)-\frac{n-p_{2}(\tau)}{1+c\tau^{\beta''}}\right)\frac{\mathcal{f}(\tau)}{\tau}\ge \frac{\tau^{p_{2}(\tau)-n}}{1+c\tau^{\beta''}}
\int_{\partial B_{\tau}^{+}}\left|\ \frac{\partial \tilde{u}}{\partial r} \ \right|^{p_{2}(\tau)} \ \d\mathcal{H}^{n-1}(x) \nonumber \\
&-c\tau^{\kappa-1}\left[\frac{2(n-p_{2}(\tau))}{1+c\tau^{\beta''}}+1\right],
\end{flalign*}
where $c\equiv c(\texttt{data},\nr{Dg}_{L^{q}(B_{1}^{+})},\kappa)$. In the previous display we also used that $\kappa\le \beta_{0}$. It is easy to see that
\begin{flalign*}
\left| \ \varphi(\tau)-\frac{n-p_{2}(\tau)}{1+c\tau^{\beta''}} \ \right|\le \tilde{c}(\texttt{data},\nr{Dg}_{L^{q}(B_{1}^{+})},\kappa)\tau^{\beta''},
\end{flalign*}
therefore we get
\begin{flalign}\label{mono12}
\mathcal{f}'(\tau)+\tilde{c}\tau^{\beta''-1}\mathcal{f}(\tau)+c\tau^{\kappa-1}\ge \frac{\tau^{p_{2}(\tau)-n}}{1+c\tau^{\beta''}}
\int_{\partial B_{\tau}^{+}}\left|\ \frac{\partial \tilde{u}}{\partial r} \ \right|^{p_{2}(\tau)} \ \d\mathcal{H}^{n-1}(x).
\end{flalign}
Let $\Phi(\cdot)$ be the function defined in \eqref{monophi}. Combining \eqref{mono12} and the fact that $\tau\in (0,1]$, we immediately see that
\begin{flalign*}
\Phi'(\tau)\ge& \exp\left\{\frac{\tilde{c}\tau^{\beta''}}{\beta''}\right\}\left[\tilde{c}\tau^{\beta''-1}\mathcal{f}(\tau)+\mathcal{f}'(\tau)+c\tau^{\kappa-1}\right]\nonumber \\
\ge& \frac{\tau^{p_{2}(\tau)-n}}{1+c}
\int_{\partial B_{\tau}^{+}}\left|\ \frac{\partial \tilde{u}}{\partial r} \ \right|^{p_{2}(\tau)} \ \d\mathcal{H}^{n-1}(x),
\end{flalign*}
with $c\equiv c(\texttt{data},\nr{Dg}_{L^{q}(B_{1}^{+})},\kappa)$. At this stage, we integrate the inequality in the previous display over $\tau \in (\rr,R)$ with $0<\rr<R\le T\le 1$ to get
\begin{flalign}\label{mono13}
\Phi(R)-\Phi(\rr)\ge& \frac{1}{1+c}\int_{\rr}^{R}\tau^{p_{2}(\tau)-n}\left(\int_{\partial B_{\tau}^{+}}\left| \ \frac{\partial \tilde{u}}{\partial r} \ \right|^{p_{2}(\tau)} \ \d\mathcal{H}^{n-1}(x)\right) \ \d \tau\nonumber \\
\ge &\frac{\rr^{p_{2}(R)-p_{2}(\rr)}}{1+c}\int_{\rr}^{R}\tau^{p_{2}(\rr)-n}\left(\int_{\partial B_{\tau}^{+}}\left| \ \frac{\partial \tilde{u}}{\partial r} \ \right|^{p_{2}(\tau)} \ \d\mathcal{H}^{n-1}(x)\right) \ \d \tau.
\end{flalign}
Once \eqref{mono13} is available, we can proceed exactly as in \cite[Lemma 4.1]{tac} to end up with
\begin{flalign}\label{mono14}
\int_{\partial B_{1}^{+}}&\snr{\tilde{u}(Rx)-\tilde{u}(\rr x)}^{p_{2}(\rr)} \ \d \mathcal{H}^{n-1}(x)\nonumber \\
\le& \log(R/\rr)\int_{\rr}^{R}\tau^{p_{2}(\rr)-n}\left(\int_{\partial B_{\tau}^{+}}\left| \ \frac{\partial \tilde{u}}{\partial r} \ \right|^{p_{2}(\rr)} \ \d\mathcal{H}^{n-1}(x)\right) \ \d \tau\nonumber \\
\le &\frac{1+c}{\rr^{p_{2}(R)-p_{2}(\rr)}}\log(R/\rr)\left[\Phi(R)-\Phi(\rr)\right],
\end{flalign}
for $c\equiv c(\texttt{data},\nr{Dg}_{L^{q}(B_{1}^{+})},\kappa)$. Finally, keeping in mind \eqref{g} and position \eqref{tiu} we bound via \eqref{mono14}:
\begin{flalign*}
\int_{\partial B_{1}^{+}}&\snr{u(Rx)-u(\rr x)}^{p_{2}(\rr)} \ \d\mathcal{H}^{n-1}(x)\le c\int_{\partial B_{1}^{+}}\snr{\tilde{u}(Rx)-\tilde{u}(\rr x)}^{p_{2}(\rr)} \ \d\mathcal{H}^{n-1}(x)\nonumber \\
&+c\int_{\partial B_{1}^{+}}\snr{g(Rx)-g(\rr x)}^{p_{2}(\rr)} \ \d\mathcal{H}^{n-1}(x)\nonumber \\
\le &c\log(R/\rr)\left[\rr^{p_{2}(\rr)-p_{2}(R)}\left(\Phi(R)-\Phi(\rr)\right)\right]+c(R-\rr)^{\gamma_{1}\left(1-\frac{n}{q}\right)},
\end{flalign*}
with $c\equiv c(\texttt{data},\nr{Dg}_{L^{q}(B_{1}^{+})},\kappa)$ and the proof is complete.
\end{proof}
Before going on, let us stress that, as in Section \ref{partial}, we can reduce problem \eqref{cvp1.1} to an equivalent one defined on the half-ball $B_{1}^{+}$. In fact, in the proof of Theorem \ref{t2} we shall consider $u\in W^{1,p(\cdot)}(B_{1}^{+},\m)$ solution to
\begin{flalign}\label{pdfull}
\mathcal{C}_{g}^{p(\cdot)}(B_{1}^{+},\m)\ni w \mapsto \int_{B_{1}^{+}}\snr{Dw}^{p(x)} \ \dx,
\end{flalign}
with boundary datum $g(\cdot)$ as in \eqref{g} (of course $\bar{\Omega}$ is replaced by $\bar{B}_{1}^{+}$). Now we are ready to prove Theorem \ref{t2}.

\subsection{Proof of Theorem \ref{t2}}
As a consequence of Theorem \ref{t1}, we know that $u\in C^{0,\beta_{0}}_{loc}(\bar{B}_{1}^{+}\setminus \Sigma_{0}(u),\m)$, for a closed, negligible set $\Sigma_{0}\subset \bar{B}_{1}^{+}$. Let us prove that $\Sigma_{0}\cap \partial B_{1}^{+}=\emptyset$. By contradiction, assume that $x_{0}\in \Gamma_{1}$ is a singular point for $u\in W^{1,p(\cdot)}(B_{1}^{+},\m)$, solution to \eqref{pdfull}. Up to translations, there is no loss of generality in assuming $x_{0}=0$. Now, for $j \in \N$, define the rescaled maps
\begin{flalign*}
&u_{j}(x):=u(x/j),\qquad p_{j}(x):=p(x/j),\qquad k_{j}(x):=j^{p_{j}(x)-p(0)},\qquad g_{j}(x):=g(x/j).
\end{flalign*}
Since $u\in W^{1,p(\cdot)}(B_{1}^{+},\m)$ solves \eqref{pdfull}, we deduce that each $u_{j}\in W^{1,p_{j}(\cdot)}(B_{j}^{+},\m)$ solves problem
\begin{flalign}\label{pdj}
\mathcal{C}^{p_{j}(\cdot)}_{g_{j}}(B_{j}^{+},\m)\ni w\mapsto \min \int_{B_{j}^{+}}k_{j}(x)\snr{Dw}^{p_{j}(x)} \ \dx,
\end{flalign}
therefore it is easy to see that it also solves
\begin{flalign}\label{pdj1}
\mathcal{C}^{p_{j}(\cdot)}_{g_{j}}(B_{1}^{+},\m)\ni w\mapsto \min \int_{B_{1}^{+}}k_{j}(x)\snr{Dw}^{p_{j}(x)} \ \dx.
\end{flalign}
Notice that, whenever $x\in \bar{B}_{1}^{+}$, a straightforward computation shows that 
\begin{flalign}\label{f1}
\{p_{j}\}\ \mbox{and}\ \{k_{j}\}\ \mbox{are Lipschitz continuous uniformly in} \ j\in \N \ \mbox{in} \ \bar{B}_{1}^{+}. 
\end{flalign}
Again for $x\in \bar{B}_{1}^{+}$, recalling Morrey's embedding theorem we see that
\begin{flalign}
&\sup_{x\in \bar{B}_{1}^{+}}\snr{p_{j}(x)-p(0)}\le 4[p]_{0,1}\snr{x/j}\le 4[p]_{0,1}(1/j)\to 0 \label{f2}\\
&\sup_{x\in \bar{B}_{1}^{+}}\snr{k_{j}(x)-1}\le \max\left\{\exp\left(\frac{4[p]_{0,1}\log(j)}{j}\right)-1,1-\exp\left(\frac{-4[p]_{0,1}\log(j)}{j}\right)\right\} \to 0 \label{f3}\\
&\sup_{x\in \bar{B}_{1}^{+}}\snr{g_{j}(x)-g(0)}\le 4[g]_{0,1-\frac{n}{q}}\snr{x/j}^{1-\frac{n}{q}}\le4[g]_{0,1-\frac{n}{q}}(1/j)^{1-\frac{n}{q}} \to 0\label{f4}.
\end{flalign}
Furthermore, recalling \eqref{g} and \eqref{f4} we see that
\begin{flalign*}
\int_{B_{1}^{+}}\snr{Dg_{j}}^{q} \ \dx \le j^{q-n}\nr{Dg}_{L^{q}(B_{1}^{+})}^{q} \ \dx\to 0,
\end{flalign*}
so 
\begin{flalign}\label{ff4.1}
g_{j}\to g(0)\quad \mbox{in} \ \ W^{1,q}(\bar{B}_{1}^{+},\m).
\end{flalign}
Collecting \eqref{pdj} and \eqref{f1}-\eqref{ff4.1} we see that the assumptions of Lemma \ref{comp} are satisfied in $B_{1}^{+}$, so, in particular $u_{j}\rightharpoonup u_{0}$ weakly in $W^{1,(1+\tilde{\sigma}p(0))}_{loc}(B_{1}^{+},\m)$, $u_{0}$ is a solution of problem
\begin{flalign}\label{pdg0}
\mathcal{C}^{p(0)}_{g(0)}(B_{R}^{+},\m)\ni w\mapsto \min \int_{B_{R}^{+}}\snr{Dw}^{p(0)} \ \dx
\end{flalign}
for any $R\in (0,1)$ and, since $x_{0}=0$ is a singular point of all the $u_{j}$'s, then it is also a singular point for $u_{0}$. We fix $0<\mu_{1}<\mu_{2}<1$ and let $j\in \N$ be so large that $j^{-1}<\frac{T}{4}$ with $T$ as in Lemma \ref{mono}. Recalling also \eqref{t2.1} (on $\bar{B}_{1}^{+}$ of course), we see that the assumptions of Lemma \ref{mono} are satisfied, we can apply \eqref{monophi+} with $\rr=\mu_{1}/j$ and $R=\mu_{2}/j$ to get
\begin{flalign}\label{f5}
\int_{\partial B_{1}^{+}}&\snr{u_{j}(\mu_{1}x)-u_{j}(\mu_{2}x)}^{p_{2}(\mu_{1}/j)} \ \d\mathcal{H}^{n-1}(x)\nonumber\\
=&\int_{\partial B_{1}^{+}}\snr{u(j^{-1}\mu_{1}x)-u(j^{-1}\mu_{2}x)}^{p_{2}(\mu_{1}/j)} \ \d\mathcal{H}^{n-1}(x)\nonumber \\
\le & c\log(\mu_{2}/\mu_{1})\left(\Phi(\mu_{2}/j)-\Phi(\mu_{1}/j)\right)+cj^{-\gamma_{1}\left(1-\frac{n}{q}\right)}(\mu_{2}-\mu_{1})^{\gamma_{1}\left(1-\frac{n}{q}\right)},
\end{flalign}
with $\Phi(\cdot)$ defined as in \eqref{monophi} with $k(\cdot)\equiv 1$. By Lemma \ref{mono}, we deduce that
\begin{flalign*}
\lim_{j\to \infty}\Phi(\mu_{1}/j)=\lim_{j\to \infty}\Phi(\mu_{2}/j)=L \ \ \mbox{for some finite} \ \ L\ge 0,
\end{flalign*}
thus
\begin{flalign}\label{f6}
c\log(\mu_{2}/\mu_{1})\left(\Phi(\mu_{2}/j)-\Phi(\mu_{1}/j)\right)+cj^{-\gamma_{1}\left(1-\frac{n}{q}\right)}(\mu_{2}-\mu_{1})^{\gamma_{1}\left(1-\frac{n}{q}\right)}\to 0.
\end{flalign}
Furthermore, in light of \eqref{com3} we have that $u_{j}\to u_{0}$ almost everywhere in $B_{1}^{+}$, so, recalling also \eqref{f2} we get
\begin{flalign}\label{f7}
\snr{u_{j}(\mu_{2}x)-u_{j}(\mu_{1}x)}^{p_{2}(\mu_{1}/j)}\to \snr{u_{0}(\mu_{2}x)-u_{0}(\mu_{1}x)}^{p(0)} \quad \mbox{for a.e.} \ \ x\in B_{1}^{+}.
\end{flalign}
Combining $\eqref{f7}$, $\eqref{m}_{1}$ and the dominated convergence theorem, we obtain
\begin{flalign}\label{f8}
\lim_{j\to \infty}\int_{\partial B_{1}^{+}}&\snr{u_{j}(\mu_{2}x)-u_{j}(\mu_{1}x)}^{p_{2}(\mu_{1}/j)} \ \d\mathcal{H}^{n-1}(x)\nonumber \\
=&\int_{\partial B_{1}^{+}}\snr{u_{0}(\mu_{2}x)-u_{0}(\mu_{1}x)}^{p(0)} \ \d\mathcal{H}^{n-1}(x).
\end{flalign}
Inserting \eqref{f8} and \eqref{f6} in \eqref{f5}, we end up with
\begin{flalign*}
\int_{\partial B_{1}^{+}}\snr{u_{0}(\mu_{2}x)-u_{0}(\mu_{1}x)}^{p(0)} \ \d\mathcal{H}^{n-1}(x)=0,
\end{flalign*}
which in turn implies that $u_{0}$ is homogeneous of degree zero. Recalling that $u_{0}$ is a solution of \eqref{pdg0}, by \cite[Theorem 5.7]{harlin} we can conclude that $u_{0}$ is constant, so $x_{0}=0$ cannot be a singular point. This means that $\Sigma_{0}\Subset B_{1}^{+}$ and the proof is complete.


\begin{thebibliography}{}
\bibitem{acmi} E. Acerbi,  G. Mingione, Regularity results for stationary electro-rheological fluids. \emph{Arch. Ration. Mech. Anal.} 164, no. 3, 213-259, (2002).
\bibitem{ba} P. Baroni, Riesz potential estimates for a general class of quasilinear equations. \emph{Calc. Var. \& PDE} 53(3-4),12, pp. 803-846, (2015).
\bibitem{be} L. Beck, Partial regularity for weak solutions of nonlinear elliptic systems: The subquadratic case. \emph{Manuscripta Math.} 123, (4), 453-491, (2007).
\bibitem{bemi} L. Beck, G. Mingione, Lipschitz bounds and non-uniform ellipticity. \emph{Comm. Pure Appl. Math.}, (2020). \url{https://doi.org/10.1002/cpa.21880}
\bibitem{besc} P. Bella, M. Schaffner, Local boundedness and Harnack inequality for solutions of linear non-uniformly elliptic equations. \emph{Comm. Pure Appl. Math.}, to appear.
\bibitem{besc1} P. Bella, M. Schaffner, On the regularity of scalar integral functionals with $(p,q)$-growth. \emph{Anal. PDE}, to appear.
\bibitem{o2} I.~Chlebicka, A pocket guide to nonlinear differential equations in {M}usielak--{O}rlicz spaces, \emph{Nonl. Analysis}, 175, 1-27, (2018).
\bibitem{cosmin} A. Coscia, G. Mingione, H\"older continuity of the gradient of $p(x)$-harmonic mappings. \emph{C. R. Acad. Sci. Paris S\'er. I Math} 328 (4), 363-368, (1999).
\bibitem{decv} C. De Filippis, Partial regularity for manifold constrained $p(x)$-harmonic maps. \emph{Calc. Var. \& PDE} 58:47, (2019). 
\bibitem{demi1} C. De Filippis, G. Mingione, Manifold constrained non-uniformly elliptic problems. \emph{J. Geom. Analysis}, 1-63, (2019). \url{https://doi.org/10.1007/s12220-019-00275-3}
\bibitem{demi2} C. De Filippis, G. Mingione, On the Regularity of Minima of Non-autonomous Functionals. \emph{J. Geom. Analysis}, 1-43, (2019). \url{https://doi.org/10.1007/s12220-019-00225-z}
\bibitem{dhroblemsr} L. Diening, P. H\"ast\"o, M. Ru$\mbox{\v{z}}$i$\mbox{\v{c}}$ka, Lebesgue and Sobolev spaces with variable exponents. \emph{Lecture Notes in Mathematics} 2017, Springer-Verlag, Berlin, (2011).
\bibitem{dugrkr} F. Duzaar, J. F. Grotowski, M. Kronz, Partial and Full Boundary Regularity for Minimizers of Functionals with Nonquadratic Growth. \emph{J. Convex Analysis} 11, 2, 437–476, (2004).
\bibitem{dukrmi} F. Duzaar, J. Kristensen, G. Mingione, The existence of regular boundary points for non-linear elliptic systems. \emph{J. Reine Angew. Math.} 602, 17–58, (2007).
\bibitem{dumi} F. Duzaar, G. Mingione, The $p$-harmonic approximation and the regularity of $p$-harmonic maps. \emph{Calc. Var. Partial Differential Equations} 20, no. 3, 235–256, (2004).
\bibitem {ES} J. Eells, J. H. Sampson, Harmonic mappings of Riemannian manifolds. \emph{Amer. J. Math.} 86, 109-160, (1964). 
\bibitem {sharp} L. Esposito, F. Leonetti, G. Mingione, Sharp regularity for functionals with $(p,q)$ growth. 
\emph{J. Differential Equations} 204, 5-55, (2004).
\bibitem{evga} L. C. Evans, R. F. Gariepy, Partial regularity for constrained minimizers of convex or quasiconvex functionals. \emph{Rend. Sem. Mat. Univ. Pol. Torino}, vol. 47, (1989).
\bibitem{fomami} I. Fonseca, J. Mal\'y, G. Mingione, Scalar minimizers with fractal singular sets. \emph{Arch. Ration. Mech. Anal.} 172, 295-307, (2004).
\bibitem{fu1} M. Fuchs, $p$-Harmonic Obstacle Problems, I. \emph{Annali di Matematica} 156 (4), 127-158, (1990).
\bibitem{fu2} M. Fuchs, $p$-Harmonic Obstacle Problems, II. \emph{Manuscripta Math.} 63, 381-419, (1989).
\bibitem{fu3} M. Fuchs, $p$-Harmonic Obstacle Problems, III. \emph{Annali di Matematica} 156 (4), 159-180, (1990).
\bibitem{giamar} M. Giaquinta, L. Martinazzi, An Introduction to the Regularity Theory for Elliptic Systems, Harmonic Maps and Minimal Graphs. \emph{Edizioni della Normale}, (2012).
\bibitem{giamodsou} M. Giaquinta, G. Modica, J. Sou\u cek, Cartesian Currents in the Calculus of Variations II. \emph{[Results in Mathematics and Related Areas. 3rd Series. A Series of Modern Surveys in Mathematics]} vol. 38, Springer-Verlag, Berlin, 1998, Variational integrals. MR 1645082 (2000b:49001b)
\bibitem{giu} E. Giusti, Direct Methods in the calculus of variations. \emph{World Scientific Publishing Co., Inc., River Edge}, (2003).
\bibitem{gr} J. F. Grotowski, Boundary regularity for nonlinear elliptic systems. \emph{Calc. Var. \& PDE} 15, 353–388, (2002).
\bibitem{ham} C. Hamburger, Partial boundary regularity of solutions of nonlinear superelliptic systems. \emph{Bollettino UMI} 1, Ser. 8, 63-81, (2007). 
\bibitem{ha} C. Hamburger, Regularity of differential forms minimizing degenerate elliptic functionals. \emph{J. Reine Angew. Math., (Crelles J.)}, 431, 7-64, (1992).
\bibitem{harkinlin} R. Hardt, D. Kinderlehrer, F.-H. Lin, Stable defects of minimizers of constrained variational principles. \emph{Ann. Inst. H. Poincar\'e Anal. Non Lin\'eaire} 5, 297-322, (1988).
\bibitem{harlin} R. Hardt, F.-H. Lin, Mappings Minimizing the $L^{p}$ Norm of the Gradient. \emph{Comm. Pure Appl. Math.} 40, 555-588, (1987).
\bibitem{o1} P. Harjulehto, P. H\"ast\"o, U. V. Le, M. Nuortio, Overview of differential equations with non-standard growth, \emph{ Nonlinear Anal. } 72 (12), 4551-4574, (2010).
\bibitem{harhas} P. Harjulehto, P. H\"ast\"o, Double Phase image restoration. \emph{Preprint}, (2019). \url{https://arxiv.org/pdf/1906.09837.pdf}
\bibitem{haok} P. H\"ast\"o, J. Ok, Maximal regularity for local minimizers of non-autonomous functionals. \emph{Preprint}, (2019). \url{https://arxiv.org/pdf/1902.00261.pdf}
\bibitem{hat} A. Hatcher, Algebraic Topology. \emph{Cambridge University Press}, Cambridge, (2002).
\bibitem{hisc} J. Hirsch, M. Sch\"affner, Growth conditions and regularity, an optimal local boundedness result. \emph{Preprint}, (2019). \url{https://arxiv.org/pdf/1911.12822.pdf}
\bibitem{ho} C. Hopper, Partial Regularity for Holonomic Minimizers of Quasiconvex Functionals. \emph{Arch. Rational Mech. Anal.} 222, 91-141, (2016).
\bibitem{jm} J. Jost, M, Meier, Boundary regularity for Minima of Certain Quadratic Functionals. \emph{Math. Ann.} 262, 549-561, (1983).
\bibitem{km} J. Kristensen, G. Mingione, Boundary Regularity in Variational Problems. \emph{Arch. Rational Mech. Anal.} 198, 369-455, (2010). 
\bibitem{kumi1} T. Kuusi, G. Mingione, Vectorial nonlinear potential theory. \emph{J. Europ. Math. Soc. (JEMS)} 20, 929-1004, (2018).
\bibitem{leosie} F. Leonetti, F. Siepe, Maximum Principle for Vector Valued Minimizers. \emph{J. Convex Anal.} 12, 267-278, (2005).
\bibitem{luc} S. Luckhaus, Partial H\"older continuity for minima of certain energies among maps into a Riemannian manifold. \emph{Indiana Univ. Math. J.} 37, 349-367, (1988).
\bibitem{man} J. J. Manfredi, Regularity of the gradient for a class of nonlinear possibly
degenerate elliptic equations. \emph{Ph.D. Thesis, University of
Washington, St. Louis}, (1986).
\bibitem{ma1} P. Marcellini, Regularity of minimizers of integrals of the calculus of variations with non standard growth conditions. \emph{Arch. Rat. Mech. Anal.} 105, 267-284, (1989).
\bibitem{ma2} P. Marcellini, On the definition and the lower semicontinuity of certain quasiconvex integrals. \emph{Annales de l'I.H.P. Analyse non lin\'eaire} 3, nr. 5, 391-409, (1986).
\bibitem {MMS2} M. Mazowiecka, M. Mi\'skiewicz, A. Schikorra, On the size of the singular set of minimizing harmonic maps into the 2-sphere in dimension four and higher. \emph{Preprint}, (2019), \texttt{https://arxiv.org/pdf/1902.03161.pdf}
\bibitem{nava} A. Naber, D. Valtorta, Rectifiable-Reifenberg and the regularity of stationary and minimizing harmonic maps. \emph{Ann. of Math.} (2) 185, no. 1, 131-227, (2017).
\bibitem{rata} M. A. Ragusa, A. Tachikawa, Boundary regularity of minimizers of $p(x)$-energy functionals. \emph{Ann. I. H. Poincar\'e - AN} 33, 451-476, (2016).
\bibitem{ratata} M. A. Ragusa, A. Tachikawa, H. Takabayashi, Partial regularity of $p(x)$-harmonic maps. \emph{Transaction of the AMS} 365, nr. 6, 3329-3353, (2013).
\bibitem{shouhl} R. Schoen, K. Uhlenbeck, A Regularity Theory for Harmonic Maps. \emph{J. Differential Geometry} 17, 307-335, (1982).
\bibitem{shouhlb} R. Schoen, K. Uhlenbeck, Boundary regularity and the Dirichlet problem for harmonic maps. \emph{J. Differential Geometry} 18, 253-268, (1983).
\bibitem{sim} L. Simon, Theorems on Regularity and Singularity of Energy Minimizing Maps. \emph{Lectures in Mathematics}, ETH Z\"urich, Birkh\"auser, Basel, (1996).
\bibitem{tac} A. Tachikawa, On the singular set of $p(x)$-energy. \emph{Calc. Var. \& PDE} 50, 145-169, (2014).
\bibitem{uhl} K. Uhlenbeck, Regularity for a class of non-linear elliptic sysyem. \emph{Acta Math.} 138, 219-240, (1977).
\bibitem{zhi1} V. V. Zhikov, On some variational problems. \emph{Russian J. Math. Phys.} 5, 105-116, (1997).
\bibitem{zhi2} V. V. Zhikov, Averaging of functionals of the calculus of variations and
elasticity theory, \emph{Izv. Akad. Nauk SSSR Ser. Mat.} 50,
675-710, (1986).
\bibitem{zhi3} V. V. Zhikov, On Lavrentiev's Phenomenon, \emph{Russian J. Math. Phys.} 3, 249-269, (1995).
\bibitem{zhi4} V. V. Zhikov, S. M. Kozlov, O. A.Oleinik: Homogenization of
differential operators and integral functionals. \emph{Springer-Verlag,
Berlin}, (1994).


\end{thebibliography}
\end{document}